\documentclass[smallextended]{svjour3}
\smartqed
\usepackage{amssymb}
\usepackage{graphicx}
\usepackage[hidelinks]{hyperref}
\usepackage{float}
\usepackage{mathtools}
\usepackage{mathrsfs}
\usepackage{afterpage}
\usepackage{amsmath,amssymb}
\usepackage[export]{adjustbox}
\usepackage{caption}
\usepackage{subcaption}
\usepackage{dsfont}
\usepackage{geometry}
\usepackage{xcolor}
\usepackage[toc,page, title, titletoc]{appendix}
\usepackage{enumitem}

\usepackage{color}
\usepackage[english]{babel}
\usepackage{mathtools}
\DeclarePairedDelimiter\abs{\lvert}{\rvert}

\DeclareMathOperator*{\argmin}{arg\,min}

\makeatletter

\renewcommand\section{\@startsection {section}{1}{\z@}%
                               {-3.5ex \@plus -1ex \@minus -.2ex}%
                               {2.3ex \@plus.2ex}%
                               {\normalfont\large\bfseries}}
\renewcommand\subsection{\@startsection{subsection}{2}{\z@}%
                                 {-3.25ex\@plus -1ex \@minus -.2ex}%
                                 {1.5ex \@plus .2ex}%
                                 {\normalfont\bfseries}}

\makeatother

\usepackage[english]{babel} 
\usepackage{amsmath}
\usepackage{amsfonts}
\usepackage{bold-extra}

\usepackage{fancyhdr}        
\numberwithin{equation}{section}       
\numberwithin{figure}{section}         
\numberwithin{table}{section}          

\newtheorem{prop}[theorem]{Proposition}

\newtheorem{assumption}{Assumption}

\graphicspath{{img/}}

\newcommand{\vertiii}[1]{{\left\vert\kern-0.25ex\left\vert\kern-0.25ex\left\vert #1\right\vert\kern-0.25ex\right\vert\kern-0.25ex\right\vert}}

\begin{document}

\title{On the Generalised Langevin Equation for Simulated Annealing}
\author{ Martin Chak \and Nikolas Kantas \and Grigorios A. Pavliotis}

\institute{Department of Mathematics, Imperial College London, SW72AZ \\
              \email{martin.chak14@imperial.ac.uk \and n.kantas@imperial.ac.uk \and g.pavliotis@imperial.ac.uk}   
}

\date{Received: date / Accepted: date}

\maketitle

\begin{abstract}
In this paper, we consider the generalised (higher order) Langevin equation for the purpose 
of simulated annealing and optimisation of nonconvex functions. Our approach modifies the underdamped Langevin equation 
by replacing the Brownian noise with an appropriate Ornstein-Uhlenbeck process to account for memory in the system. 
Under reasonable conditions on the loss function and the annealing schedule, we establish 
convergence of the continuous time dynamics to a global minimum. In addition, we investigate the performance numerically and show better performance and
higher exploration of the state space compared to the underdamped and overdamped Langevin dynamics with the same annealing schedule.
\end{abstract}

\section{Introduction}
Algorithms for optimisation have received significant interest in recent years due to applications in machine learning, data science and molecular dynamics. Models in machine learning are formulated to have some loss function and parameters with respect to which it is to be minimised, where use of optimisation techniques is heavily relied upon. We refer to \cite{rev,rev1} for related discussions. Many models, for instance neural networks, use parameters that vary over a continuous space, where gradient-based optimisation methods can be used to find good parameters that generate effective predictive ability. As such, the design and analysis of such algorithms for global optimisation has been the subject of considerable research \cite{Rud} and it has proved useful to study algorithms for global optimisation using tools from the theory of stochastic processes and dynamical systems. A paradigm of the use of stochastic dynamics for the design of algorithms for global optimisation is one of simulated annealing, where overdamped Langevin dynamics with a time dependent temperature \eqref{Sim} that decreases with an appropriate cooling schedule is used to guarantee the global minimum of a nonconvex loss function $U:\mathbb{R}^n\rightarrow \mathbb{R}$:
\begin{equation}\label{Sim}
dX_t = -\nabla U(X_t)\,dt+\sqrt{2T_t}\,dW_t.
\end{equation}
Here $W_t$ is a standard $n$-dimensional Wiener process and $T_t:\mathbb{R_+}\rightarrow \mathbb{R}$ is an appropriate determinstic function of time 
often referred to as the annealing or cooling schedule. For fixed $T_t=T>0$, this is the dynamics used for the related problem of sampling from a possibly high dimensional probability measure, for example in the unadjusted Langevin algorithm \cite{dur}. Gradually decreasing $T_t$ to zero
balances the exploration-exploitation trade-off by allowing at early times larger noise to drive $X_t$ and hence sufficient mixing to escape local minima. 
Designing an appropriate annealing schedule is well-understood. We briefly mention classical references \cite{Chiang,Gem,GeH,Gid2,Gid,Hol2,Hol,Kush}, as well as the more recent \cite{larg2,larg1,weak}, where one can find
details and convergence results. In this paper we aim to consider generalised versions of \eqref{Sim} for the same purpose.

Using dynamics such as \eqref{Sim} has clear connections with sampling. When $T_t=T$ is a constant function, the invariant distribution of $X$ is
proportional to $\exp(-\frac{U(x)}{T})dx$. In addition, when $T_t$ decreases with time, the probability measure $\nu_t(dx)\propto\exp(-\frac{U(x)}{T_t})dx$ 
converges weakly to the set of global minima based on the Laplace principle \cite{hwang1980laplace}. 
One can expect that if one replaces \eqref{Sim} with a
stochastic process that mixes faster and maintains the same invariant distribution for constant temperatures, 
then the superior speed of convergence should improve performance in optimisation due to the increased exploration of the state space. 
Indeed, it is well known that many different dynamics can be used in order to sample from a given probability distribution, or for finding the minima of a function when the dynamics is combined with an appropriate cooling schedule for the temperature. Different kinds of dynamics have already been considered for sampling, e.g. nonreversible dynamics, preconditioned unadjusted Langevin dynamics \cite{prec2,mass,prec,pat}, as well as for optimisation, e.g. interacting Langevin dynamics \cite{Sun}, consensus based optimisation \cite{cbo3,tse,tse2}, to name a few.

A natural candidate in this direction is to use the underdamped or kinetic Langevin dynamics:
\begin{subequations}\label{Sim2}
\begin{align}
dX_t &= Y_t\, dt\label{Sim21}\\
dY_t &= -\nabla U(X_t)\,dt-{T_t^{-1}} {\mu} Y_t\, dt+\sqrt{2\mu}\, dW_t\label{Sim22}
\end{align}
\end{subequations}
Here the reversibility property of \eqref{Sim} has been lost; the improvement from breaking reversibility in both the context of sampling and that of optimisation
 is investigated in \cite{nonrev2,nonrev1} and \cite{Gao} respectively.
When $T_t=T$, \eqref{Sim2} can converge faster than \eqref{Sim} to its invariant distribution
\begin{equation*}
\rho(dx,dy)\propto\exp\bigg(-\frac{1}{T}\bigg(U(x)+\frac{\abs*{y}^2}{2}\bigg)\bigg)dx\:dy,
\end{equation*}
see \cite{eberle2019couplings} or Section 6.3 of \cite{pav} for particular comparisons and also \cite{chem2,fire} for more applications using variants of \eqref{Sim2}.
In the context of simulated annealing, using this set of dynamics has recently been studied rigorously in \cite{meta},
where the author established convergence to global minima using the generalised $\Gamma$-calculus \cite{gen} framework that is based on Bakry-Emery theory. Note that 
\eqref{Sim22} uses the temperature in the drift rather than the diffusion constant in the noise as in \eqref{Sim}. Both formulations admit the same invariant measure when $T_t=T$. 
In the remainder of the paper, we adopt this formulation to be closer to \cite{meta}. 

In this paper we will consider an extension of the kinetic Langevin equation by adding an additional auxiliary variable that accounts for the memory in the system. To the best of the authors' knowledge, this has not been attempted before in the context of simulated annealing and global optimisation. In particular we consider the Markovian approximation to the generalised Langevin equation:
\begin{subequations}\label{gl}
\begin{align}
dX_t &= Y_t \, dt\label{gl1}\\
dY_t &= -\nabla U(X_t) \, dt + \lambda^\top Z_t \, dt\label{gl2}\\
dZ_t &= -\lambda Y_t \, dt -T_t^{-1}A Z_t \, dt +\Sigma \, dW_t\label{gl3},
\end{align}
\end{subequations}
where $A\in\mathbb{R}^{m\times m}$ is symmetric positive definite matrix\textcolor{black}{, meaning that there exists a constant $A_c >0$ such that 
\begin{equation*}
z^\top Az \geq A_c \abs{z}^2
\end{equation*}
for all $z\in\mathbb{R}^m$}, $\Sigma\in\mathbb{R}^{m\times m}$ satisfies
\begin{equation*}
\Sigma\Sigma^\top = 2A
\end{equation*}
and $W_t$ is now $m$-dimensional. Here $X_t, Y_t\in\mathbb{R}^n$ and $Z_t\in\mathbb{R}^m$ (with $m\geq n$), $M^\top$ denotes the transpose of a matrix $M$, $\lambda\in\mathbb{R}^{m\times n}$ is a 
rank $n$ matrix with a left inverse $\lambda^{-1}\in\mathbb{R}^{n\times m}$. 

Our aim is to establish convergence using similar techniques as \cite{meta} and investigate the improvements in 
performance. 
Equation \eqref{gl} is related to the generalised Langevin equation, where memory is added to \eqref{Sim2} by integrating over past velocities with a kernel $\Gamma:\mathbb{R}_+\rightarrow \mathbb{R}^{n\times n}$:
\begin{equation}\label{gen}
\ddot{x}=-\nabla U(x)-\int_0^t \Gamma(t-s)\dot{x}(s)\, ds+F_t
\end{equation}
with $F_t$ being a zero mean stationary Gaussian process with an autocorrelation matrix given by the fluctuation-dissipation theorem 
\begin{equation*}
\mathbb{E}(F_t F_s^\top)=T_t\Gamma(t-s).
\end{equation*}
When $T_t=T$,  
\eqref{gen} is equivalent to \eqref{gl} when setting 
\begin{equation}
\Gamma(t)=\lambda^\top e^{-\!At}\lambda,
\end{equation}
and the invariant distribution becomes
\begin{equation*}
\rho(dx,dy, dz)\propto\exp\bigg(-\frac{1}{T}\bigg(U(x)+\frac{\abs*{y}^2}{2}+\frac{\abs*{z}^2}{2}\bigg)\bigg)dx\:dy\:dz,
\end{equation*}
see Section 8.2 of \cite{pav} for details\footnote{To the best of the authors' knowledge, there is no known direct translation between \eqref{gen} and \eqref{gl} for a non-constant $T_t$; such a translation quite possibly exists and at the very least the intuition here is useful.}.  
In the spirit of adding a momentum variable in \eqref{Sim} to get \eqref{Sim2}, \eqref{gl} 
adds an additional auxiliary variable to the Langevin system whilst preserving the invariant distribution in the $x$ marginal. 
In the constant temperature context, \eqref{gen} is natural from the point of statistical mechanics and has already been considered as a sampling tool in \cite{ceri1,ceri3,ceri2,ceri4} with considerable success.
We will demonstrate numerically that the additional tuning parameters can improve performance; see also \cite{glesamp} for recent work demonstrating advantages of using \eqref{gen} compared to using \eqref{Sim2} when sampling from a log concave density. A detailed study of the Markovian approximation \eqref{gl} of the generalised Langevin dynamics in \eqref{gen} can be found in \cite{ottobre2011asymptotic}.

To motivate the use of \eqref{gl}, consider the quadratic case where $U = \alpha x^2$ and $0< \alpha<1$. This case
allows for explicit or numerical calculation of the spectral gaps
of the generators in \eqref{Sim}-\eqref{gl} in order to compare the rate of convergence to equilibrium; see \cite{spec,Otto} for details. 
For a given $T$, it is possible to choose $\lambda$ and $A$, 
such that the spectral gap of the generator of \eqref{gl} is much larger than that
of \eqref{Sim2} with the best possible choice of $\mu$ being used. The latter is already larger than that of the overdamped dynamics in \eqref{Sim}. 
We will later demonstrate numerically that this will translate to 
better exploration in simulated annealing (when $T_t$ is decreasing in time).

Use of \eqref{gen} is also motivated by parallels with accelerated gradient descent algorithms. When the noise is removed from \eqref{Sim2},
the second order differential equation can be loosely considered as a continuous time version of Nesterov's algorithm \cite{nips}. 
The latter is commonly preferred to discretising the first order differential equation given by the noiseless version of \eqref{Sim}, because
in the high dimensional and low iterations setting it achieves the optimal rate of convergence for convex optimisation; 
see Chapter 2 in \cite{nest} and also \cite{Ghad} for a nonconvex setting.
Here we would like to investigate the effect of adding another auxiliary variable, 
which would correspond to a third order differential equation when noise is removed.
When noise is added for the fixed temperature case, \cite{gad} has studied the long time behaviour and 
stability for different choices of a memory kernel as in \eqref{gen}. 
Finally, we note that generalised Langevin dynamics in \eqref{gen} have additionally been studied in related 
areas such as sampling problems in molecular dynamics from chemical modelling \cite{chem,ceri1,ceri3,ceri2,ceri4,chem3}, 
see also \cite{datadriv} for work determining the kernel $\Gamma$ in the generalised system \eqref{gen} from data.


Our theoretical results will focus only on the continuous time dynamics and follow the approach in \cite{meta}. 
The main requirement in terms of assumptions
are quadratic upper and lower bounds on $U$ and bounded second derivatives. This is different to 
classical references such as \cite{GeH}, \cite{Gid2} or \cite{Hol}. These works also rely on the Poincar\'e inequality, an approach which will be mirrored here 
(and in \cite{meta} for the underdamped case) using a log-Sobolev inequality; see also \cite{Hol2} for the relationship between such functional inequalities and the annealing schedule in the finite state space case.
We will also present detailed numerical results for different 
choices of $U$. There are many possibilities for the method of discretisation of \eqref{gl}, we will use a time discretisation scheme that appeared in \cite{BB}, but will not
present theoretical results on the time discretised dynamics; this is beyond the scope of this article.
We refer instead the interested reader to \cite{Lstud} for a study on discretisation schemes for the system \eqref{gl},
\cite{cheng} for a recent consideration on \eqref{Sim2} and its time-discretisation and
\cite{Gelf,Gelf2} for linking discrete time Markov chains with the overdamped Langevin system in \eqref{Sim}.

\subsection{Contributions and organisation of the paper}
Here we summarise the main contributions of the paper.
\begin{itemize}
\item
We provide a complete theoretical analysis of the simulated annealing algorithm for the generalised Langevin equation \eqref{gl}. The main theoretical contribution consists of Theorem \ref{convergence} that
establishes convergence in probability of $X_t$ in the higher order Markovian dynamics \eqref{gl} to a global minimiser of $U$. For the optimal cooling schedule $T_t$, the rate of convergence is as the known rate for the Langevin system \eqref{Sim2} presented in \cite{meta}.
\item
The initially non-Markovian property and pronounced degeneracy in the sense of requiring a second commutator bracket for hypoellipticity by way of H\"ormander introduces additional difficulties that are overcome using techniques from \cite{meta}. As such, we use a different form of the distorted entropy, stated formally in \eqref{entdef}.
Additional technical improvements include a different truncation argument and a limiting sequence of nondegenerate SDEs for establishing dissipation of this distorted entropy. 
These extensions also address certain technical issues in \cite{meta}; see Remarks \ref{rem_assumptions}, \ref{tanirem} and \ref{trunc_issue} for more details. 
Also we make an effort to emphasise the role of the critical factor of the cooling schedule in the rate of convergence in Theorem \ref{convergence}. 
This can be seen in our assumptions for $T_t$ and $U$ below. 
\item
As a byproduct, we prove exponential convergence to equilibrium for the constant temperature $T_t=T>0$ case in Proposition \ref{constT0}, which is relevant for sampling problems. See \cite{ottobre2011asymptotic} and \cite{Vaes} for similar results.
\item
Detailed numerical experiments are provided to illustrate the performance of our approach. We also discuss 
thoroughly tuning issues. In particular, we investigate numerically the role of matrix $A$ and how it can be chosen to increase exploration of the state space. 
As regards to time discretisation of \eqref{gl} we use the leapfrog scheme of \cite{BB}. We compare this with a similar time discretisation of \eqref{Sim2} and observe that exploration of the state space is increased considerably.
\end{itemize}
The paper is organised as follows. Section \ref{main} will present the assumptions and main theoretical results. Detailed proofs can be found in Section \ref{proofs}.
Section \ref{num} presents numerical results demonstrating the effectiveness of our approach both in terms of reaching the 
global minimum and the exploration of the state space. In Section \ref{conclusion}, we provide some concluding remarks.

\section{Main Result}\label{main}

Let $L_t$ denote the infinitesimal generator of the associated semigroup to \eqref{gl} at $t>0$ and temperature $T_t$. This is formally given by
\begin{equation}\label{glgen}
L_{t} = (y\cdot\nabla\!_x-\nabla\!_x U(x)\cdot\nabla\!_y)+(z^\top\lambda\nabla\!_y-y^\top\lambda^\top\nabla\!_z)-T_t^{-1}z^\top  A \nabla\!_z +A:D_z^2,
\end{equation}
where we denote the gradient vector as
$\nabla_x=(\partial_{x_1},\ldots,\partial_{\!x_n})^\top$, the Hessian with $D_x^2$ and similarly for the $y$ and $z$ variables. 
For matrices $M,N\in\mathbb{R}^{r\times r}$ we denote $M:N=\sum_{i,j}M_{ij}N_{ij}$ for all $1\leq i,j\leq r$ 
and the operator norm as $|{M}| =\sup \left\{{\frac {|Mv|}{|v|}}:v\in \mathbb{R}^{r}{\text{ with }}x\neq 0\right\}$.
We will also use $|v|$ to denote Euclidean 
distance for a vector $v$.
Let $m_t$ be the law of $(X_t,Y_t,Z_t)$ in \eqref{gl} and, with slight abuse of notation, we will also denote as $m_t$ the corresponding Lebesgue density. Similarly we define
$\mu_{T_t}$ be the instantaneous invariant law of the process
\begin{equation}\label{equilibrium}
\mu_{T_t}(dx,dy,dz)\quad =\quad 
\frac{1} {Z_{T_t}}{\exp\bigg(-\frac{1}{T_t}\bigg(U(x)+\frac{\abs*{y}^2}{2}+\frac{\abs*{z}^2}{2}\bigg)\bigg)}dx\:dy\:dz 
\end{equation}
with $Z_{T_t}=\int \exp{\big(-\frac{1}{T_t}\big(U(x)+\frac{\abs*{y}^2}{2}+\frac{\abs*{z}^2}{2}\big)\big)}dx\:dy\:dz$. 
Finally, denote the density between the two laws:
\begin{equation}\label{ht}
h_t\quad=\quad\frac{dm_t}{d\mu_{T_t}}.
\end{equation}

We proceed by stating our assumptions on the potential $U$.

\begin{assumption}\label{assumption1}
\textcolor{white}{a}\\[0.5em]
There exists $0<\theta<1$ such that $U$ is smooth (belongs in $\mathcal{C}^\infty(\mathbb{R}^n)$) with bounded second derivatives 
\begin{align}
|{D_x^2 U}|_{\infty} := \sup_{x\in\mathbb{R}^n} \max\bigg(\sup_{ij}\abs*{\partial_i\partial_j U(x)},|{D_x^2 U(x)}|\bigg)< \infty, \label{secbdd}
\end{align}
satisfies 
\begin{align}
\nabla\!_x U(x)\cdot x &\geq r_1\abs{x}^2-U_{\!g}\label{q2}\\
\abs{\nabla\!_x U(x)}^2 &\leq r_2\abs{x}^2+U_{\!g}\label{q3}
\end{align}
for some constants $r_1, r_2\in\mathbb{R}$, $U_{\!g}>0$ and either 
\begin{enumerate}[label=(\alph*)]
\item\begin{equation}
\abs*{\bar{a}\circ x}^2+U_{\!m}\quad\leq \quad U(x)\quad\leq\quad \abs*{\bar{a}\circ x}^2+U_{\!M}\label{q1}
\end{equation}
for some $U_{\!m},U_{\!M}\in\mathbb{R}$, $\bar{a}\in\mathbb{R}_+^n$, where $\circ$ denotes the Hadamard product, or
\item
\begin{itemize}
\item $U$ is a nonnegative Morse function, in the sense that there exists $1\leq C_H<\infty$ such that if $x\in\mathbb{R}^n$ satisfies $\nabla\!_x U(x) = 0$, then
\begin{equation*}
\frac{\abs*{\xi}^2}{C_H}\leq \langle \xi,D_x^2 U(x) \xi\rangle\leq C_H \abs*{\xi}^2\qquad \forall \xi\in\mathbb{R}^n,
\end{equation*}
\item $U$ is nondegenerate in the sense that:
\begin{itemize}
\item For any two local minima $m_i,m_j\in\mathbb{R}^n$, there exists a unique (communicating saddle) point $s_{i,j}\in\mathbb{R}^n$ such that 
\begin{itemize}
\item $\nabla\!_x U(s_{i,j}) = 0$,
\item $U(s_{i,j}) = \inf\{\max_{s\in [0,1]}U(\gamma(s)):\gamma\in C([0,1],\mathbb{R}^n),\gamma(0) = m_i, \gamma(1) = m_j\}$,
\item the dimension of the unstable subspace of $D_x^2 U(s_{i,j})$ is equal to $1$.
\end{itemize}
\item Setting $m_1$ to be the global minimum of $U$, there exists $\delta >0$ and an ordering of the local minima $\{m_2,m_3,\dots\}$ such that $U(s_{1,2}) - U(m_2) \geq U(s_{1,i}) - U(m_i) + \delta$ for all $i\geq 3$.
\end{itemize}
\end{itemize}
\end{enumerate}
\end{assumption}
Note that \eqref{q2} and \eqref{q3} imply
\begin{equation}\label{q1menz}
a_m\abs*{x}^2+U_{\!m}\quad\leq \quad U(x)\quad\leq\quad a_M\abs*{x}^2+U_{\!M}
\end{equation}
for some $a_m, a_M >0$, $U_{\!m},U_{\!M} \in\mathbb{R}$.
In the rest of the paper, if \eqref{q1} holds then the smallest and largest element of $\bar{a}$ is denoted with
\begin{equation*}
a_m = \min_i \bar{a}_i,\quad a_M = \max_i \bar{a}_i,
\end{equation*}
where $\bar{a} = (\bar{a}_1,\dots,\bar{a}_n)$.\\
\begin{assumption}\label{assumption2}
The temperature $T_t$ satisfies
\begin{equation*}
\lim_{t\rightarrow\infty}T_t=0.
\end{equation*}
\end{assumption}
Before we proceed with further assumptions on the annealing schedule $T_t$ and on the initial distribution, note that under Assumption \ref{assumption1} and \ref{assumption2}, a log-Sobolev inequaity holds.
\begin{prop}\label{logsobprop0}
Under Assumption \ref{assumption1} and Assumption \ref{assumption2}, there exists constants $t_{ls}^{(0)},\hat{E}, A_*^{(0)}>0$ and a finite order polynomial $r^{(0)}:\mathbb{R}_+\rightarrow\mathbb{R}_+$ with coefficients depending on $U$ such that for all $0<h\in C^\infty(\mathbb{R}^{2n+m})$,
\begin{equation}\label{logSob0}
\int h \ln h d\mu_{T_t} \leq C_t^{(0)} \int\frac{\abs*{\nabla h}^2}{h}d\mu_{T_t},
\end{equation}
where for $t>t_{ls}^{(0)}$,
\begin{equation}\label{logSobConst0}
C_t^{(0)} = r^{(0)}\Big(T_t^{-\frac{1}{2}}\Big) e^{\hat{E}T_t^{-1}}.
\end{equation}
\end{prop}
The proof is deferred and the constant $\hat{E}$ from the above proposition will be used in stating the following assumption about $T_t$, as well as what follows. In the case of \eqref{q1}, $\hat{E}$ can be taken as $U_M - U_m$, otherwise it is the critical depth \cite{Sch} of $U$.
\begin{assumption}\label{assumption3}
\textcolor{white}{a}\\[0.5em]
The temperature $T_\cdot:[0,\infty) \rightarrow (0,\infty)$ is continuously differentiable, bounded above and there exists some constant \textcolor{black}{$t_0>1$} such that $T_t$ satisfies for all $t>t_0$: 
\begin{enumerate}[label=(\roman*)]
\item $T_t\geq E(\ln t)^{-1}$ for some constant $E>\hat{E}\geq 0$, where $\hat{E}$ is the constant in Proposition \ref{logsobprop0},
\item $\abs*{T_t'} \leq \widetilde{T}t^{-1}$  for some constant $\widetilde{T}>0$.
\end{enumerate}
\end{assumption}
\begin{assumption}\label{assumption4}
\textcolor{white}{a}\\[0.5em]
The initial law $m_0$ admits a bounded density with respect to the Lebesgue measure on $\mathbb{R}^{2n+m}$, also denoted $m_0$, satisfying:
\begin{enumerate}[label=(\roman*)]
\item $ {m_0}\in\mathcal{C}^\infty(\mathbb{R}^{2n+m})$, 
\item $\int\frac{\abs*{\nabla m_0}^2}{m_0}dxdydz < \infty$,
\item \textcolor{black}{$\int (\abs*{x}^2+\abs*{y}^2+\abs*{z}^2) m_0\: dxdydz < \infty$,}
\end{enumerate}
\end{assumption}
\vspace{1em}
\begin{remark}\label{rem_assumptions}
Note that \eqref{q3} and \eqref{q1} deviate from \cite{meta}. \eqref{q1} is useful for a more self-contained exposition for the log-Sobolev constant in \eqref{logSobConst}, but the alternative that $U$ is a nonnegative nondegenerate Morse function is optimal in the sense that $\hat{E}$ is in this case given as the critical depth of $U$. Condition \eqref{q1} is satisfied for instance by a multivariate Gaussian after a rotation of the $x$ coordinates.
\end{remark}

We present two key propositions.
\begin{prop}\label{refer1}
Under Assumptions \ref{assumption1} and \ref{assumption3}, for all $t>0$, denote by $\big(X^{T_t},Y^{T_t},Z^{T_t}\big)$ a r.v. with distribution $\mu_{T_t}$. For any $\delta,\alpha>0$, there exists a constant $\hat{A}>0$ such that 
\begin{equation*}
\mathbb{P}\big(U\big(X^{T_t}\big)>\min U+\delta\big) \leq \hat{A}e^{-\frac{\delta-\alpha}{T_t}}
\end{equation*}
holds for all $t>0$.
\end{prop}
\begin{proof}
The result follows exactly as in Lemma 3 in \cite{meta}.
\end{proof}
\begin{prop}\label{refer2}
Under Assumptions \ref{assumption1}, \ref{assumption3} and \ref{assumption4}, for all $t>0$, $(X_t,Y_t,Z_t)$ are well defined as the unique strong solution to \eqref{gl}, $\mathbb{E}\big[\abs*{X_t}^2+\abs*{Y_t}^2+\abs*{Z_t}^2\big]<\infty$ and the law $m_t$ admits an everywhere positive density with respect to the Lebesgue measure on $\mathbb{R}^{2n+m}$.
\end{prop}
For the proof of Proposition \ref{refer2}, see Proposition \ref{refer2app} in Section \ref{proofs}.

Proposition \ref{refer1} can be thought of as a Laplace principle; Proposition \ref{refer2} asserts that the process
\eqref{gl} does not blow up in finite time and the noise in the dynamics \eqref{gl3} for $Z_t$ spreads throughout the system, that is to $X_t$ and $Y_t$.

\begin{prop}\label{dissipation0}
Under Assumption \ref{assumption1}, \ref{assumption3} and \ref{assumption4}, for any $\alpha>0$, there exists some constant $B>0$ and $t_h>0$, 
such that for all $t>t_h$,
\begin{equation}
\int h_t\ln h_td\mu_{T_t}\leq B\bigg( \frac{1}{t}\bigg)^{1-\frac{\hat{E}}{E}-2\alpha}.
\end{equation}
\end{prop}
The full proof is the contained in Section \ref{proofs} and follows from Proposition \ref{dissipation}.
Therein a similar statement is proved for the \emph{distorted entropy} that has the following form:
\begin{equation*}
H(t)\ :=\ \int\bigg( \frac{\langle S\nabla h_t ,\nabla h_t\rangle}{h_t}+ \beta(T_t^{-1})h_t \ln(h_t)\bigg)d\mu_{T_t},
\end{equation*}
where $S$ being a well chosen matrix (so that \eqref{gamineq} holds) and $\beta(\cdot)$ is a polynomial 
(see \eqref{entdef} for the precise form of $H(t)$ and \eqref{betdef} for $\beta(\cdot)$). 
This construction of $H$ compared to a standard definition of entropy compensates for the fact that the diffusion is degenerate (see \cite{Hypo} for a general discussion).
The proof uses an approximating sequence of SDE's, in which all of the elements have nondegenerate noise. The problem is split into the partial 
time and partial temperature derivatives where, amongst other tools, \eqref{gamineq} and a log-Sobolev inequality are used as in \cite{meta} to arrive at a bound that allows a Gr{\"o}nwall-type argument.

\begin{remark}
Proposition \ref{dissipation} is a statement about the distorted entropy $H(t)$, which bounds the entropy $\int  h_t\ln h_td\mu_{T_t}$.
In fact this is achieved in such a way that the bound becomes less sharp as $t$ becomes large but without consequences for our main Theorem \ref{convergence} below. 
\end{remark}

We proceed with the statement of our main result, using $t_h$ from Proposition \ref{dissipation0}.

\begin{theorem}\label{convergence}
Under Assumptions \ref{assumption1}, \ref{assumption2}, \ref{assumption3} and \ref{assumption4}, for any $\delta>0$, as $t\rightarrow\infty$,
\begin{equation*}
\mathbb{P}(U(X_t)\leq \min U +\delta)\rightarrow1.
\end{equation*}
If in addition $T_t=E(\ln t)^{-1}$, then for any $\alpha>0$, there exists a constant $C>0$ such that for all $t>t_h$,
\begin{equation*}
\mathbb{P}(U(X_t)\leq \min U +\delta)\leq C \bigg(\frac{1}{t}\bigg)^{r(E)},
\end{equation*}
{where the exponential rate $r^e:(\hat{E},\infty)\rightarrow\mathbb{R}$ is defined by}

\begin{align*}
r^e(E)&:=\min\bigg(\frac{1-\frac{\hat{E}}{E}-2\alpha}{2},\frac{\delta-\alpha}{E}\bigg)\\
&=
\begin{cases}
\frac{1}{2}\Big(1-\frac{\hat{E}}{E}-2\alpha\Big) &\qquad\textrm{if } E < \frac{\hat{E}+2(\delta-2\alpha)}{1-2\alpha}\\
\frac{\delta-2\alpha}{E} &\qquad\textrm{otherwise}.
\end{cases}
\end{align*}
\end{theorem}

\begin{proof}
For all $t>0$, denote by $\big(X^{T_t},Y^{T_t},Z^{T_t}\big)$ a random variable with distribution $\mu_{T_t}$. For all $\delta>0$, with the definition \eqref{ht} of $h_t$ and triangle inequality,
\begin{equation*}
\mathbb{P}(U(X_t)>\min U+\delta) \leq \mathbb{P}\big(U\big(X^{T_t}\big)>\min U+\delta\big)+\int  |h_t-1| d\mu_{T_t}.
\end{equation*}
Pinsker's inequality gives
\begin{equation} \label{Pins}
\int  |h_t-1| d\mu_{T_t}\leq \bigg(2\int h_t\ln h_td\mu_{T_t}\bigg)^\frac{1}{2},
\end{equation}
which, by Proposition \ref{dissipation0}, together with Proposition \ref{refer1} gives the result.
\end{proof}

The cooling schedule $T_t = E(\ln t)^{-1}$ is optimal with respect to the method of proof for 
Proposition \ref{dissipation}; see Proposition \ref{optimal}. This is a consistent with works in simulated annealing, e.g. \cite{Chiang,Gem,GeH,Gid2,Gid,Hol2,Hol,Kush}.

The 'mountain-like' shape of $r^e$ indicates the bottleneck for the rate of convergence at low and high values of $E$: a small $E$ means convergence to the instantaneous equilibrium $\mu_{T_t}$ is slow and a large $E$ means the convergence of $\mu_{T_t}$ to the global minima of $U$ is slow. 

Although the focus in Theorem \ref{convergence} is for decaying $T_t$, it is only for convergence to the global minimum where Assumption \ref{assumption2} is used. In particular, the convergence result in Proposition \ref{dissipation0} is valid for temperature schedules that are not converging to zero. This includes the instance of using a variable temperature in order to tackle the problem of metastability in the sampling problem.

We proceed below to show exponential convergence to equilibrium for the generalised Langevin equation \eqref{gl} with constant temperature. 
Part of the analysis used in the proof of Proposition \ref{dissipation} can be used for the sampling case and $T_t=T$, i.e. working only with the partial time derivatives mentioned above
for the invariant distribution. 
{
\begin{prop}\label{constT0}
Let Assumption \ref{assumption1} and \ref{assumption4} hold and let $T_t=T$ for all $t$ for some constant $T>0$. There exist constants $C^c, C_*>0$ such that
\begin{equation*}
\int  |h_t-1| d\mu_{T}\leq C^c e^{-\frac{C_*^{-1}}{2}t},
\end{equation*}
for all $t>0$.
\end{prop}
\begin{proof}
See Appendix \ref{appendixa}.
\end{proof}
}

\section{Numerical results}\label{num}

Here we investigate the numerical performance of \eqref{gl}
in terms of convergence to a global optimum and exploration capabilities 
and compare with \eqref{Sim2}. In Section \ref{num1}, we will present the discretisation we use for both sets of dynamics and some details related to the annealing schedule and parameters. 
In Section \ref{num2} and \ref{tune}, for different parameters and cost functions, 
we present results for the probability of convergence to the global minimum and rates of 
transition between different regions of the state space. We will investigate thoroughly the effect of $E$ appearing in the annealing schedule 
as well as the parameters in the dynamics \eqref{Sim2} and \eqref{gl}.

\subsection{Time discretisation}\label{num1}
In order to simulate from \eqref{gl}, we will use the following time discretisation. For $k\in\mathbb{N}$,
\begin{subequations}\label{glsim}
\begin{align}
Y_{k+\frac{1}{2}} &= Y_k-\frac{\Delta t}{2} \gamma \nabla U(X_k) +\frac{\Delta t}{2} \lambda^\top Z_k, \\
X_{k+1} &= X_k+\Delta t \gamma Y_{k+\frac{1}{2}}, \\
Z_{k+1} &= Z_k-\theta \lambda Y_{k+\frac{1}{2}} -\theta A Z_k + \alpha \sqrt{T_k}\ \Sigma \xi_k,\\
Y_{k+1} &= Y_{k+\frac{1}{2}}-\frac{\Delta t}{2} \gamma \nabla U(X_{k+1}) + \frac{\Delta t}{2} \lambda^\top Z_{k+1},
\end{align}
\end{subequations}
where $\Delta t$ denotes the time incremements in the discretisation, $\xi_k$ are i.i.d. standard $m$-dimensional normal random variables with unit variance and
$\theta = 1-\exp(-\Delta t)$, and $\alpha = \sqrt{1-\theta^2}$. Specifically this is method 2 of \cite{BB} applied on a slight modification of \eqref{gl}, where 
$\gamma Y_t dt$ and $\gamma\nabla U dt$ is used instead in the r.h.s. of \eqref{gl1} and \eqref{gl2}. Tuning $\gamma$ can improve numerical perfomance especially
in high dimensional problems,
but we note that this has no effect in terms of the instantaneous
invariant density in \eqref{equilibrium}; similar to $\lambda$ and $A$, $\gamma$ will not appear in \eqref{equilibrium}. Unless stated otherwise, in the remainder
we will use $\gamma = 1$.  

As we will see below the choices for $A$ make a difference in terms of performance. To illustrate this we will use different choices of the form 
\begin{equation*}
A = \mu A_i; 
\end{equation*}
$i$ here is an index for different forms of $A$. The first choice will be to set $m=n$  
and set $A_1=I_n$ where $I_n$ is $n\times n$ identity matrix. For the rest, we
will use $m=2n$ and set
\begin{equation*}
A_2 =
\begin{pmatrix}
1.9I_n & 0.4I_n\\
 0.1I_n & 0.1I_n
\end{pmatrix},\quad
A_3 =
\begin{pmatrix}
 I_n & 0.5I_n\\
 0.5I_n & I_n
\end{pmatrix}, \quad
(A_4)_{ij} =
\begin{cases}
1& \textrm{if }i = j\\
\frac{1}{mn}& \textrm{otherwise}
\end{cases}.
\end{equation*}
Doubling the state space of $Z_t$ relative to $X_t$, $Y_t$ allows investigating the effect of injecting more noise in the dynamics has to the overall performance and the state space exploration.
As per \cite{MR3568348} (following \cite{MR3030715}), the constraint that the trace of $A$ is uniformly bounded has been used in selecting the above matrices. Note that $A_2$ does not satisfy the symmetry assumption for the results, but figures for $A_2$ are displayed in spite of this because there is an interesting improvement in performance for one of the cases below (see Figure \ref{fig:u1} and also others for the sake of comparison). 
Similarly we will use in each case $\lambda=\bar{\lambda}\lambda_i$ with $\bar{\lambda}>0$, $\lambda_1=I_n$ and 
\[\lambda_i=\begin{pmatrix}
I_n\\
 0
\end{pmatrix}
\]
for $i=2,3,4$. As a result $\bar{\lambda},\mu>0$ are the main tuning constants for \eqref{glsim} that do not involve the 
annealing schedule.

The Langevin system \eqref{Sim2} will be approximated with a similar leapfrog scheme,
\begin{subequations}\label{lsim}
\begin{align}
Y_{k+\frac{1}{2}} &= Y_k - \hat{\theta} \gamma \nabla U(X_k) - \hat{\theta} \mu Y_k + \hat{\alpha}\sqrt{\mu T_k} \xi_k,\label{lsim1}\\
X_{k+1} &= X_k + \Delta t \gamma Y_k, \\
Y_{k+1} &= Y_{k+\frac{1}{2}} - \hat{\theta} \gamma \nabla U(X_{k+1}) - \hat{\theta} \mu Y_{k+\frac{1}{2}} + \hat{\alpha}\sqrt{\mu T_{k+1}} \xi_{k+\frac{1}{2}},\label{lsim2}
\end{align}
\end{subequations}
for $\hat{\theta} = 1-\exp(-\frac{\Delta t}{2})$, and $\hat{\alpha} = \sqrt{1-\hat{\theta}^2}$, where in the implementation, \eqref{lsim1} and \eqref{lsim2} are combined (aside from the first iteration) and only integer-indexed $\xi$ are used. To make valid comparisons, both \eqref{glsim} and \eqref{lsim} will use $\gamma=1$ and the same noise realisation $\xi_k$ (or the first common $n$ elements) and the same step size $\Delta t$.

Finally for both cases we will use following annealing schedule:
\begin{equation*}
T_k = \bigg(\frac{1}{5}+\frac{\ln(1+k\Delta t)}{E}\bigg)^{-1},
\end{equation*}
where $E$ is an additional tuning parameter (since $\hat{E}$ is unknown in general).


\subsection{Sample path properties}\label{num2}
Our first set of simulations focus on illustrating some properties of the sample paths generated by \eqref{glsim} and \eqref{lsim}.
We will use the following bivariate potential function 
\begin{align}
U(x_1,x_2) &= \frac{x_1^2}{5}+\frac{x_2^2}{10}+5e^{-x_1^2}-7e^{-(x_1+5)^2-(x_2-3)^2}-6e^{-(x_1-5)^2-(x_2+2)^2}\nonumber\\
&\qquad+\frac{\frac{2}{3}x_1^2e^{-\frac{x_1^2}{9}}\cos(x_1+2x_2)\cos(2x_1-x_2)}{1+\frac{x_2^2}{9}}.\label{bivar_potential}
\end{align} 

The global minimum is located at $(-5,3)$, but there are plenty of local minima where the process can get trapped. In addition, there is a barrier
along the vertical line $\{x_1=0\}$ that makes crossing from each half plane less likely. Here we set $\Delta t = 0.1$, $E=5$ and
each sample is initialised at $(4,2)$. As a result, it is harder to cross $\{x_1=0\}$ to reach the global minimum and it is quite common to get stuck in other local
minima such as near $(5,-2)$. We use the number of crossings on $\{x_1=0\}$ as a scale for how stuck the process is in Table \ref{tab:crossings}. Note that the asymmetric $A=A_2$ case displays the smallest number of crossings.

To illustrate this, in Figure \ref{fig:samples} we present contour plots of $U$ together 
with a typical realisation of sample paths (in the left panels) for \eqref{lsim} and \eqref{glsim} for the different choices of $A_i$. 
As expected, \eqref{glsim} generates smoother paths than those of \eqref{lsim}. 
We also employ independent runs of each stochastic process for the same initialisation. The results
are presented in the right panels of Figure \ref{fig:samples},
where we show heat maps for two dimensional histograms representing the frequency of visiting each $(x_1,x_2)$ location over $20$ independent realisations of each process. 
The heat maps in Figure~\ref{fig:samples} do not directly depict time dependence in the paths 
and only illustrate which areas are visited more frequently.
Of course converging at the global minumum or the local one at $(5,-2)$ will result in more visits at these areas. 
The aim here is to investigate the exploration of the state space. 

\begin{figure}[h]
  \centering
  \begin{subfigure}[h]{0.3\linewidth}
    \adjincludegraphics[width=\linewidth]{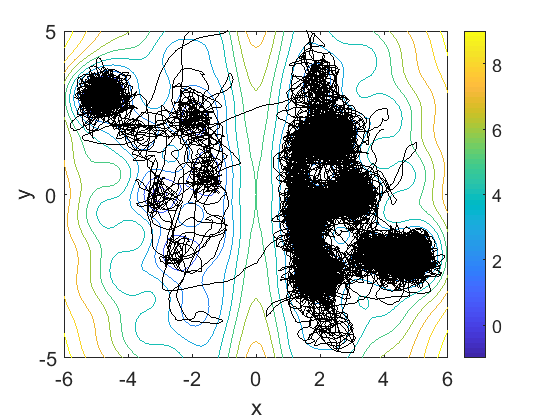}
  \end{subfigure}
    \begin{subfigure}[h]{0.3\linewidth}
    \adjincludegraphics[ width=\linewidth]{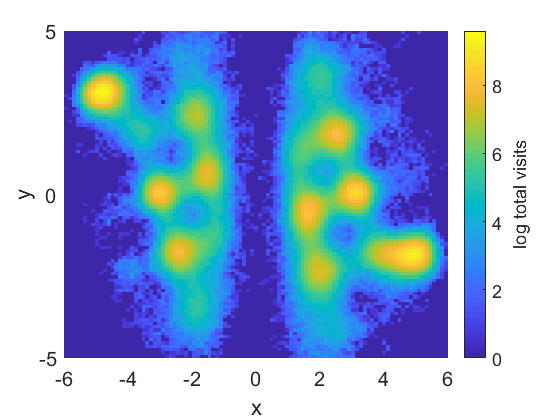}
  \end{subfigure}
  
  \begin{subfigure}[h]{0.3\linewidth}
    \adjincludegraphics[ width=\linewidth]{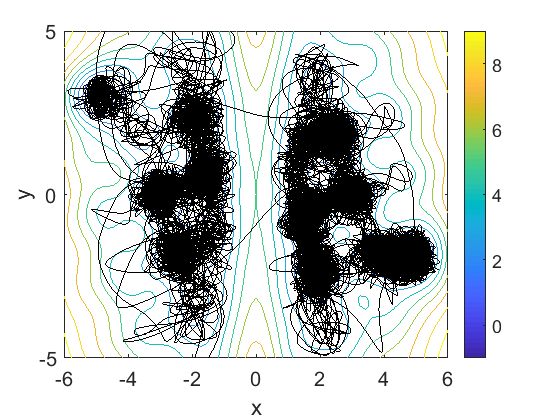}
  \end{subfigure}
   \begin{subfigure}[h]{0.3\linewidth}
    \adjincludegraphics[width=\linewidth]{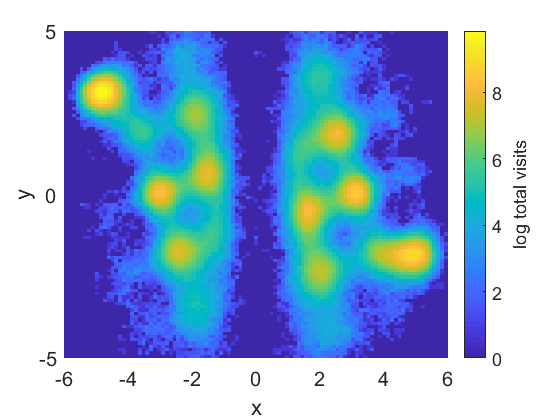}
  \end{subfigure}
  
  \begin{subfigure}[h]{0.3\linewidth}
    \adjincludegraphics[width=\linewidth]{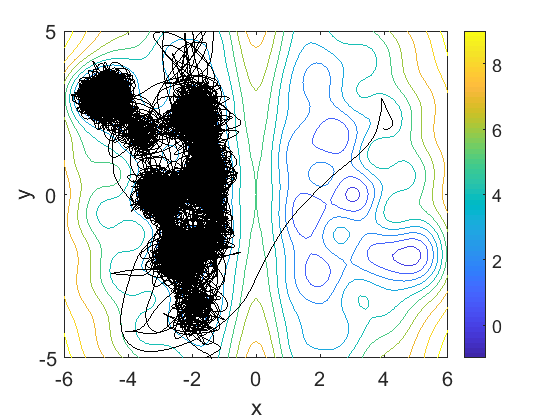}
  \end{subfigure} 
   \begin{subfigure}[h]{0.3\linewidth}
    \adjincludegraphics[width=\linewidth]{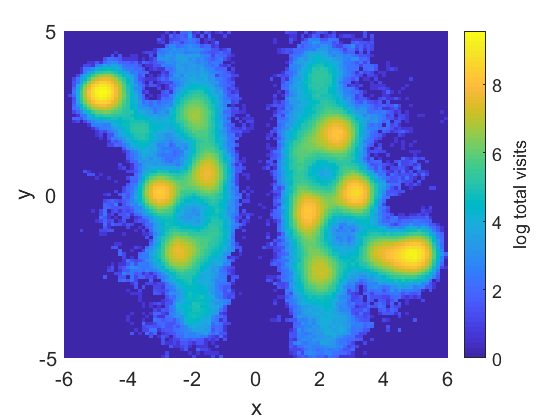}
  \end{subfigure}
  
  \begin{subfigure}[h]{0.3\linewidth}
    \adjincludegraphics[width=\linewidth]{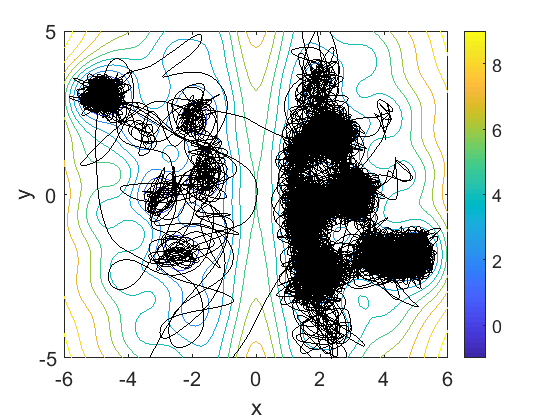}
  \end{subfigure} 
  \begin{subfigure}[h]{0.3\linewidth}
    \adjincludegraphics[width=\linewidth]{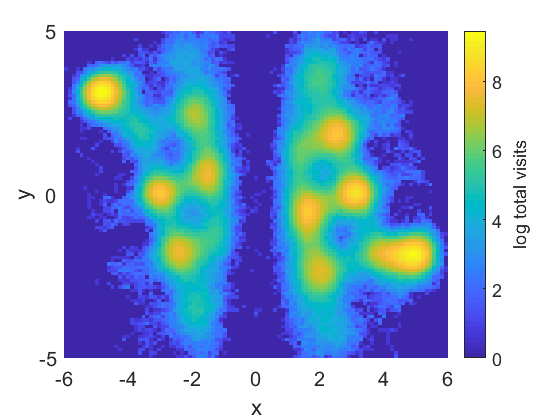}
  \end{subfigure}
  
  \begin{subfigure}[h]{0.3\linewidth}
    \adjincludegraphics[width=\linewidth]{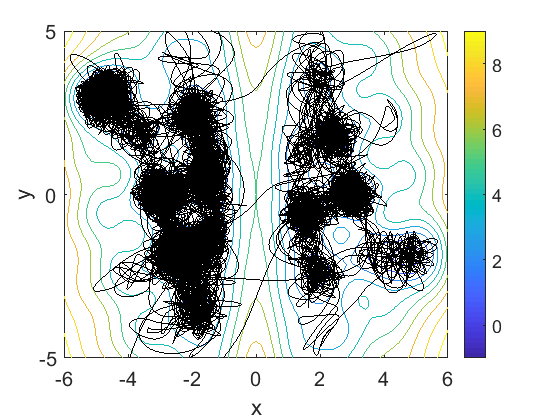}
  \end{subfigure}
  \begin{subfigure}[h]{0.3\linewidth}
    \adjincludegraphics[width=\linewidth]{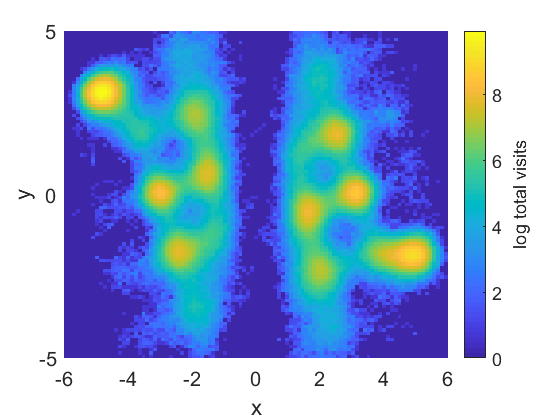}
  \end{subfigure}
  
  \caption{Dynamics in order from top: \eqref{lsim}, \eqref{glsim} with $A=A_1,\dots,A_4$. 
  Left: One instance of noise realisation. Right: Log histogram of $20$ independent runs. }
  \label{fig:samples}
\end{figure}

\begin{table}[h]
\begin{center}
\begin{tabular}{ | c|c | } 
 \hline
Method equation & Number of transitions across $x=0$\\ 
 \hline
\eqref{lsim} & 11295 \\ 
\eqref{glsim} with $A=A_1$ & 11893 \\ 
\eqref{glsim} with $A=A_2$ & 10915\\ 
\eqref{glsim} with $A=A_3$ & 11728\\ 
\eqref{glsim} with $A=A_4$ & 11771\\ 
\hline
\end{tabular}
\end{center}
\caption{Number of crossings across the vertical line $\{x_1=0\}$ for $U$ defined in \eqref{bivar_potential}. The results are summed from $k=10^5$ iterations of $10^4$ independent runs.}\label{tab:crossings}
\end{table}

\subsection{Performance and tuning}\label{tune}
As expected, the tuning parameters, $E$, $\bar{\lambda}$ and $\mu$ play significant roles in the performance of \eqref{glsim} and \eqref{lsim}. 
As $E$ is common to both, we wish to demontrate that the additional tuning variable for \eqref{glsim} will improve performance.

We first comment on relative scaling of $\bar{\lambda}$ and $\mu$ based on earlier work for quadratic $U$ and $T_t=T$ being constant.
A quadratic $U$ satisfies the bounds in Assumption \ref{assumption1} and is of particular interest because analytical calculations are possible for the spectral gap of $L_t$, which in turn gives the (exponential) rate of convergence to the equilibrium distribution. 
It is observed numerically in \cite{Otto} that in this case, \eqref{gl} has a spectral gap that is approximately a function of $\frac{\bar{\lambda}^2}{\mu}$.
On the other hand, the spectral gap of \eqref{Sim2} with quadratic $U$ is a function of $\mu$ thanks to Theorem 3.1 in \cite{spec}.
For the rest of the comparison, we will use $\frac{\bar{\lambda}^2}{\mu}$ and $\mu$ as variables for \eqref{glsim} and \eqref{lsim} respectively as 
these quantities appear to have a distinct effect on the mixing in each case.


We will consider three different cost functions $U$ and set $\Delta t=0.02$. As before we will initialise
at a point well separated from the global minimum and consider each method to be successful if it convergences
at a particular tolerance region near the global minumum. The
details are presented in Table \ref{tab:tolerance}. We choose the popular Alpine function in $12$ dimensions ($\nabla U_1$ here is a subgradient) 
and two variants of \eqref{bivar_potential}. 
$U_2$ is modified to have the same quadratic confinement in $x_1$ and $x_2$ direction and there are several additional local minima due to the last term in the sum.
More importantly, compared to \eqref{bivar_potential} (and $U_3$) it has a narrow region near the origin that allows easier passage through $\{x_1=0\}$.
On the other hand $U_3$ similar to \eqref{bivar_potential} except that the well near the global minimum (and the dominant local minimum
at $(5,-2)$) are elongated in the direction of $x_2$ (and $x_1$ respectively). 
 
\begin{table}[h]
\begin{center}
\begin{tabular}{ |c|c|c| } 
\hline
Cost function & Initial condition & Tolerance sets\\ 
 \hline
$U_1(x)=\frac{1}{2}\sum_{i=1}^{12}|x_i \sin(x_i)+0.1x_i|$ & $x_j=6$  $\forall j$ & $x_j\in[-2,2]$  $\forall j$\\ 
\hline
$\begin{array}{c}
U_2(x_1,x_2) = \frac{x_1^2}{7}+\frac{x_2^2}{7}+5\Big(1-e^{-9x_2^2}\Big)e^{-x_1^2} - 7e^{-(x_1+5)^2-(x_2-3)^2}\\
-6e^{-(x_1-5)^2-(x_2+2)^2}+\frac{\frac{2}{3}x_1^2e^{-\frac{x_1^2}{9}}\cos(x_1+2x_2)\cos(2x_1-x_2)}{1+\frac{x_2^2}{9}}
\end{array}$ & $\begin{array}{c}x_1=4,\\ x_2=2\end{array}$ & $\begin{array}{c} x_1\in[-6.5,-4.5],\\ x_2\in [1.5,4.5]\end{array}$\\ 
\hline
$\begin{array}{c}
 U_3(x_1,x_2) = \frac{x_1^2}{5}+\frac{x_2^2}{10}+5e^{-x_1^2} -7e^{-2(x_1+5)^2-\frac{(x_2-3)^2}{5}} -6e^{-\frac{(x_1-5)^2}{5}-2(x_2+2)^2}
\end{array}
$ & $\begin{array}{c}x_1=4,\\ x_2=2\end{array}$ & $\begin{array}{c} x_1\in[-6.5,-4.5],\\ x_2\in [1.5,4.5]\end{array}$\\ 
\hline
\end{tabular}
\caption{Details of three different cost functions, initialisation and tolerance regions corresponding to regions of attraction of the global minimum.}\label{tab:tolerance}
\end{center}
\end{table}

In Figure \ref{fig:u1} we present proportions of simulations converging at the region near the global 
minimum for $U=U_1$ depending on $E$ and $\mu$ for \eqref{lsim} and on $E$ and $\frac{\bar{\lambda}^2}{\mu}$ for \eqref{glsim} 
based on discussion above. To produce the figures related to \eqref{glsim} after setting $E,\frac{\bar{\lambda}^2}{\mu}$ we pick a random value of $\mu$ from a grid.
The aim of this procedure is to ease visualisation, reduce computational cost and to emphasise that it is  $\frac{\bar{\lambda}^2}{\mu}$ that is crucial for mixing and 
the performance here is not a product of a tedious tuning for $\mu$.
The left panels of Figure \ref{fig:u1} are based on final state and the right on an average of the positions (of $X$) over the last $5000$ iterations.
In this example it is clear empirically that the generalised Langevin dynamics \eqref{glsim} result in a higher probability of reaching the global minumum. Another interesting observation is 
that for the generalised Langevin dynamics good performance is more robust to the chosen value of $E$. In this example, 
this means that adding
an additional tuning variable and scaling $\mu$ proportional to $\bar{\lambda}^2$, 
makes it easier to find a configuration of the parameters $E,\mu,\bar{\lambda}$ that leads to good perfomance,
compared to using \eqref{lsim} and tuning $E,\mu$. 
\textcolor{black}{It's also worth noting the cases of small $E$ where the generalised Langevin dynamics performs significantly better than the Langevin dynamics in the top plot and even than the case of the same dynamics and larger $E$. This is an improvement that is not completely encapsulated by the analytic results here; it indicates that the deterministic dynamics ($E=0$) can be inherently much more successful at climbing out of local minima, which translates to better convergence rates in the $E > 0$ cases.}


In Figures \ref{fig:u2} and \ref{fig:u3} we present results for $U_2$ and $U_3$. A notable difference to Figure \ref{fig:u1} here is that the panels on the left show 
proportions of the position average of the last 5000 iterations being near the
correct global minimum and the panels on the right present the number of jumps across $\{x_1=0\}$ demonstrated by a position average at each iteration of the previous 5000 iterations. More precisely, the panels on the right show the number of jumps shown by the trajectory 
\begin{equation*}
\tilde{X}_k = \frac{1}{5000}\sum_{k' = 1}^{5000} X_{k-k'+1} 
\end{equation*}
for all $k>5000$. All results are averaged over $20$ independent runs. 
The aim here is to measure the extent of exploration 
of each process similar to Table \ref{tab:crossings}. 
We observe that in both cases using \eqref{glsim} leads to a similar number of jumps. 
We believe the benefit of the higher order dynamics here are the robustness of performance for different values of $E$
and $\frac{\bar{\lambda}^2}{\mu}$. This is especially for using $A_3$ and $A_4$. 
Finally we note that despite similarities between $U_2$
and $U_3$ there are significant features that are different: the sharpness in the confinement, 
the shape and number of attracting wells and the shape of barriers that obstruct crossing regions in the state space.
This will have a direct effect in performance, which can explain the difference in 
performance when comparing Figures \ref{fig:u2} and \ref{fig:u3}; $U_3$ is a harder cost function to minimise. 

\textcolor{black}{The selection $A=A_2$, shown as the middle row in each of Figures \ref{fig:u1}, \ref{fig:u2} and \ref{fig:u3}, does not satisfy the probably superfluous symmetry assumption as stated in the introduction, but it is noteworthy that the performance varies to such a large extent for different $U$ and that any optimality of $A$, left as future work, could change depending on whether the symmetry assumption is in place.  }

\begin{figure}[!h]
  \centering
  \begin{subfigure}[h]{0.3\linewidth}
    \adjincludegraphics[trim = {0 0 {0.04\width} {0.04\width}}, clip, width=\linewidth]{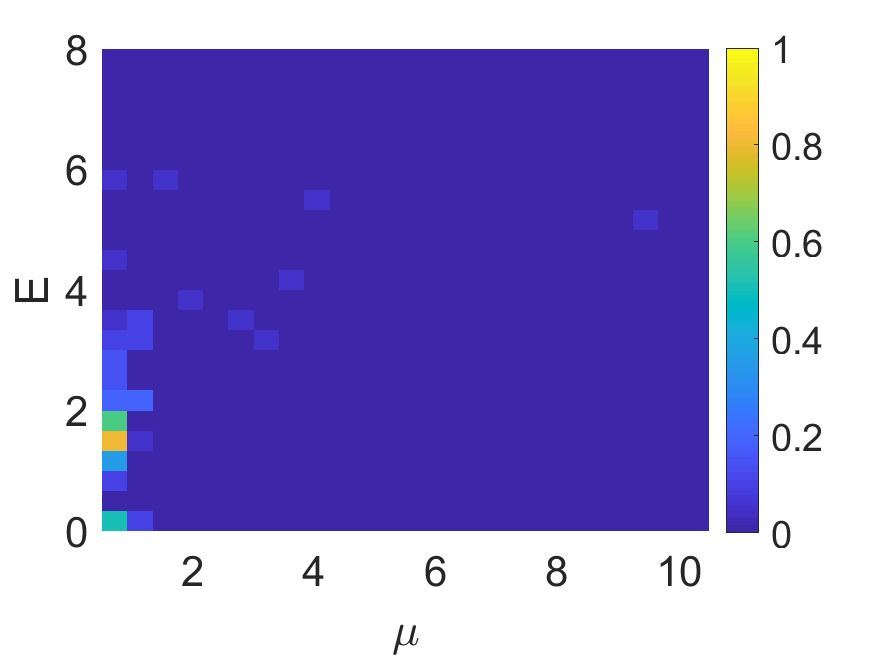}
  \end{subfigure}
  \begin{subfigure}[h]{0.3\linewidth}
    \adjincludegraphics[trim = {0 0 {0.04\width} {0.04\width}}, clip, width=\linewidth]{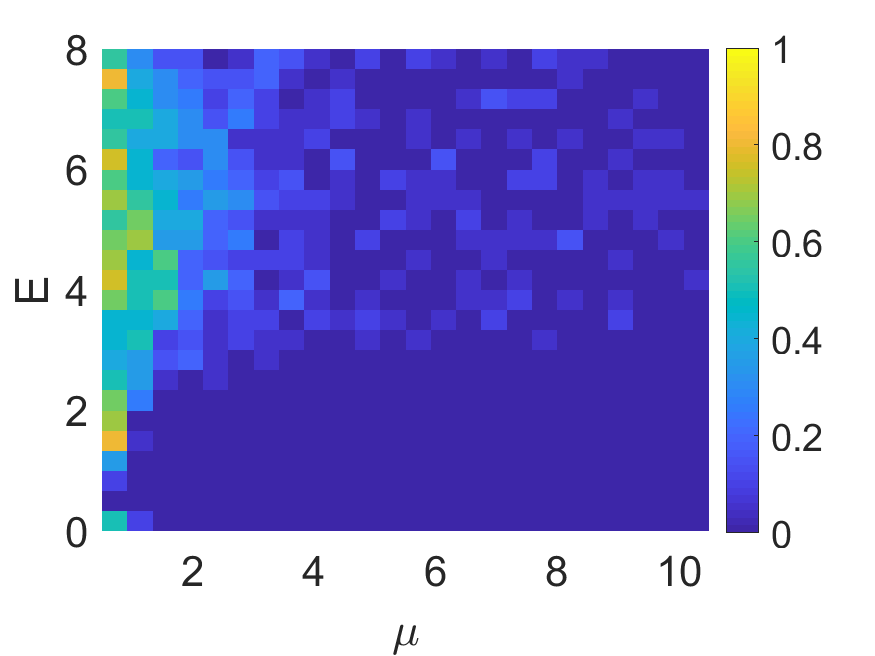}
  \end{subfigure}\\
  \begin{subfigure}[h]{0.3\linewidth}
    \adjincludegraphics[trim = {0 0 {0.04\width} {0.04\width}}, clip, width=\linewidth]{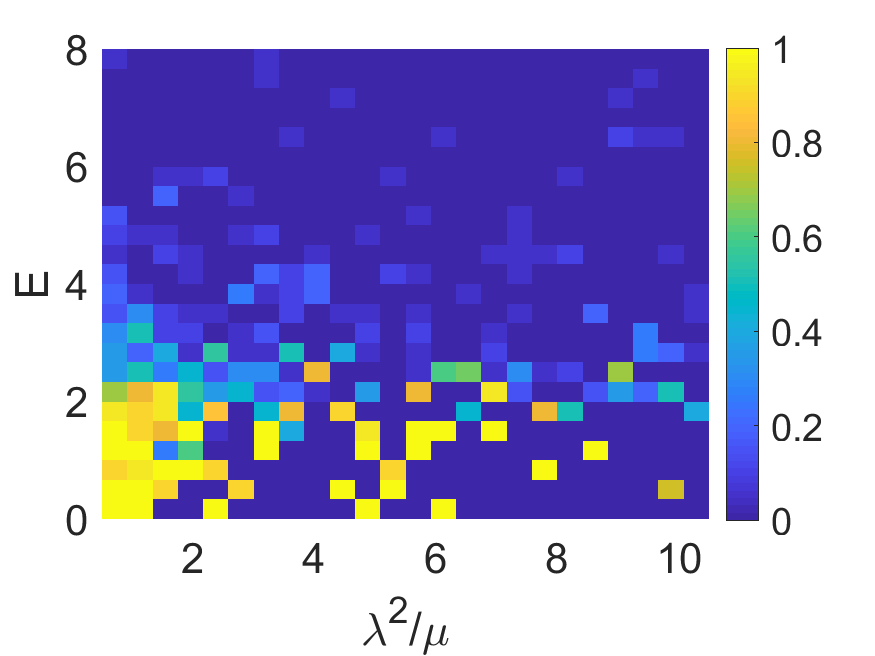}
  \end{subfigure}
  \begin{subfigure}[h]{0.3\linewidth}
    \adjincludegraphics[trim = {0 0 {0.04\width} {0.04\width}}, clip, width=\linewidth]{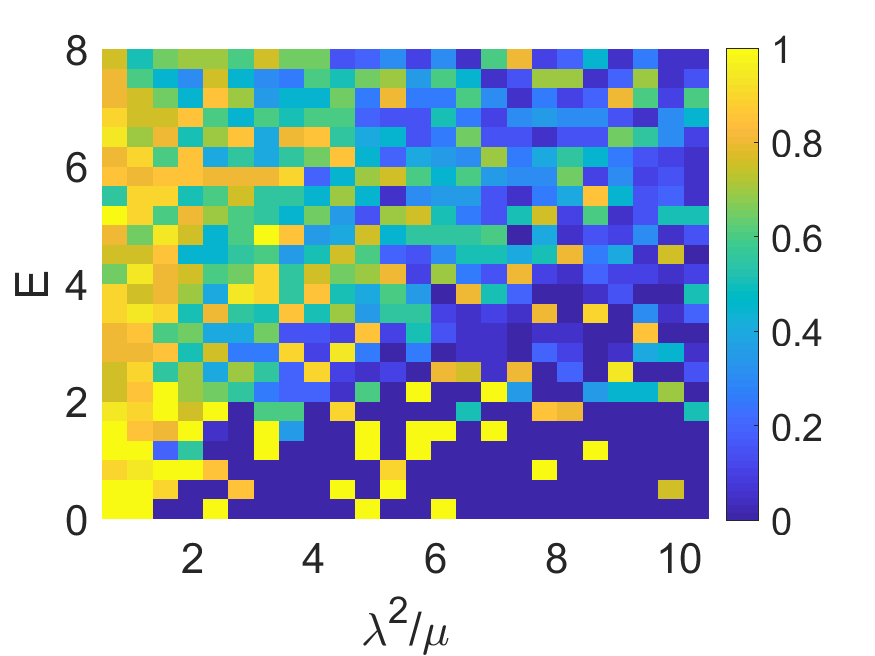}
  \end{subfigure}\\
  \begin{subfigure}[h]{0.3\linewidth}
    \adjincludegraphics[trim = {0 0 {0.04\width} {0.04\width}}, clip, width=\linewidth]{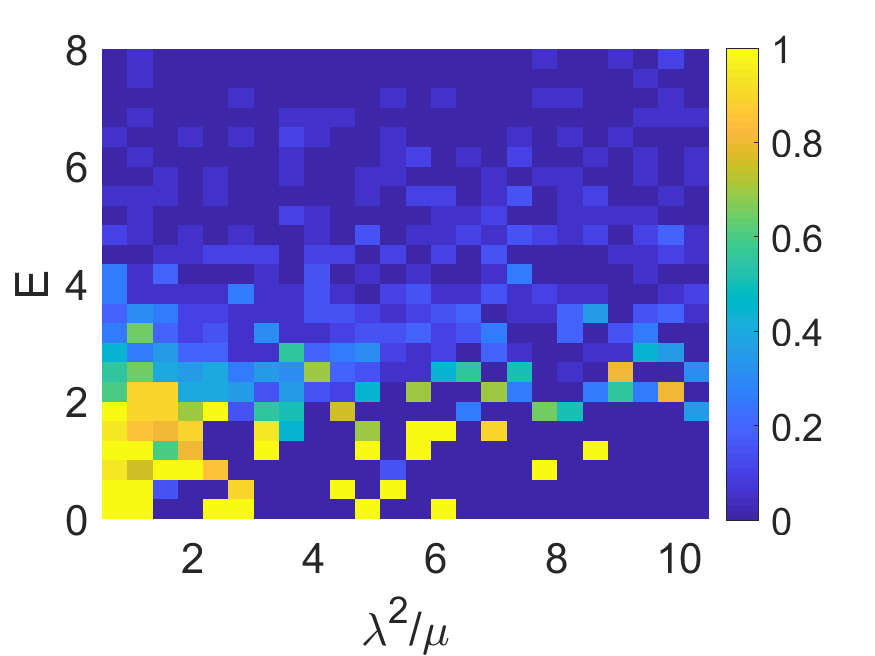}
  \end{subfigure}
  \begin{subfigure}[h]{0.3\linewidth}
    \adjincludegraphics[trim = {0 0 {0.04\width} {0.04\width}}, clip, width=\linewidth]{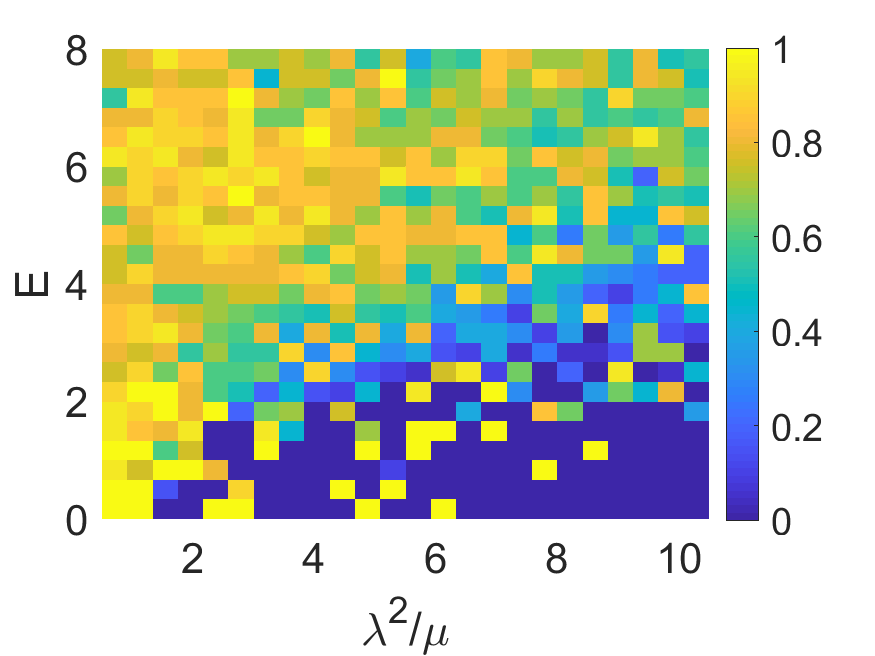}
  \end{subfigure}\\
  \begin{subfigure}[h]{0.3\linewidth}
    \adjincludegraphics[trim = {0 0 {0.04\width} {0.04\width}}, clip, width=\linewidth]{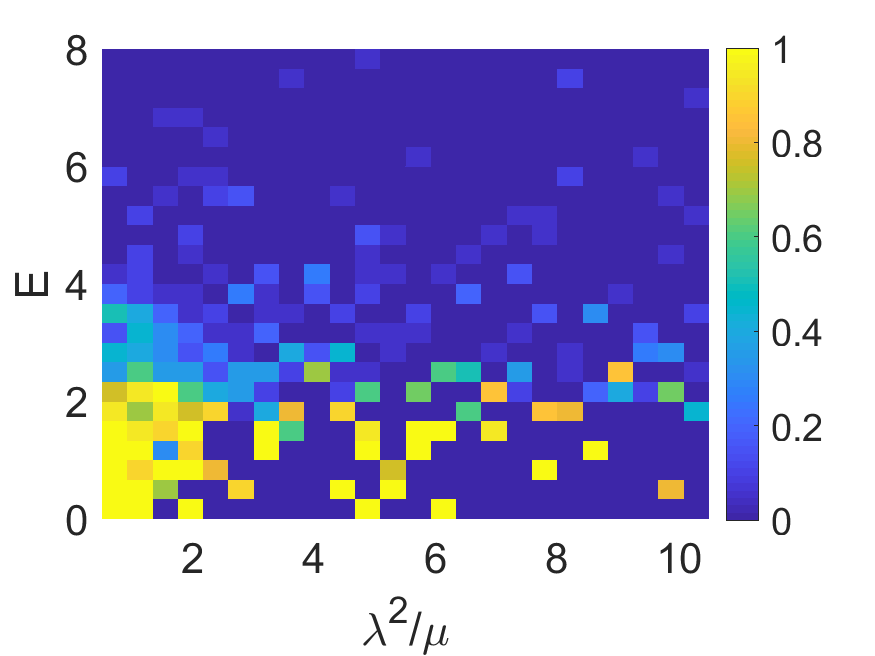}
  \end{subfigure}
  \begin{subfigure}[h]{0.3\linewidth}
    \adjincludegraphics[trim = {0 0 {0.04\width} {0.04\width}}, clip, width=\linewidth]{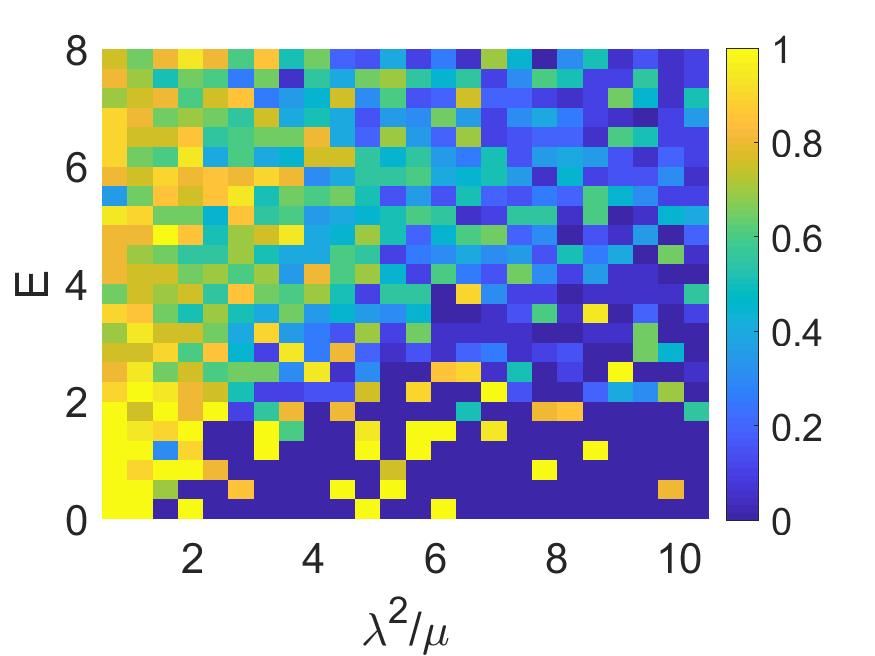}
  \end{subfigure}\\
  \begin{subfigure}[h]{0.3\linewidth}
    \adjincludegraphics[trim = {{0.04\width} 0 {0.04\width} {0.04\width}}, clip, width=\linewidth]{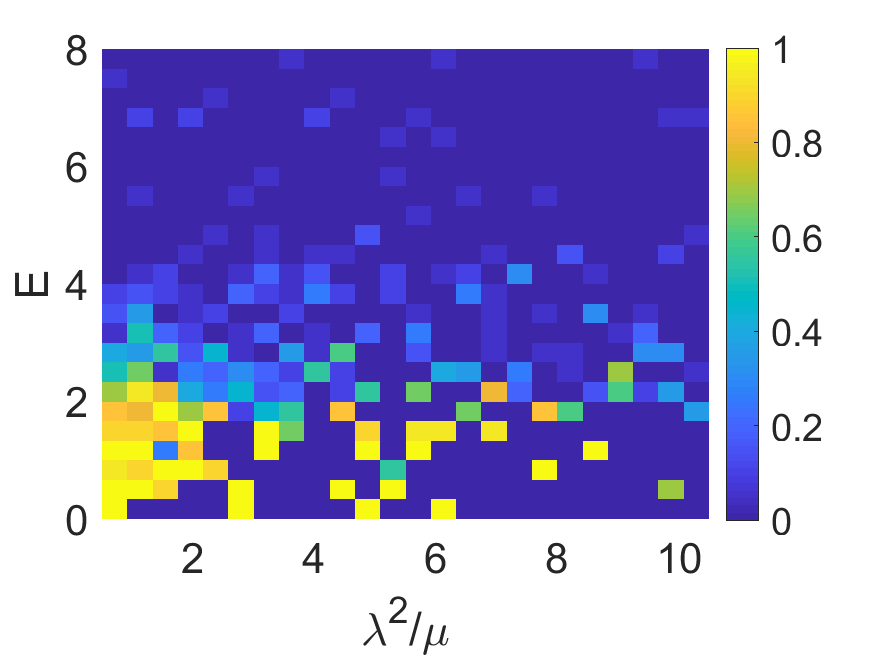}
  \end{subfigure}
  \begin{subfigure}[h]{0.3\linewidth}
    \adjincludegraphics[trim = {{0.04\width} 0 {0.04\width} {0.04\width}}, clip, width=\linewidth]{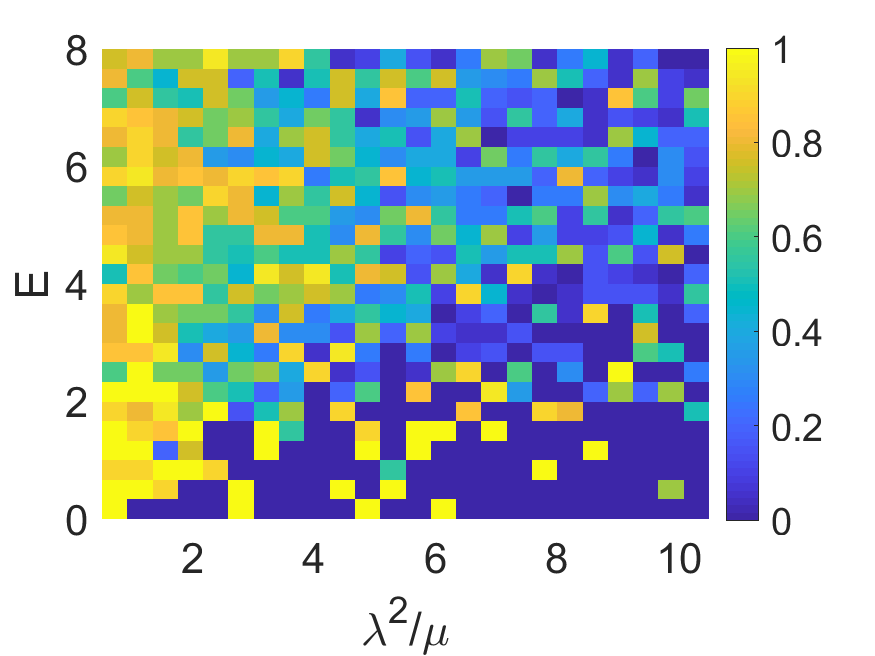}
  \end{subfigure}
  \caption{Proportion of simulations satisfying optimality tolerance for $U=U_1$. Panels from top to bottom: \eqref{lsim}, \eqref{glsim} with $A=A_1,A_2,A_3$. 
  Left: Final position. Right: time-average of last 5000 iterations.  We use $\gamma = 3$ for improving visualisation, the results and improvement in using \eqref{glsim} are similar
  for the case of $\gamma=1$. Results here are for
  20 independent runs and  $k\leq 5\cdot10^4$ -iterations.}
  \label{fig:u1}
\end{figure}

\begin{figure}[!h]
  \centering
  \begin{subfigure}[h]{0.3\linewidth}
    \adjincludegraphics[trim = {0 0 {0.04\width} {0.04\width}}, clip, width=\linewidth]{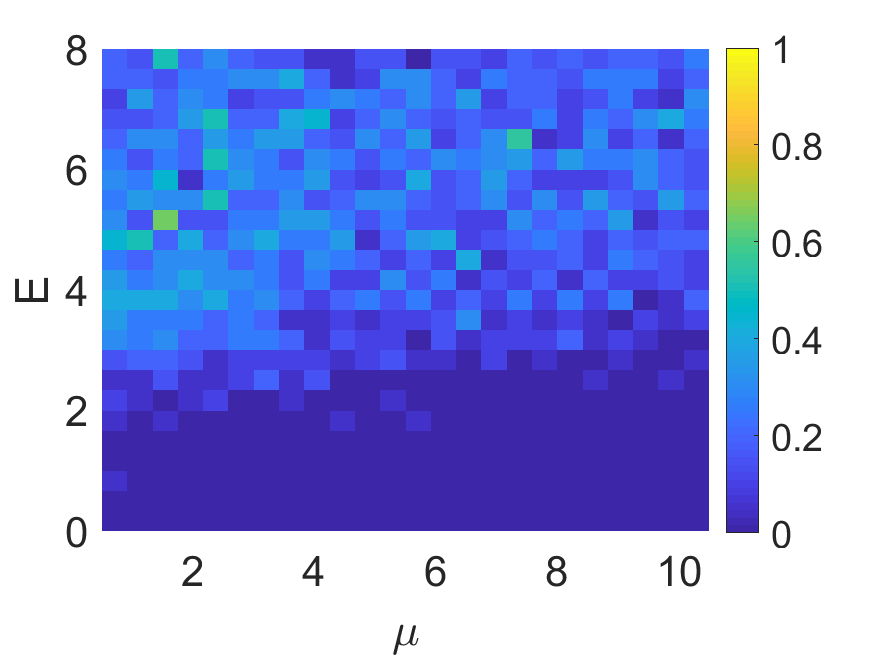}
  \end{subfigure}
  \begin{subfigure}[h]{0.3\linewidth}
    \adjincludegraphics[trim = {0 0 {0.04\width} {0.04\width}}, clip, width=\linewidth]{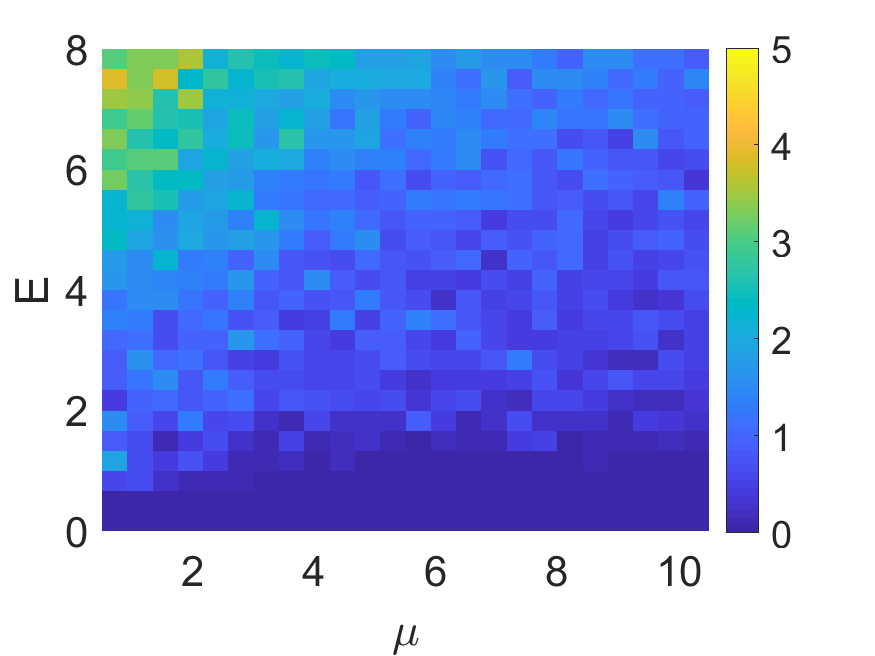}
  \end{subfigure}\\
  \begin{subfigure}[h]{0.3\linewidth}
    \adjincludegraphics[trim = {0 0 {0.04\width} {0.04\width}}, clip, width=\linewidth]{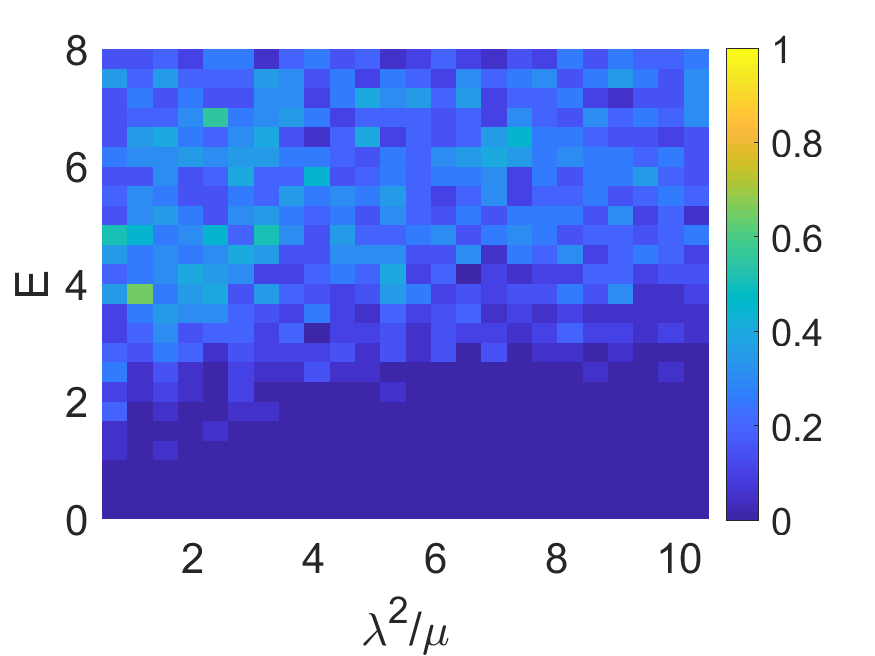}
  \end{subfigure}
  \begin{subfigure}[h]{0.3\linewidth}
    \adjincludegraphics[trim = {0 0 {0.04\width} {0.04\width}}, clip, width=\linewidth]{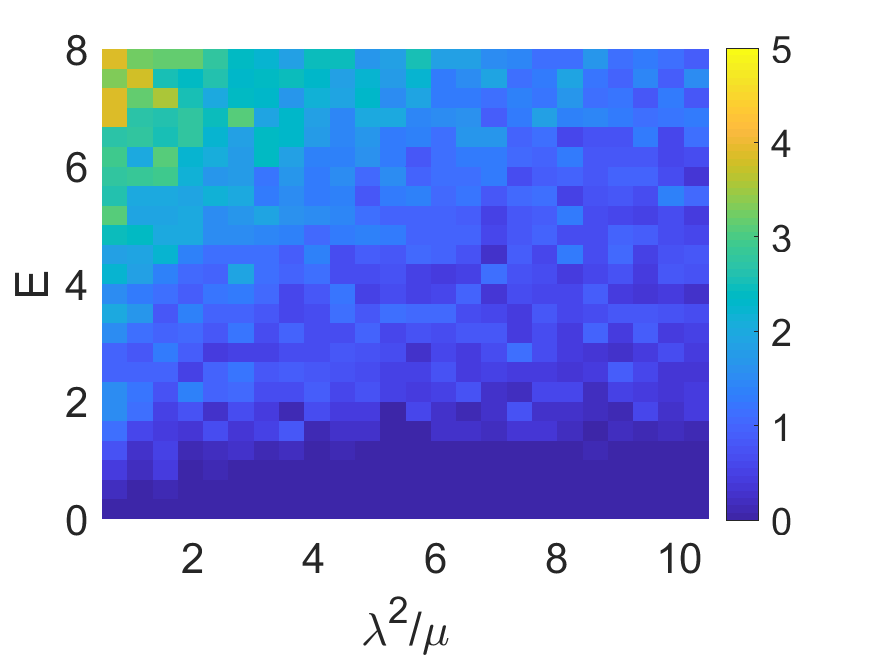}
  \end{subfigure}\\
  \begin{subfigure}[h]{0.3\linewidth}
    \adjincludegraphics[trim = {0 0 {0.04\width} {0.04\width}}, clip, width=\linewidth]{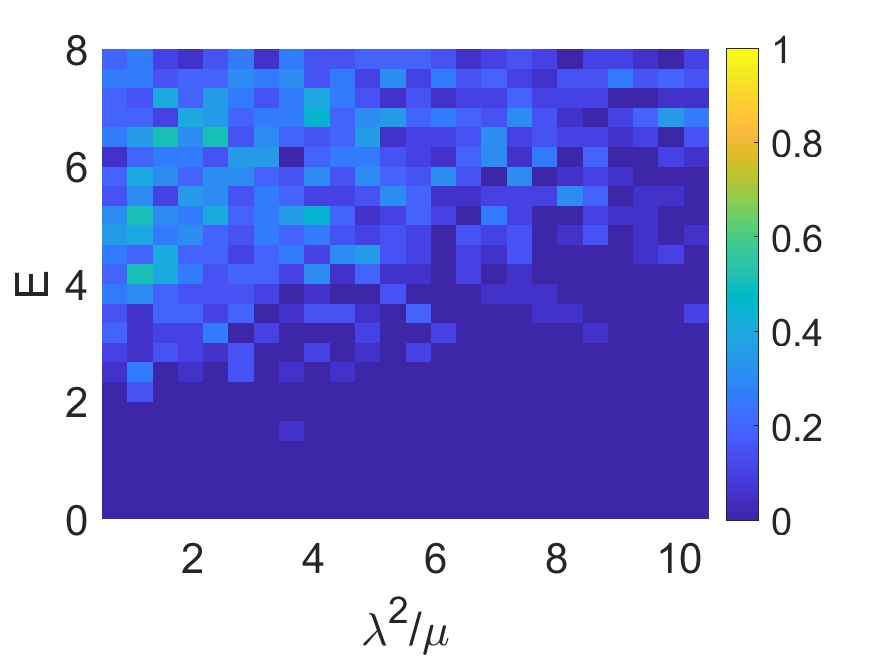}
  \end{subfigure}
  \begin{subfigure}[h]{0.3\linewidth}
    \adjincludegraphics[trim = {0 0 {0.04\width} {0.04\width}}, clip, width=\linewidth]{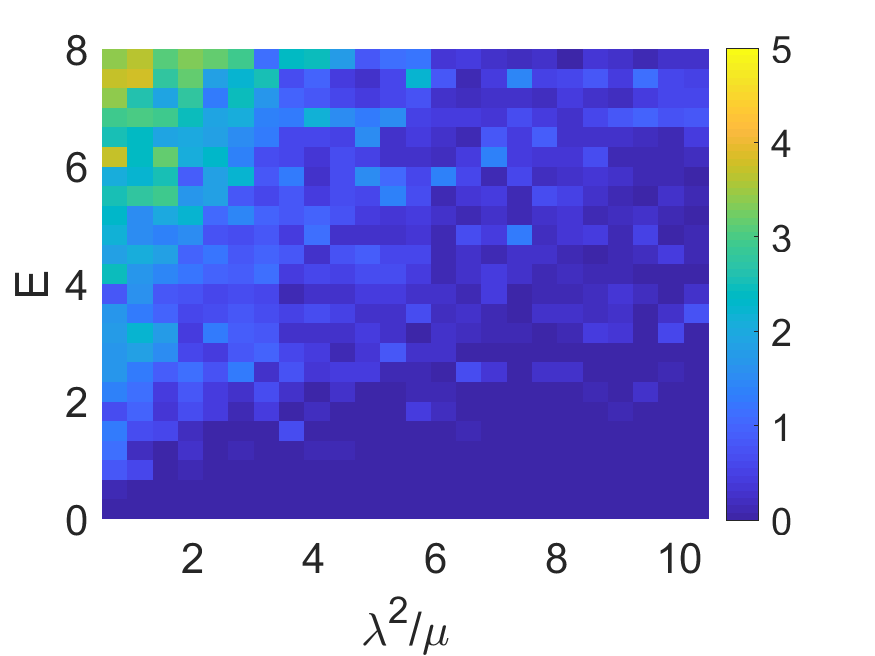}
  \end{subfigure}\\
  \begin{subfigure}[h]{0.3\linewidth}
    \adjincludegraphics[trim = {0 0 {0.04\width} {0.04\width}}, clip, width=\linewidth]{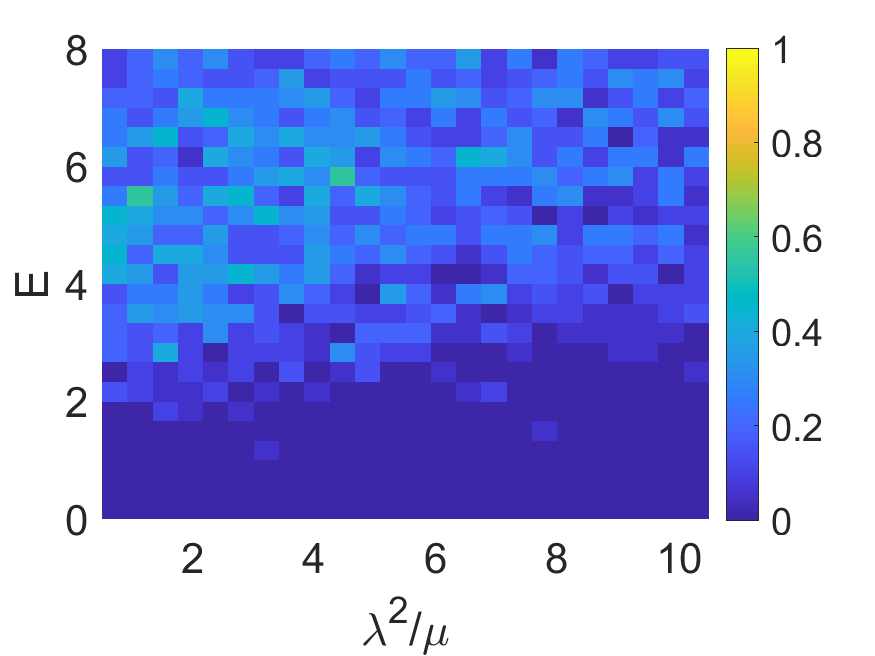}
  \end{subfigure}
  \begin{subfigure}[h]{0.3\linewidth}
    \adjincludegraphics[trim = {0 0 {0.04\width} {0.04\width}}, clip, width=\linewidth]{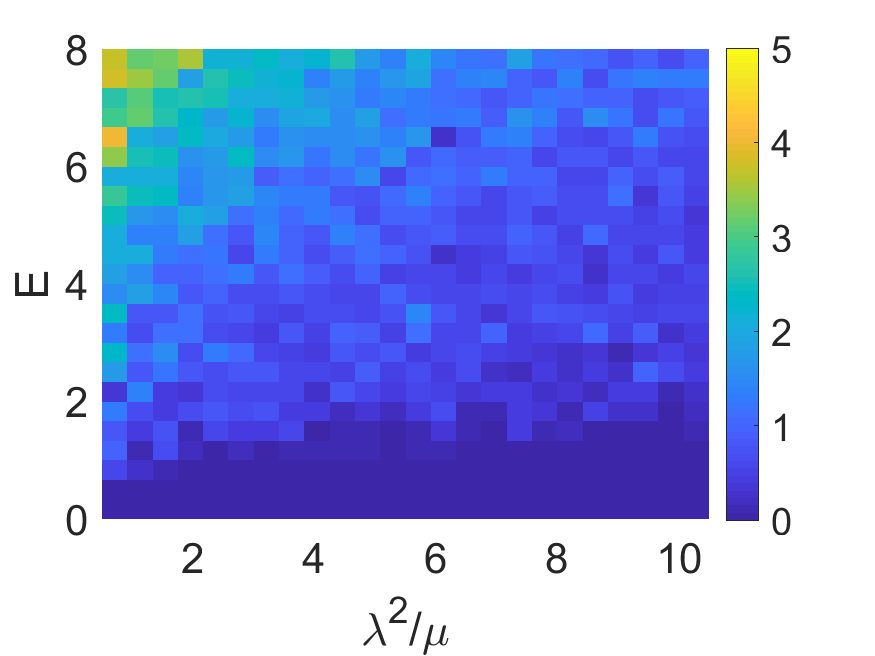}
  \end{subfigure}\\
  \begin{subfigure}[h]{0.3\linewidth}
    \adjincludegraphics[trim = {0 0 {0.04\width} {0.04\width}}, clip, width=\linewidth]{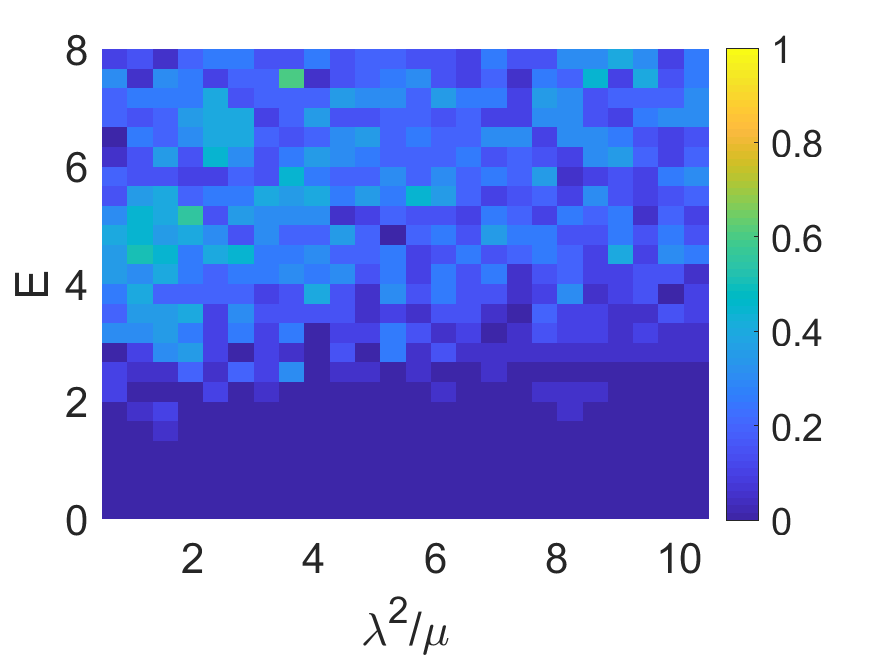}
  \end{subfigure}
  \begin{subfigure}[h]{0.3\linewidth}
    \adjincludegraphics[trim = {0 0 {0.04\width} {0.04\width}}, clip, width=\linewidth]{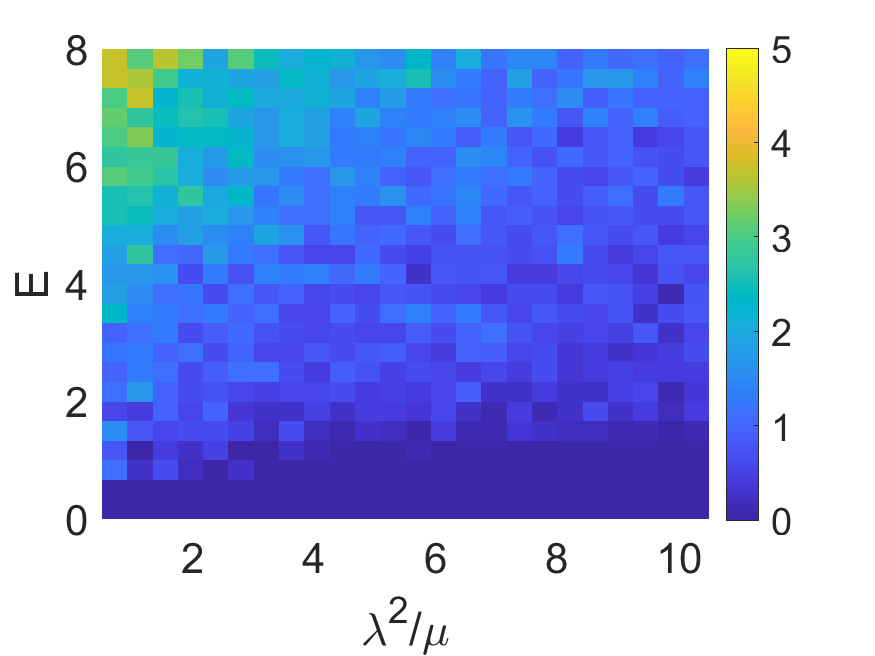}
  \end{subfigure}
  \caption{Both proportion of success and numerical transition rates for $U=U_2$. Panels from top to bottom: \eqref{lsim}, \eqref{glsim} with $A=A_1,A_2,A_3,A_4$.
  Left: Proportion satisfying the optimality tolerance for time-average of last 5000 iterations. 
  Right: Average number of crossings of position averages over 5000 iterations across $\{x_1=0\}$ for each independent run. 
  The remaining details are as in caption of Figure \ref{fig:u1}.}
  \label{fig:u2}
\end{figure}

\begin{figure}[!h]
  \centering
  \begin{subfigure}[h]{0.3\linewidth}
    \adjincludegraphics[trim = {0 0 {0.04\width} {0.04\width}}, clip, width=\linewidth]{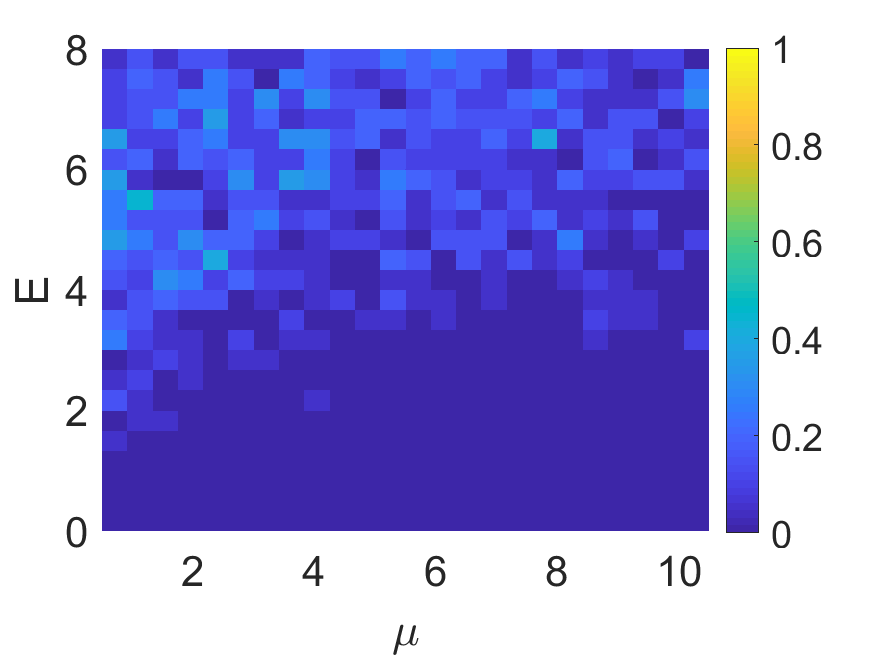}
  \end{subfigure}
  \begin{subfigure}[h]{0.3\linewidth}
    \adjincludegraphics[trim = {0 0 {0.04\width} {0.04\width}}, clip, width=\linewidth]{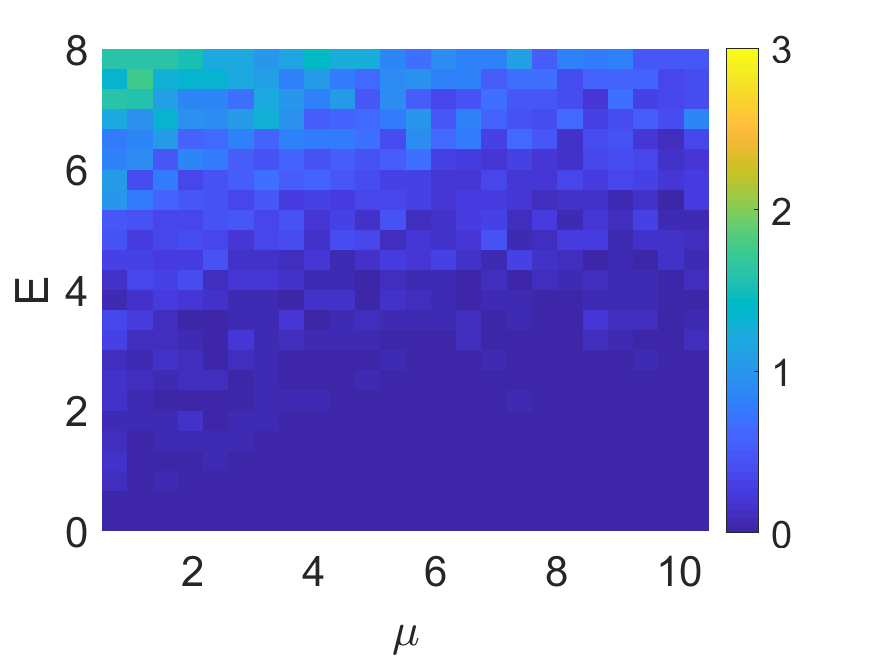}
  \end{subfigure}\\
  \begin{subfigure}[h]{0.3\linewidth}
    \adjincludegraphics[trim = {0 0 {0.04\width} {0.04\width}}, clip, width=\linewidth]{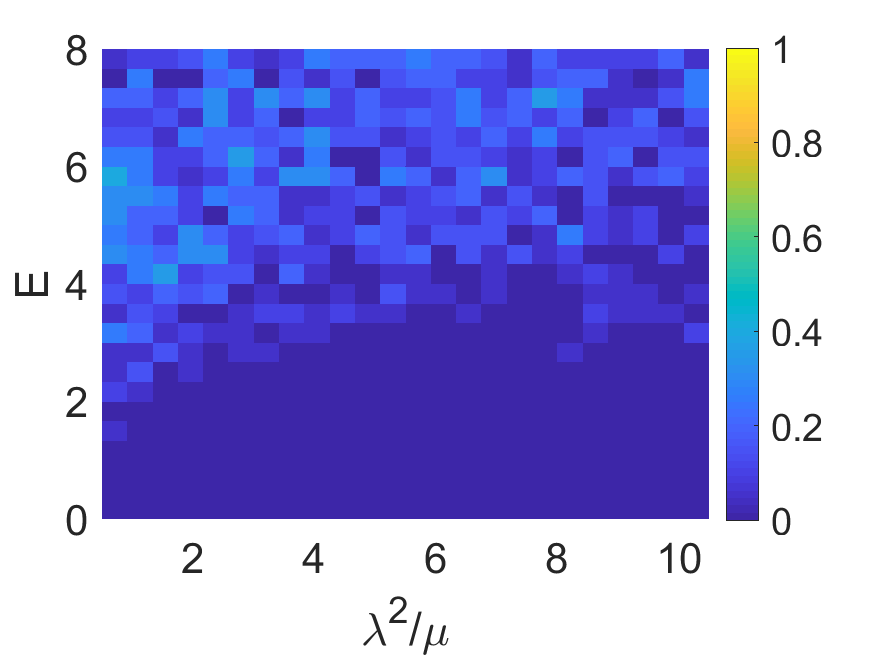}
  \end{subfigure}
  \begin{subfigure}[h]{0.3\linewidth}
    \adjincludegraphics[trim = {0 0 {0.04\width} {0.04\width}}, clip, width=\linewidth]{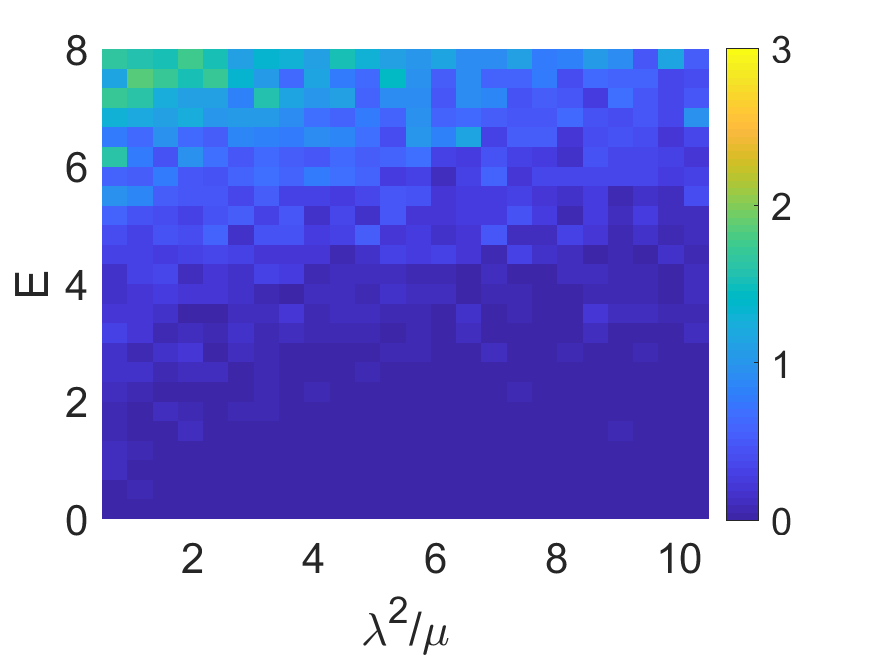}
  \end{subfigure}\\
  \begin{subfigure}[h]{0.3\linewidth}
    \adjincludegraphics[trim = {0 0 {0.04\width} {0.04\width}}, clip, width=\linewidth]{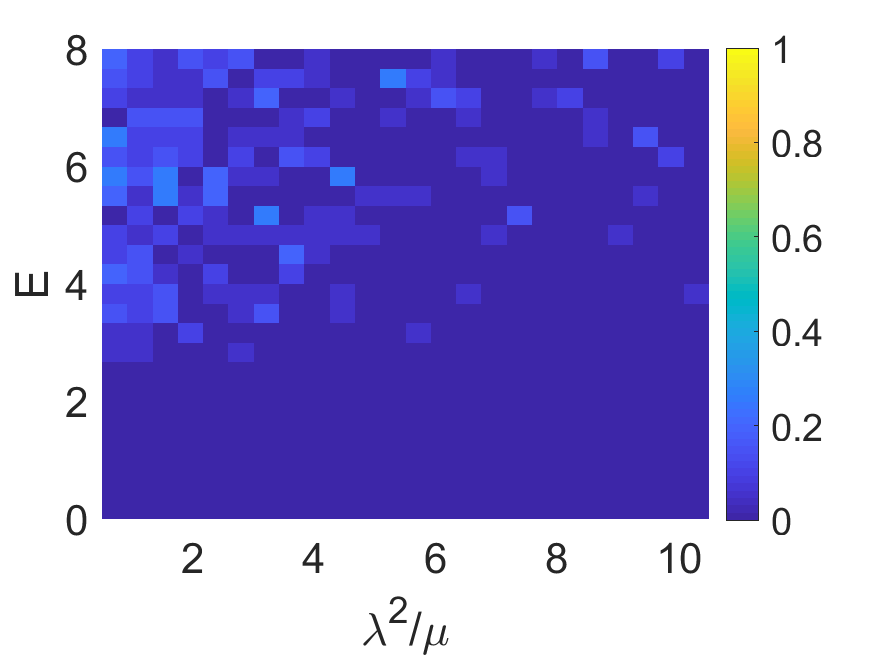}
  \end{subfigure}
  \begin{subfigure}[h]{0.3\linewidth}
    \adjincludegraphics[trim = {0 0 {0.04\width} {0.04\width}}, clip, width=\linewidth]{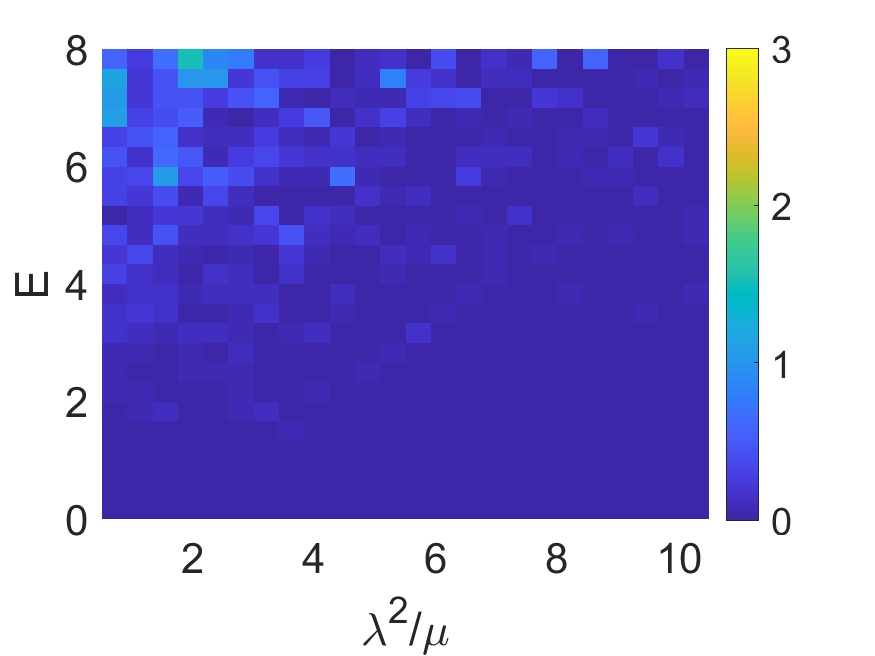}
  \end{subfigure}\\
  \begin{subfigure}[h]{0.3\linewidth}
    \adjincludegraphics[trim = {0 0 {0.04\width} {0.04\width}}, clip, width=\linewidth]{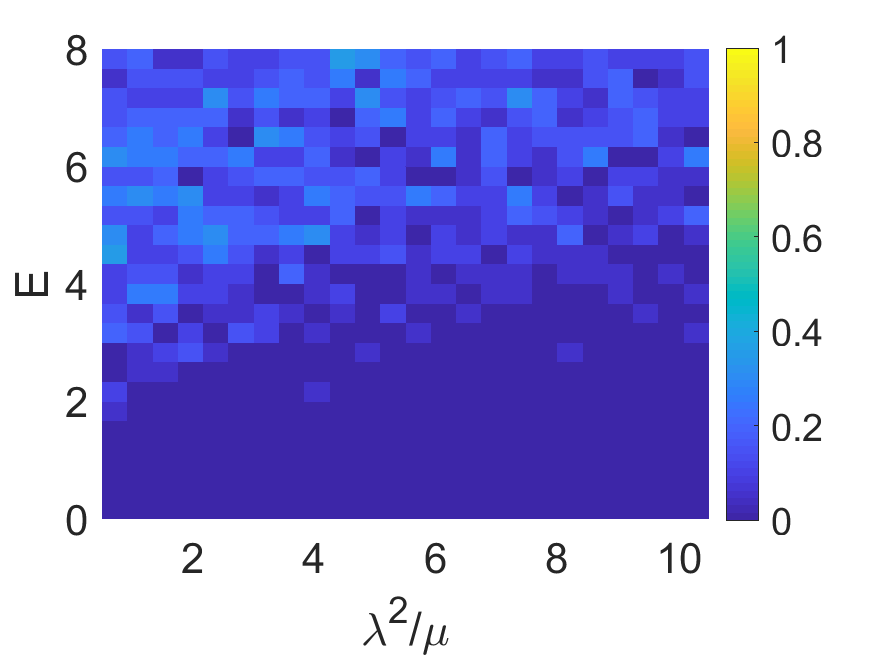}
  \end{subfigure}
  \begin{subfigure}[h]{0.3\linewidth}
    \adjincludegraphics[trim = {0 0 {0.04\width} {0.04\width}}, clip, width=\linewidth]{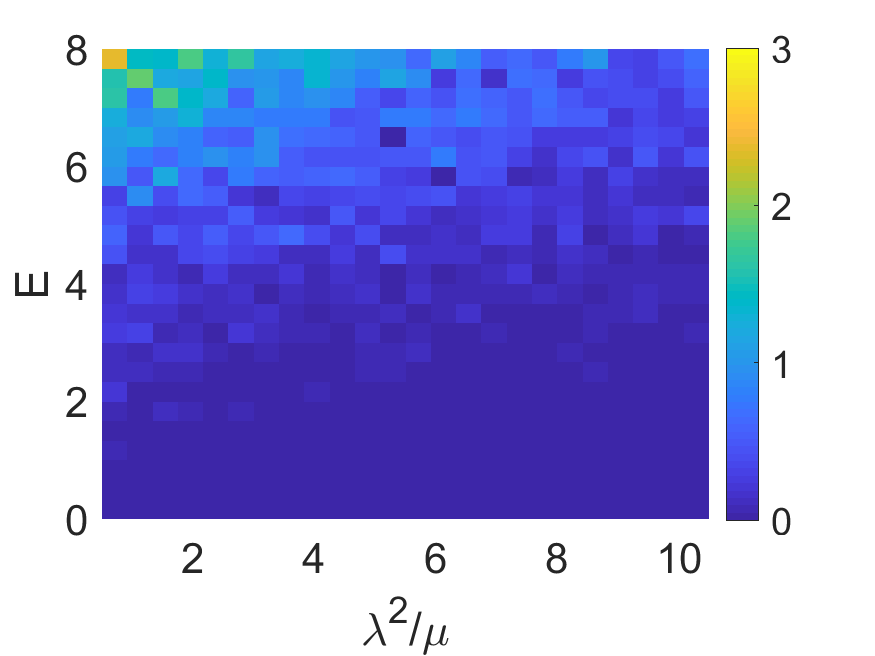}
  \end{subfigure}\\
  \begin{subfigure}[h]{0.3\linewidth}
    \adjincludegraphics[trim = {0 0 {0.04\width} {0.04\width}}, clip, width=\linewidth]{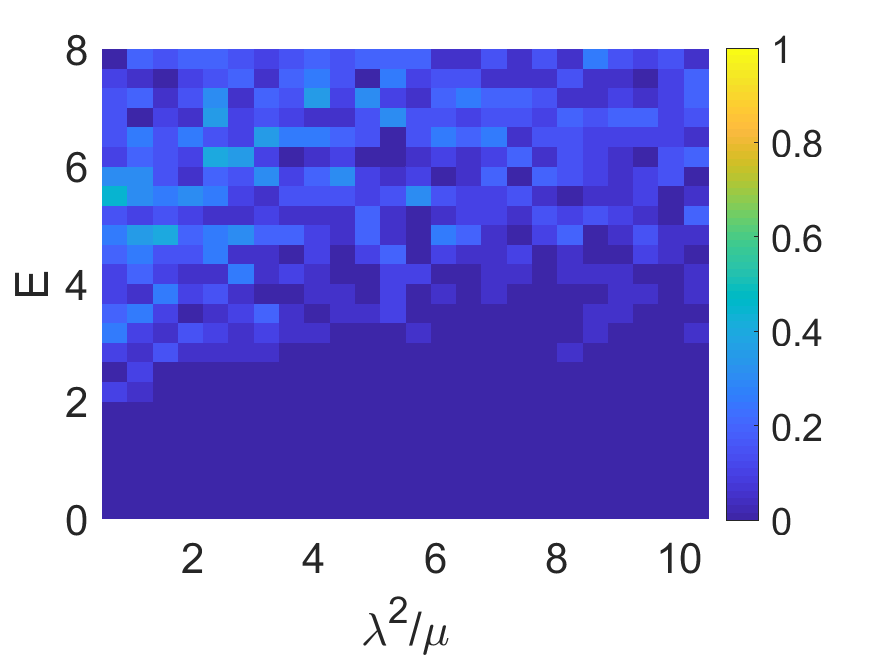}
  \end{subfigure}
  \begin{subfigure}[h]{0.3\linewidth}
    \adjincludegraphics[trim = {0 0 {0.04\width} {0.04\width}}, clip, width=\linewidth]{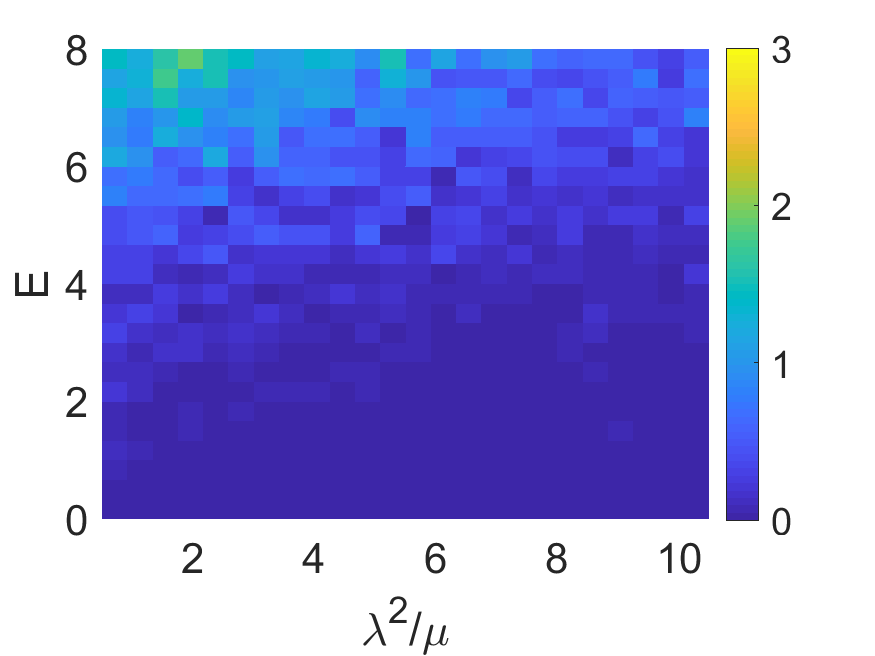}
  \end{subfigure}
  \caption{Results for $U=U_3$. Details are as in caption of Figure \ref{fig:u2}.}
  \label{fig:u3}
\end{figure}

\section{Proofs} \label{proofs}

\subsection{Notation and preliminaries}
Unless stated otherwise, $\partial_t$ is used to denote the partial derivative with respect to $t$ with $T_t$ fixed (whenever its operand depends on $T_t$), whereas $\frac{d}{dt}$ denotes the full derivative in $t$. In addition, $\nabla$ denotes the gradient in $\mathbb{R}^{2n+m}$ space and $d\zeta$ will be used for the Lebesgue measure on $\mathbb{R}^{2n+m}$. The notation $\mathds{1}_S$ will be used for the indicator function on the set $S$.\\
Recall the standard mollifier $\varphi:\mathbb{R}\rightarrow\mathbb{R}$ to be:
\begin{equation*}
\varphi(x):=
\begin{cases}
e^{\frac{1}{x^2-1}}\bigg(\int_{-1}^1e^{\frac{1}{y^2-1}}dy\bigg)^{-1} &\qquad\textrm{if }-1<x\leq1\\
0 &\qquad\textrm{otherwise},
\end{cases}
\end{equation*}
\begin{equation}\label{molli}
\varphi_k(x) := \frac{1}{k}\varphi\bigg(\frac{x}{k}\bigg)\qquad \forall k > 0.
\end{equation}
Let $(\Omega,\mathcal{F},\mathbb{P})$ be a complete probability space and $\mathcal{F}_t$, $t \in [0,\infty)$ be a normal
filtration\footnote{In other words, satisfying the usual conditions.} with $(W_t)_{t\geq0}$ a standard Wiener process on $\mathbb{R}^{m}$ with
respect to $\mathcal{F}_t$, $t\in [0,\infty)$ and $\zeta_0 = (X_0,Y_0,Z_0):\Omega\rightarrow\mathbb{R}^{2n+m}$ an $\mathcal{F}_0$-measurable map admitting Lebesgue density $m_0$ satisfying Assumption \ref{assumption4}.

The formal\footnote{See for instance Appendix B in \cite{MR1764365}. In the present paper the infinitesimal generators and their adjoints are considered as honest differential operators acting on smooth functions.} $L^2(\mu_{T_t})$-adjoint $L_t^*$ of $L_t$ is given by
\begin{equation}\label{muad}
L_{t}^* = -\ \!(y\cdot\nabla\!_x-\nabla\!_x U(x)\cdot\nabla\!_y)-(z^\top \lambda \nabla\!_y-y^\top\lambda^\top\nabla\!_z)-T_t^{-1}z^\top A \nabla\!_z +A:D_z^2.
\end{equation}
Let $\mathcal{C}^\infty_+ = \{f\in \mathcal{C}^\infty: f>0\}$. For $\Phi: \mathcal{C}^\infty_+ \rightarrow \mathcal{C}^\infty$ differentiable in the sense that for any $f\in \mathcal{C}^\infty_+$, $g\in \mathcal{C}^\infty$,
\begin{equation*}
(d\Phi(f).g)(\zeta)\quad:=\quad \lim_{s\rightarrow0}\frac{(\Phi(f+sg))(\zeta)-(\Phi(f))(\zeta)}{s}
\end{equation*}
exists for all $\zeta\in \mathbb{R}^{2n+m}$, the $\Gamma_\Phi$ operator for $L_t^{\epsilon *}$ is defined by
\begin{equation}\label{gamdef}
\Gamma_{L_t^{\epsilon *},\Phi}(h)\quad:=\quad \frac{1}{2}(L_t^{\epsilon *}\Phi(h)-d\Phi(h).(L_t^{\epsilon *} h)).
\end{equation}
It will be helpful to keep in mind that $L_t^{\epsilon *}$ does not satisfy the standard chain and product rules, due to the additional term from the second derivatives in $L_t^{\epsilon *}$; straightforward calculations give: 
\begin{align}
L_t^{\epsilon *}(\psi(f)) \ &=\ \psi'(f)L_t^{\epsilon *}f+\psi''(f)\nabla f\cdot(A^\epsilon\nabla f)\label{chain}\\
L_t^{\epsilon *}(fg) \ &=\ fL_t^{\epsilon *}(g)+gL_t^{\epsilon *}(f)+\nabla f \cdot(2A^\epsilon \nabla g)\label{product}
\end{align}
for all $f,g\in \mathcal{C}^\infty$ and $\psi\in \mathcal{C}^\infty$. Note $\nabla f\cdot(A^\epsilon\nabla f)$ and $\nabla f \cdot(2A^\epsilon \nabla g)$ are respectively the \textit{carr\'e du champ} and its symmetric bilinear operator via polarisation for $L_t^{\epsilon *}$.\\
In addition, for a scalar-valued $D_1$ and a vector-valued operator $D_2$ both acting on scalar-valued functions, denote the commutator bracket as follows:
\begin{equation}\label{commutat}
[D_1,D_2]h = (D_1(D_2 h)_1-(D_2 D_1 h)_1,\ \dots, D_1(D_2 h)_{d_{D_2}}-(D_2 D_1 h)_{d_{D_2}})
\end{equation}
for $h\in\mathcal{C}^\infty$, where $d_{D_2}\in\mathbb{N}$ is the number of elements in the output of $D_2$.\\

\textcolor{black}{
Let $\epsilon \geq 0$ and consider the perturbed system
\begin{subequations}\label{glep0}
\begin{align}
dX_t^\epsilon &= Y_t^\epsilon \, dt + \epsilon(-T_t^{-1}\nabla\!_x U(X_t^\epsilon) \, dt + dW_t^1),\\
dY_t^\epsilon &= -\nabla\!_x U(X_t^\epsilon) \, dt + \lambda^\top Z_t^\epsilon \, dt  +\epsilon(- T_t^{-1}Y_t^\epsilon \, dt + dW_t^2),\\
dZ_t^\epsilon &= -\lambda Y_t^\epsilon \, dt -T_t^{-1}A Z_t^\epsilon \, dt +\Sigma \, dW_t^3,
\end{align}
\end{subequations}
with $(X_0^\epsilon,Y_0^\epsilon,Z_0^\epsilon)=(X_0,Y_0,Z_0)$ restricted as in Assumption \ref{assumption4}, where $W_t^1,W_t^2,W_t^3$ are independent $n$-dimensional and $m$-dimensional Wiener processes. As before, the law and density of \eqref{glep0} will be denoted by $m_t^\epsilon$ along with 
\begin{equation*}
h_t^\epsilon = \frac{dm_t^\epsilon}{d\mu_{T_t}}.
\end{equation*}
Let the linear differential operators $S_t^x$, $S_t^y$ and their respective formal $L^2$-adjoints $S_t^{x\!\top}$ and $S_t^{y\!\top}$ be given by
\begin{align*}
S_t^x &= -T_t^{-1}\nabla\!_x U\cdot\nabla\!_x + \Delta_x,& S_t^y &= - T_t^{-1}y \cdot\nabla\!_y + \Delta_y,\\
S_t^{x\!\top} &= T_t^{-1} \nabla\!_x U\cdot\nabla\!_x + T_t^{-1}\Delta_x U  + \Delta_x,& S_t^{y\!\top} &= T_t^{-1} y\cdot\nabla\!_y + T_t^{-1}n + \Delta_y,
\end{align*}
so that the generator, denoted $L_t^\epsilon$, associated to \eqref{glep0} is given by the formal operator
\begin{equation*}
L_t^\epsilon = L_t + \epsilon(S_t^x + S_t^y).
\end{equation*}
Note that the formal $L^2(\mu_{T_t})$-adjoints of $S_t^x$ and $S_t^y$ coincide with $S_t^x$ and $S_t^y$, so that the formal $L^2(\mu_{T_t})$-adjoint of $L_t^\epsilon$, denoted $L_t^{\epsilon *}$, is 
\begin{equation*}
L_t^{\epsilon *} = L_t^* + \epsilon(S_t^x + S_t^y).
\end{equation*}
For any $\phi\in\mathcal{C}^\infty$ and $f:\mathbb{R}^{2n+m}\rightarrow\mathbb{R}$ smooth enough, 
\begin{equation}\label{Lchain1}
L_t^\epsilon(\phi(f)) = \phi'(f)L_t^\epsilon(f)+\phi''(f)\Gamma_t^\epsilon(f),
\end{equation}
where $\Gamma_t^\epsilon$ is the carr\'e du champ operator for $L_t^\epsilon$ given by 
\begin{equation}\label{Lchain2}
\Gamma_t^\epsilon(f)=\frac{1}{2}L_t^\epsilon(f^2)-fL_t^\epsilon(f)=\nabla f\cdot( A^\epsilon \nabla f ),
\end{equation}
$A^\epsilon\in\mathbb{R}^{(2n+m)\times (2n+m)}$ denotes the matrix with entries
\begin{equation*}
A_{ij}^\epsilon := \begin{cases}
\epsilon & \textrm{if } 1\leq i=j \leq 2n,\\
A_{i-2n,j-2n} & \textrm{if } 2n+1\leq i,j \leq 2n+m,\\
0 & \textrm{otherwise}
\end{cases}
\end{equation*}
and $A_{i,j}$ denotes the $(i,j)^\textrm{th}$ entry of $A$.
}

\subsection{Auxiliary results}
For the next result, the space of smooth functions that will be used is from \cite{MR736147}: let $\mathcal{C}^\infty_{b,c} = \mathcal{C}^\infty_{b,c}(\mathbb{R}_+\times\mathbb{R}^{2n+m})$ be the space of real-valued functions $f:\mathbb{R}_+\times\mathbb{R}^{2n+m}\rightarrow \mathbb{R}$ such that 
\begin{enumerate}
\item $f$ is measurable with respect to $\mathcal{B}(\mathbb{R}_+)\otimes\mathcal{B}(\mathbb{R}^{2n+m})$,
\item for all $t>0$, $f(t,\cdot)$ is smooth and $f$ is bounded on compact subsets of $\mathbb{R}_{>0}\times\mathbb{R}^{2n+m}$.
\end{enumerate}
\begin{prop}\label{refer2app}
Under Assumption \ref{assumption1}, \ref{assumption3} and \ref{assumption4}, for all $t>0$ and $\epsilon \geq 0$, the unique strong solution $(X_t^\epsilon,Y_t^\epsilon,Z_t^\epsilon)$ to \eqref{glep0} is well-defined and there exists some constant $\kappa>0$ such that
\begin{equation}\label{prop7}
\mathbb{E}\big[\abs*{X_t^\epsilon}^2+\abs*{Y_t^\epsilon}^2+\abs*{Z_t^\epsilon}^2\big] \leq e^{\kappa t}\mathbb{E}\big[\abs*{X_0}^2+\abs*{Y_0}^2+\abs*{Z_0}^2\big] <\infty.
\end{equation}
Furthermore, for all time $t>0$, the law of the $(X_t^\epsilon,Y_t^\epsilon,Z_t^\epsilon)$
\begin{itemize}
\item admits an almost-everywhere finite strictly positive density, also denoted $m_t^\epsilon$, w.r.t. the Lebesgue measure on $\mathbb{R}^{2n+m}$,
\item is the unique integrable distributional solution to 
\begin{equation}\label{ueq0}
\begin{cases}
\partial_t m_t^\epsilon = (L_t^\top + \epsilon(S_t^{x\!\top}+S_t^{y\!\top})) m_t^\epsilon  &\\
m_0^\epsilon = m_0, &
\end{cases}
\end{equation}
where $L_t^\top$ is the formal $L^2$-adjoint of $L_t$.
\end{itemize}
Finally when $\epsilon>0$, $m_\bullet$ and its partial derivative in time belongs in $\mathcal{C}^\infty_{b,c}$.
\end{prop}
For the notion of integrable distributional solutions, see p.338 in \cite{MR3443169}.

\begin{proof}\color{black}
Existence and uniqueness of an almost surely continuous $\mathcal{F}_t$-adapted processes follows by conditions \eqref{secbdd} and \eqref{q3}  using Theorem 3.1.1 in \cite{MR2329435}; in addition, \eqref{prop7} holds by the same theorem.
For the claim that the law admits a density, we will apply Theorem 1 in \cite{MR3648039} for the case of an arbitrary deterministic starting point. First, condition (H1) in the same article is verified. Take the sets `$K_n$' to be 
\begin{equation*}
K_p = \prod_{i=1}^{2n+m} [-p,p]
\end{equation*}
for all $p\in\mathbb{N}$. The unique solution to \eqref{glep0} with a deterministic starting point $(X_0,Y_0,Z_0) = (x_0,y_0,z_0)\in\mathbb{R}^{2n+m}$ satisfies the same bound \eqref{prop7} as before when initialising from $m_0$. Moreover, for the random sets
\begin{equation*}
\Xi_p = \{s>0:(X_u^\epsilon,X_u^\epsilon,X_u^\epsilon)\in K_p, 0\leq u\leq s\},
\end{equation*}
for $p\in\mathbb{N}$, the solution $(\hat{X}_t^\epsilon,\hat{Y}_t^\epsilon,\hat{Z}_t^\epsilon)$ to the stopped stochastic differential equation
\begin{subequations}\label{glep1}
\begin{align}
d\hat{X}_t^{\epsilon,p} &= \mathds{1}_{\Xi_p}(t) (\hat{Y}_t^{\epsilon,p} \, dt + \epsilon(-T_t^{-1}\nabla\!_x U(\hat{X}_t^{\epsilon,p}) \, dt + dW_t^1)),\\
d\hat{Y}_t^{\epsilon,p} &= \mathds{1}_{\Xi_p}(t) (-\nabla\!_x U(\hat{X}_t^{\epsilon,p}) \, dt + \lambda^\top \hat{Z}_t^{\epsilon,p} \, dt  +\epsilon(- T_t^{-1}\hat{Y}_t^{\epsilon,p} \, dt + dW_t^2)),\\
d\hat{Z}_t^{\epsilon,p} &= \mathds{1}_{\Xi_p}(t) (-\lambda \hat{Y}_t^{\epsilon,p} \, dt -T_t^{-1}A \hat{Z}_t^{\epsilon,p} \, dt +\Sigma \, dW_t^3),
\end{align}
\end{subequations}
is well-defined by the same Theorem 3.1.1 in \cite{MR2329435} and the corresponding bound
\begin{equation*}
\mathbb{E}\big[\abs{\hat{X}_t^{\epsilon,p}}^2+\abs{\hat{Y}_t^{\epsilon,p}}^2+\abs{\hat{Z}_t^{\epsilon,p}}^2\big] \leq e^{\kappa t}\big(\abs*{x_0}^2+\abs*{y_0}^2+\abs*{z_0}^2\big) <\infty
\end{equation*}
holds. Identifying $(\hat{X}_t^{\epsilon,p},\hat{Y}_t^{\epsilon,p},\hat{Z}_t^{\epsilon,p})=(X_{t\wedge \sup\Xi_p}^\epsilon,Y_{t\wedge \sup\Xi_p}^\epsilon,Z_{t\wedge \sup\Xi_p}^\epsilon)$ a.s. yields that\footnote{Alternatively Corollary 1.2 of Section 5 in \cite{MR0494490} can be used.} for any $\tau>0$
\begin{align*}
\mathbb{P}(\inf\{ t\geq 0 : (X_t^\epsilon,Y_t^\epsilon,Z_t^\epsilon)\notin K_p\} \leq \tau) &\leq \frac{1}{p^2} \mathbb{E}\big[\abs{X_{\tau\wedge \sup\Xi_p}^\epsilon}^2+\abs{Y_{\tau\wedge \sup\Xi_p}^\epsilon}^2+\abs{Z_{\tau\wedge \sup\Xi_p}^\epsilon}^2\big]\\
&\leq \frac{e^{\kappa \tau}}{p^2}\big(\abs*{x_0}^2+\abs*{y_0}^2+\abs*{z_0}^2\big)
\end{align*}
and in particular that for any $\tau>0$,
\begin{equation}\label{setch}
\mathbb{P}(\inf\{ t\geq 0 : (X_t^\epsilon,Y_t^\epsilon,Z_t^\epsilon)\notin K_p\}\leq \tau) \rightarrow 0 \quad\textrm{as } p\rightarrow \infty.
\end{equation}
Suppose for contradiction that with nonzero probability, the increasing-in-$p$ random variable $\inf\{t\geq 0: (X_t^\epsilon,Y_t^\epsilon,Z_t^\epsilon)\notin K_p\}$ converges to a real value as $p\rightarrow\infty$. Then there exists a time $\hat{\tau} >0$ such that with nonzero probability,
\begin{equation*}
\inf\{t\geq 0: (X_t^\epsilon,Y_t^\epsilon,Z_t^\epsilon)\notin K_p\}\leq \hat{\tau}\quad \forall p\in\mathbb{N},
\end{equation*}
which contradicts \eqref{setch}. Therefore condition (H1) in \cite{MR3648039} holds for \eqref{glep0}. Condition (H2) in the same article holds due the $K_p$ being compact and the smoothness assumption on $U$. It can be readily checked that the local weak H\"ormander condition (LWH) in \cite{MR3648039} also holds at any $(t,y_0)$ for any $r\in (0,t)$ and $R>0$. Therefore by Theorem 1 in \cite{MR3648039}, due to our Assumptions \ref{assumption1} and \ref{assumption3}, the solution to \eqref{glep0} with a deterministic starting point $\zeta_0\in\mathbb{R}^{2n+m}$ admits a smooth density $p_t^{\zeta_0}\in\mathcal{C}^\infty(\mathbb{R}^{2n+m})$ for all $t >0$. Moreover by Theorem 2 in \cite{MR3648039}, for any fixed $\zeta\in\mathbb{R}^{2n+m}$, $\mathbb{R}^{2n+m}\ni\zeta_0\mapsto p_t^{\zeta_0}(\zeta)$ is lower semi continuous and hence measurable, so that the $\mathbb{R}\cup\{\pm\infty\}$-valued function on $\mathbb{R}^{2n+m}$,
\begin{equation}\label{funcexp}
\int_{\mathbb{R}^{2n+m}}p_t^{\zeta_0}m_0(d\zeta_0),
\end{equation}
is integrable by Fubini's theorem and so is almost everywhere $\mathbb{R}$-valued on $\mathbb{R}^{2n+m}$. By It\^{o}'s rule, \eqref{funcexp} solves \eqref{ueq0} in the distributional sense. In addition, \eqref{ueq0} is the unique integrable solution by Theorem 9.6.3 in \cite{MR3443169}, which requires for any $T>0$ that there exists $V\in C^2(\mathbb{R}^{2n+m})$ such that 
\begin{enumerate}
\item$V(x)\rightarrow \infty$ as $\abs*{x}\rightarrow \infty$ and 
\item for some constant $C_V>0$ and all $(x,t)\in\mathbb{R}^{2n+m}\times (0,T)$, it holds that $L_t^\epsilon V \geq -C_V V$ and $\abs*{\nabla V} \leq C_V V$.
\end{enumerate}
Setting $V(x,y,z) = 1+ U(x)-U_m + \frac{\abs{y}^2}{2} + \frac{\abs{z}^2}{2}$ and calculating
\begin{equation}\label{calc1}
L_t^\epsilon\bigg(U(x)+\frac{\abs{y}^2}{2}+\frac{\abs{z}^2}{2}\bigg) =  \epsilon\bigg( -\frac{1}{T_t} \abs*{\nabla\!_x U}^2 + \Delta_x U -\frac{1}{T_t} \abs{y}^2 + n \bigg) -\frac{1}{T_t}z^\top A z+ \textrm{Tr} A
\end{equation}
it is clear from assumptions \eqref{secbdd}, \eqref{q3} and either \eqref{q1} or \eqref{q1menz} on $U$ that these conditions are satisfied since $T$ is finite; therefore there is a unique integrable solution to \eqref{ueq0} in the sense of p.338 in \cite{MR3443169}. The expression in \eqref{funcexp} is thus the density for the law of the solution to \eqref{glep0} with initial law $m_0$ at time $t$.

For $\epsilon > 0$, the time-depending law of $(X_t^\epsilon,Y_t^\epsilon,Z_t^\epsilon)$ and its partial derivative with respect to time belongs in $\mathcal{C}_{b,c}^\infty$ by Theorem 1.1 in \cite{MR736147} because \eqref{funcexp} satisfies \eqref{ueq0}.\\ 
For positivity of the density where $\epsilon = 0$, the steps in Lemma 3.4 of \cite{stuart52} can be followed; the associated control problem has the expression
\begin{align}\label{control}
\frac{d}{dt}\!\begin{pmatrix}
Q_t\\P_t\\V_t
\end{pmatrix} &= 
\begin{pmatrix}
P_t\\-\nabla U(Q_t) + \lambda^\top V_t\\ -\lambda P_t - T_t A V_t + \Sigma \frac{d\tilde{U}}{dt}
\end{pmatrix}.
\end{align}
It suffices to show that given any $S>0$ and any pair $(Q_0,P_0,V_0)\in\mathbb{R}^{2n+m}$ and $(Q^*,P^*,V^*)\in\mathbb{R}^{2n+m}$, there exists a control $\tilde{U}:[0,\infty)\rightarrow\mathbb{R}^m$ such that the solution $(Q_t,P_t,V_t)$ to \eqref{control} starting at $(Q_0,P_0,V_0)$ satisfies $(Q_S,P_S,V_S)=(Q^*,P^*,V^*)$. Fix $S>0,\epsilon>0, (Q_0,P_0,V_0)\in\mathbb{R}^{2n+m}$, $(Q^*,P^*,V^*)\in\mathbb{R}^{2n+m}$. Using the mollifier \eqref{molli}, let
\begin{equation*}
\nu:=\varphi_{\frac{1}{2}}* \mathds{1}_{(-\infty,\frac{1}{2}]},
\end{equation*}
where $*$ denotes convolution. Define a smooth function $\hat{Q}_\cdot:[0,S]\rightarrow\mathbb{R}^n$ by
\begin{equation}\label{qhat}
\hat{Q}_t = \bigg((-\nabla U(Q_0) + \lambda^\top V_0)\frac{t^2}{2} + P_0 t + Q_0\bigg)\nu\bigg(\frac{t}{S}\bigg) + \bigg((-\nabla U(Q^*) + \lambda^\top V^*)\frac{t^2}{2}+P^* t+Q^*\bigg)\nu\bigg(1-\frac{t}{S}\bigg)
\end{equation}
which satisfies
\begin{align*}
\hat{Q}_0 &= Q_0, & \hat{Q}_S &= Q^*,\\
\frac{d\hat{Q}_t}{dt}(0) &= P_0, & \frac{d\hat{Q}_t}{dt}(S) &= P^*.
\end{align*}
Define $\hat{P}_\cdot:[0,S]\rightarrow\mathbb{R}^n$ through
\begin{equation}\label{phat}
\hat{P}_t = \frac{d\hat{Q}_t}{dt}.
\end{equation}
For $\hat{V}_\cdot:[0,S]\rightarrow\mathbb{R}^m$, 
$\hat{V}_\cdot:[0,S]\rightarrow\mathbb{R}^m$ is defined with
\begin{align}
\hat{V}_t &= \lambda(\lambda^\top \lambda)^{-1}\bigg(\nabla U(\hat{Q}_t) + \partial_t^2 \bigg[\bigg(-\nabla U(Q_0)\frac{t^2}{2} + P_0 t + Q_0\bigg)\nu\bigg(\frac{t}{S}\bigg) \nonumber\\
&\quad+ \bigg(-\nabla U(Q^*)\frac{t^2}{2} + P^* t + Q^*\bigg)\nu\bigg(1-\frac{t}{S}\bigg)\bigg]\bigg) + \partial_t^2 \bigg[ V_0 \frac{t^2}{2}\nu\bigg(\frac{t}{S}\bigg) + V^* \frac{t^2}{2} \nu\bigg(1-\frac{t}{S}\bigg) \bigg]\label{vhat}
\end{align}
where $(\lambda^\top \lambda)^{-1}$ exists by $\lambda$ having rank $n$.
Note that $\hat{V}_t$ satisfies $\hat{V}_0=V_0$ and $\hat{V}_S = V^*$. 
Let the smooth function $\tilde{U}:[0,\infty)\rightarrow \mathbb{R}^m$ be given by
\begin{equation}\label{du}
\frac{d\tilde{U}}{dt} = \Sigma^{-1}\bigg(\frac{d\hat{V}_t}{dt}+\lambda\hat{P}_t+T_tA\hat{V}_t \bigg),\quad \tilde{U}(0)=0.
\end{equation}
For this $\tilde{U}$, the solution to \eqref{control} with initial condition $(Q_0,P_0,V_0)$ is $(\hat{Q}_t,\hat{P}_t,\hat{V}_t)$ by construction; its uniqueness is guaranteed by considering the system satisfied by the difference between two supposedly different solutions $(Q_t^1,P_t^1,V_t^1)$ and $(Q_t^2,P_t^2,V_t^2)$
\begin{align*}
\frac{d}{dt}\!\begin{pmatrix}
Q_t^1-Q_t^2\\P_t^1-P_t^2\\V_t^1-V_t^2
\end{pmatrix} &= 
\begin{pmatrix}
P_t^1-P_t^2\\-\nabla U(Q_t^1)-\nabla U(Q_t^2) + \lambda^\top (V_t^1-V_t^2)\\ -\lambda (P_t^1-P_t^2) - T_t A (V_t^1-V_t^2)
\end{pmatrix}
\end{align*}
and the time derivative of $\abs{Q_t^1-Q_t^2}^2 + \abs{P_t^1-P_t^2}^2 + \abs{V_t^1-V_t^2}^2$, using \eqref{secbdd} and the mean value theorem on $\abs{\nabla U (Q_t^1)-\nabla U(Q_t^2)}^2$.\\
With non-zero probability, the path of Brownian motion stays within an $\epsilon$-neighbourhood of any continuously differentiable path, in particular of $\tilde{U}$. Positivity of $m_t$ follows by the support theorem of Stroock and Varadhan (Theorem 5.2 in \cite{On2}). The above construction for the $\epsilon >0$ case follows with a simple modification; equation \eqref{control} becomes
\begin{align}\label{control2}
\frac{d}{dt}\!\begin{pmatrix}
Q_t\\P_t\\V_t
\end{pmatrix} &= 
\begin{pmatrix}
P_t -\epsilon \nabla U(Q_t) + \epsilon\frac{d\tilde{U}_1}{dt}\\-\nabla U(Q_t) + \lambda^\top V_t - \epsilon P_t + \epsilon\frac{d\tilde{U}_2}{dt}\\ -\lambda P_t - T_t A V_t + \Sigma \frac{d\tilde{U}}{dt}
\end{pmatrix},
\end{align}
so setting $\frac{d\tilde{U}_1}{dt} =  \nabla U(\hat{Q}_t)$ and $\frac{d\tilde{U}_2}{dt} = \hat{P}_t$ together with \eqref{du} gives the solution \eqref{qhat}, \eqref{phat} and \eqref{vhat} to equation \eqref{control2} and concludes the proof.
\end{proof}

\begin{remark}\label{tanirem}
For smoothness of the density, the results in \cite{MR800974} can also be considered, but there the assumptions are slightly mismatched. Firstly, the statement assumes boundedness of $\partial^\alpha V$ for any multiindex $\alpha$ where $V$ would in the case here be any of the coefficients appearing in \eqref{gl}, which fails for $\abs*{\alpha}=0$. Secondly, in case of (A.1), condition (i) fails and in case of (A.2), condition (i) fails due to $V_0$. Both of these pedantries seem possibly unneeded in the proofs but we avoid this in favour of the more recent work \cite{MR3648039}.
\end{remark}

The results below up to Proposition \ref{dissipation} are directed towards showing dissipation of a distorted entropy as required in the proof of Theorem \ref{convergence}. 

\subsection{Lyapunov function}

\begin{lemma}\label{Lya} 
Under Assumption \ref{assumption1}, \ref{assumption3} and \ref{assumption4}, there exist constants $a,b,c,d,\delta>0$ independent of $\epsilon$ such that $R:\mathbb{R}^{2n+m+1}\rightarrow\mathbb{R}$ defined as
\begin{equation}\label{R}
R(x,y,z,T_t)\quad := \quad U(x)+\frac{\abs{y}^2}{2}+\frac{\abs{z}^2}{2}+\delta T_t \bigg( y^\top \lambda^{-1} z + \frac{1}{2} x\cdot y\bigg)
\end{equation}
satisfies
\begin{equation}\label{Rbd}
a(\abs{x}^2+\abs{y}^2+\abs{z}^2)-d\quad\leq \quad R(x,y,z,T_t)\quad \leq\quad b (\abs{x}^2+\abs{y}^2+\abs{z}^2) + d,
\end{equation}
and there exists $0<\epsilon'\leq 1$ for which  $\epsilon \leq \epsilon'$ implies
\begin{align}
L_t^\epsilon R &\leq  -cT_t R+\frac{d}{T_t}.\label{Lyaprop}
\end{align}
\end{lemma}
\begin{proof}
By the quadratic assumption (\ref{q1menz}) on $U$ and boundedness Assumption \ref{assumption3} on $T_t$, 
it is clear that there exists $\hat{\delta} > 0$ such that the first statement \eqref{Rbd} holds with $d=\max(\abs*{U_{\!m}},\abs*{U_{\!M}})$ for all $\delta\in(0,\hat{\delta}]$. 
For the second statement in \eqref{Lyaprop}, fix $\delta > 0$ to be
\begin{equation}\label{deltbound}
\delta \leq \min\bigg(\hat{\delta}, 1, \frac{4r_1^2}{(r_2+1)\sup_{s\geq 0} T_s}, 2\Big(\sup_{s\geq 0} T_s\Big)^{-1}, 
 \frac{A_c}{2}\bigg[\bigg(\frac{\abs*{\lambda}^2}{2r_1}+1+\frac{r_2}{r_1}\abs*{\lambda^{-1}}^2\bigg)\Big(\sup_{s\geq 0} T_s\Big)^2+2(\abs*{A}^2+1)\abs*{\lambda^{-1}}^2\bigg]^{-1}\bigg),
\end{equation}
where $\abs{\cdot}$ is the operator norm here and $A_c>0$ is the coercivity constant of the positive definite matrix $A$. 
Consider each of the terms of $L_t^\epsilon(R)$ seperately.
\begin{align}
L_t^\epsilon\bigg(U(x)+\frac{\abs{y}^2}{2}+\frac{\abs{z}^2}{2}\bigg) &=  \epsilon\bigg( -\frac{1}{T_t} \abs*{\nabla\!_x U}^2 + \Delta_x U -\frac{1}{T_t} \abs{y}^2 + n \bigg) -\frac{1}{T_t}z^\top A z+ \textrm{Tr} A. \label{LtR1}\\
&\leq \epsilon\bigg( -\frac{1}{T_t} ( r_1^2\abs*{x}^2 - 2 r_1 U_g ) + n|D_x^2 U|_\infty  -\frac{1}{T_t} \abs{y}^2 + n \bigg) -\frac{1}{T_t}z^\top A z+ \textrm{Tr} A,\label{1}
\end{align}
where the last inequality follows from \eqref{q2} and $\nabla\!_x U\cdot x \leq \frac{1}{2 r_1}\abs*{\nabla\!_x U}^2 + \frac{r_1}{2}\abs*{x}^2$. Using the quadratic bound (\ref{q3}) on $\nabla\!_x U$, we get
\begin{align}
L_t^\epsilon(y^\top \lambda^{-1} z) &= -\epsilon T_t^{-1}y^\top\lambda^{-1}z -\nabla\!_x U \lambda^{-1} z+\abs{z}^2-\abs{y}^2-T_t^{-1} z^\top A(\lambda^{-1})^\top y \label{LtR2}\\
&\leq \frac{\abs{y}^2}{4} + 2T_t^{-2}(\abs*{A}^2+\epsilon^2)\abs*{\lambda^{-1}}^2\abs{z}^2 + \frac{r_1}{4r_2}\abs{\nabla\!_x U}^2+\frac{r_2}{r_1}\abs*{\lambda^{-1}}^2\abs{z}^2+\abs{z}^2-\abs{y}^2 \nonumber\\
& \leq \frac{r_1}{4} \abs{x}^2+\frac{r_1}{4r_2} U_{\!g}-\frac{3}{4}\abs{y}^2 +\bigg( 1+\frac{r_2}{r_1}\abs*{\lambda^{-1}}^2+ 2T_t^{-2}(\abs*{A}^2+1)\abs*{\lambda^{-1}}^2\bigg)\abs{z}^2. \label{2}
\end{align}
Then using also \eqref{q2} for $\nabla\!_x U\cdot x$, we get
\begin{align}
L_t^\epsilon(x\cdot y) &= -\epsilon T_t^{-1}y\cdot\nabla\!_x U -\epsilon T_t^{-1} x\cdot y +  \abs{y}^2-\nabla\!_x U\cdot x+ z^\top\lambda x\label{LtR3}\\
&\leq \epsilon T_t^{-1}\bigg(\bigg(\frac{r_2}{2}+\frac{1}{2}\bigg)\abs*{x}^2 + \abs*{y}^2 + \frac{U_g}{2}\bigg) + \abs{y}^2-r_1\abs{x}^2+U_{\!g}+\frac{\abs*{\lambda^\top}^2}{r_1}\abs{z}^2+\frac{r_1}{4}\abs{x}^2. \label{3}
\end{align}
Combining \eqref{R}, \eqref{1}, \eqref{2}, \eqref{3} and taking $\epsilon \leq 1$,
\begin{align}
L_t^\epsilon (R(x,y,z,T_t)) &= L_t^\epsilon\bigg(U(x)+\frac{\abs{y}^2}{2}+\frac{\abs{z}^2}{2}\bigg)+\delta T_t L_t(y^\top \lambda^{-1} z) +  \frac{\delta T_t}{2} L_t(x\cdot y)\label{LtR4}\\
&\leq - \delta T_t \frac{r_1}{8}\abs{x}^2-\delta T_t\frac{1}{4}\abs{y}^2-\frac{1}{T_t}z^\top A z + C +\frac{ 2\epsilon r_1U_g}{T_t} \nonumber\\
&\quad +\delta T_t\bigg[\ \frac{\abs*{\lambda^\top}^2}{2 r_1}+\bigg( 1+\frac{r_2}{ r_1}\abs*{\lambda^{-1}}^2+ 2T_t^{-2}(\abs*{A}^2+1)\abs*{\lambda^{-1}}^2 \bigg)\bigg]\abs{z}^2. \nonumber\\
&\quad +\epsilon T_t^{-1}\bigg[\bigg(\frac{\delta T_t}{4}(r_2+1) - r_1^2\bigg)\abs*{x}^2 + \bigg(\frac{\delta T_t}{2} -1\bigg)\abs*{y}^2 \bigg],
\end{align}
where $0<C= n(|D_x^2 U|_\infty +1) + \textrm{Tr}A + \delta \frac{r_1U_g}{4r_2} \sup_{s\geq 0} T_s +\delta U_g (\frac{1}{4} + \frac{\sup_{s\geq 0} T_s}{2})$. Therefore for $\delta$ satisfying the bound \eqref{deltbound}, the first square bracket term satisfies
\begin{equation*}
\delta T_t\bigg[\ \frac{\abs*{\lambda^\top}^2}{2 r_1}+\bigg( 1+\frac{r_2}{ r_1}\abs*{\lambda^{-1}}^2+ 2T_t^{-2}(\abs*{A}^2+1)\abs*{\lambda^{-1}}^2 \bigg)\bigg] \leq \frac{1}{2}\Big(\sup_{s\geq 0} T_s\Big)^{-1}A_c\abs*{z}^2
\end{equation*}
and the second square bracket term is negative, where the assumption that $T_t$ is bounded above for all time has been used. Rearranging,
\begin{align*}
L_t^\epsilon(R(x,y,z,T_t)) &\leq - \delta T_t\frac{r_1}{8}\abs{x}^2-\delta T_t\frac{1}{4}\abs{y}^2 - \frac{A_c}{2 \sup_s T_s}\abs{z}^2+ C + \frac{2\epsilon r_1 U_g }{T_t}\\
&\leq-cT_tR+C' T_t^{-1},
\end{align*}
where $c>0$ is small enough, $C'>0$ is large enough and the right inequality of (\ref{Rbd}) has been used. The result follows using $d=\max ( C',\abs*{U_{\!m}},\abs*{U_{\!M}})$.
\end{proof}

\begin{lemma}\label{expRbd}\color{black}
Under Assumption \ref{assumption1}, \ref{assumption3}, \ref{assumption4} and for $0\leq \epsilon \leq \epsilon'$, the solution $(X_t^\epsilon,Y_t^\epsilon,Z_t^\epsilon)$ to \eqref{glep0} is such that $\frac{\mathbb{E}[R(X_t^\epsilon,Y_t^\epsilon,Z_t^\epsilon,T_t)]}{(\ln(e+t))^2}$ is bounded uniformly in time $t$ and in $\epsilon$.
\end{lemma}

\begin{proof}\color{black}
It is equivalent to prove the result for $R+d>0$ in place of $R$. Let $R_t := R(X_t^\epsilon,Y_t^\epsilon,Z_t^\epsilon,T_t).$
Consider the following terms separately for $t > t_0$,
\begin{equation}\label{splitep}
\frac{d}{dt}\mathbb{E}[(R_t+d)] =\partial_t \mathbb{E}[(R_t+d)] + T_t'\partial_{T_t} \mathbb{E}[(R_t+d)].
\end{equation}
Firstly, by \eqref{R}, the left hand bound in \eqref{Rbd} and the Assumption \ref{assumption3},
\begin{align}
T_t'\partial_{T_t} \mathbb{E}[R_t+d]&= T_t'\mathbb{E}\bigg[\delta\bigg((Y_t^\epsilon)^\top \lambda^{-1} Z_t +\frac{1}{2}X_t^\epsilon\cdot Y_t^\epsilon\bigg)\bigg]\nonumber\\
&\leq \abs{T_t'} \mathbb{E}\bigg[\delta\abs*{(Y_t^\epsilon)^\top \lambda^{-1} Z_t^\epsilon +\frac{1}{2}X_t^\epsilon\cdot Y_t^\epsilon}\bigg]\nonumber\\
&\leq\frac{B}{t}\mathbb{E}[R_t+d]\label{partialt}
\end{align}
for a constant $B\geq 0$ independent of $\epsilon$. The exchange in the order of differentiation and integration is justified by the right hand bound of \eqref{Rbd} together with the mean value theorem and the fact that for any $T>0$ the inequality 
\begin{align}
\mathbb{E}\bigg[\sup_{0\leq t\leq T}\abs*{X_t^\epsilon}^2\bigg] = \mathbb{E}\bigg[\bigg(\sup_{0\leq t\leq T}\abs*{X_t^\epsilon}\bigg)^2\ \bigg] &= \mathbb{E}\bigg[\bigg(\sup_{0\leq t\leq T}\abs*{X_0 + \int_0^t (Y_s^\epsilon + \epsilon(-T_s^{-1}\nabla\!_x U(X_s^\epsilon))) \, ds + \epsilon\int_0^t dW_s^1}\bigg)^2\ \bigg]\nonumber\\
&\leq K_T^1\mathbb{E}\int_0^T (1+\abs*{X_s^\epsilon}^2 + \abs*{Y_s^\epsilon}^2)ds < \infty\label{supbd}
\end{align}
holds for some constant $K_T^1 > 0$ depending on $T$ but independent of $\epsilon$ (using $\epsilon\leq\epsilon'$), where \eqref{q3}, \eqref{prop7}, Jensen's inequality and Doob's maximal inequality have been used. Similar expressions for $\mathbb{E}[\sup_{0\leq t\leq T}\abs*{Y_t^\epsilon}^2]$ and $\mathbb{E}[\sup_{0\leq t\leq T}\abs*{Z_t^\epsilon}^2]$ hold.\\
Considering together with the other term in \eqref{splitep}, by It\^{o}'s rule and for $t_0<s<t$,
\begin{equation*}
\mathbb{E}R_t - \mathbb{E}R_s = \mathbb{E}\int_s^t (T_u'\partial_{T_t}R + L_u^\epsilon R)(X_u^\epsilon,Y_u^\epsilon, Z_u^\epsilon, T_u)du = \int_s^t \mathbb{E} (T_u'\partial_{T_t}R + L_u^\epsilon R)(X_u^\epsilon,Y_u^\epsilon, Z_u^\epsilon, T_u)du,
\end{equation*}
where the last equality follows by Fubini, \eqref{partialt} and \eqref{LtR4} together with \eqref{LtR1}, \eqref{LtR2}, \eqref{LtR3}, \eqref{secbdd}, \eqref{q3}, \eqref{supbd}. Property \eqref{Lyaprop} from Lemma \ref{Lya} and \eqref{partialt} give
\begin{align}
\mathbb{E}[R_t+d] - \mathbb{E}[R_s+d] &\leq\int_s^t\bigg(\frac{B}{u}\mathbb{E}[R_u+d] + \mathbb{E}[-cT_u R_u+dT_u^{-1}]\bigg)du\nonumber\\
&\leq \int_s^t \bigg(\bigg(\frac{B}{u}-cT_u\bigg)\mathbb{E}[R_u+d]+B'T_u^{-1}\bigg)du\label{partialT}
\end{align}
for a constant $B'\geq 0$ independent of $\epsilon$. In order to obtain a differential inequality for all $t$ as opposed to almost all $t$ from Lebesgue differentiation theorem, the above expression can be divided by $t-s$, mollified in time with \eqref{molli} for $0 < k < 1$ and have $s\rightarrow t$ taken as follows. For $t>t_0+1$,
\begin{align*}
&\lim_{\hat{\epsilon}\rightarrow 0} \frac{1}{2\hat{\epsilon}} \int_{t-1}^{t+1} \varphi_k(t-u) (\mathbb{E}[R_{u+\hat{\epsilon}}+d] - \mathbb{E}[R_{u-\hat{\epsilon}}+d])du\\
\qquad  &\leq \lim_{\hat{\epsilon}\rightarrow 0} \frac{1}{2\hat{\epsilon}} \int_{t-1}^{t+1} \varphi_k(t-u) \int_{u-\hat{\epsilon}}^{u+\hat{\epsilon}} \bigg(\bigg(\frac{B}{u'}-cT_{u'}\bigg)\mathbb{E}[R_{u'}+d]+B'T_{u'}^{-1}\bigg)du'du\\
&\leq  \int_{t-1}^{t+1} \varphi_k(t-u) \lim_{\hat{\epsilon}\rightarrow 0} \frac{1}{2\hat{\epsilon}}\int_{u-\hat{\epsilon}}^{u+\hat{\epsilon}} \bigg(\bigg(\frac{B}{u'}-cT_{u'}\bigg)\mathbb{E}[R_{u'}+d]+B'T_{u'}^{-1}\bigg)du'du\\
&= \int_{t-1}^{t+1} \varphi_k(t-u) (\bigg(\bigg(\frac{B}{u}-cT_{u}\bigg)\mathbb{E}[R_{u}+d]+B'T_{u}^{-1}\bigg)du,
\end{align*}
where the second-to-last inequality follows by Fatou's lemma and dominated convergence; the last line follows from the Lebesgue differentiation theorem. In other words, for $\hat{g}:\mathbb{R}_+\rightarrow\mathbb{R}$ given by $\hat{g}(t) := Bt^{-1}-cT_t$, it holds that 
\begin{align*}
\frac{d}{dt}(\varphi_k\ast \mathbb{E}[R_\bullet+d]) (t) &\leq (\varphi_k \ast (\hat{g}\mathbb{E}[R_\bullet+d]+B'T_\bullet^{-1}))(t)\\
&\leq \bigg(\frac{B}{t-1}-c\inf_{t-1\leq s \leq t+1}T_s\bigg)(\varphi_k \ast \mathbb{E}[R_\bullet+d])(t)+B'\Big(\inf_{t-1\leq s \leq t+1}T_s\Big)^{-1}.
\end{align*}
Using Assumption \ref{assumption3}, we get
\begin{equation*}
\frac{d}{dt}(\varphi_k\ast\mathbb{E}[R_\bullet+d])(t) \leq -\frac{cE}{2}(\ln (t+1))^{-1}(\varphi_k\ast\mathbb{E}[R_\bullet+d])(t)+ \frac{B'}{E}\ln (t+1) 
\end{equation*}
for $t>t_0^*$, where $t_0^*>t_0$ is such that $\frac{B}{t-1}\leq \frac{cE}{2\ln (t+1)}$ for $t>t_0^*$. Then, for $t>t_0^*$,
\begin{align*}
\frac{d}{dt}\bigg(e^{\frac{cE}{2}\!\int_{t_0^*}^t\! (\ln (s+1))^{-1} ds}(\varphi_k\ast\mathbb{E}[R_\bullet+d])(t)\bigg)&\leq \frac{B'}{E} \ln(t+1) e^{\frac{cE}{2}\!\int_{t_0^*}^t\! (\ln (s+1))^{-1} ds},\\
(\varphi_k\ast\mathbb{E}[R_\bullet+d])(t) &\leq (\varphi_k\ast\mathbb{E}[R_\bullet+d])(t_0^*)e^{-\frac{cE}{2}\!\int_{t_0^*}^t\! (\ln (s+1))^{-1} ds}\\
&\quad + \int_{t_0^*}^t \frac{B'}{E} \ln(s+1) e^{-\frac{cE}{2}\!\int_s^t\! (\ln (u+1))^{-1} du} ds\\
&\leq (\varphi_k\ast\mathbb{E}[R_\bullet+d])(t_0^*) + \frac{B'}{E} \int_{t_0^*}^t \ln (s+1) \ e^{-\frac{cE}{2}\!\int_s^t\! (\ln (u+1))^{-1} du} ds\\
&\leq (\varphi_k\ast\mathbb{E}[R_\bullet+d])(t_0^*) + \frac{B'}{E}\ln (t+1) \int_{t_0^*}^t \ e^{-\frac{cE}{2\ln (t+1)}(t-s) } ds\\
&\leq (\varphi_k\ast\mathbb{E}[R_\bullet+d])(t_0^*) + B'\frac{2(\ln (t+1))^2}{cE^2}\Big(1-e^{-\frac{cE}{2}(\ln (t+1))^{-1}(t-t_0^*)}\Big)\\
&\leq (\varphi_k\ast\mathbb{E}[R_\bullet+d])(t_0^*) + B'\frac{2(\ln (t+1))^2}{cE^2}.
\end{align*}
Using \eqref{partialT}, 
\begin{equation}\label{simspi}
\mathbb{E}[R_t+d] = \mathbb{E}[R_t+d]\int_{t-2k}^t \varphi_k(t-k-s)ds \leq \int_{t-2k}^t\varphi_k(t-k-s)\mathbb{E}[R_s+d]ds + \bar{g}(2k)
\end{equation}
for some $\bar{g}:\mathbb{R}\rightarrow\mathbb{R}$ satisfying $\bar{g}'(k')\rightarrow 0$ as $k'\rightarrow 0$, so that for $t>t_0^*+2$,
\begin{equation*}
\mathbb{E}[R_t+d]\leq (\varphi_k\ast\mathbb{E}[R_\bullet+d])(t_0^*) + B'\frac{2(\ln (t-k+1))^2}{cE^2} + \bar{g}(2k),
\end{equation*}
where the first term on the right can be bounded independently of $k$ via Proposition \ref{refer2app} and \eqref{Rbd} in a similar spirit to \eqref{simspi}. The result follows by taking $k\rightarrow 0$.
\end{proof}

\begin{corollary}\label{moments}
\color{black}
Under Assumption \ref{assumption1}, \ref{assumption3}, \ref{assumption4} and for $0\leq \epsilon \leq \epsilon'$, the solution $(X_t^\epsilon,Y_t^\epsilon,Z_t^\epsilon)$ to \eqref{glep0} is such that $\frac{\mathbb{E}[\abs*{X_t^\epsilon}^2+\abs*{Y_t^\epsilon}^2+\abs*{Z_t^\epsilon}^2]}{(\ln(e+t))^2}$ is bounded uniformly in time and in $\epsilon$.
\end{corollary}
\begin{proof}\color{black}
By the lower bound on $R$ in \eqref{Rbd},
\begin{equation*}
\mathbb{E}\big[\abs*{X_t^\epsilon}^2+\abs*{Y_t^\epsilon}^2+\abs*{Z_t^\epsilon}^2\big] \leq \mathbb{E}\bigg[\frac{R(X_t^\epsilon,Y_t^\epsilon,Z_t^\epsilon,T_t)+d}{a}\bigg],
\end{equation*}
which concludes by Lemma \ref{expRbd}.
\end{proof}

\subsection{Form of Distorted Entropy}

For $\epsilon \geq 0$, let $H^\epsilon(t)$ be the distorted entropy
\begin{align}
H^\epsilon(t)\ &=\ \int\bigg( \frac{\abs*{2\nabla\!_x h_t^\epsilon+8S_0(\nabla\!_y h_t^\epsilon+\lambda^{-1}\nabla\!_z h_t^\epsilon)}^2}{h_t^\epsilon}\ +\ \frac{\abs*{\nabla\!_y h_t^\epsilon+S_1\lambda^{-1}\nabla\!_z h_t^\epsilon}^2}{h_t^\epsilon} +\ \beta(T_t^{-1})h_t^\epsilon \ln(h_t^\epsilon)\bigg)d\mu_{T_t},\label{entdef}
\end{align}
where $S_0, S_1>0$ are the constants
\begin{align}
S_0 &:= (1+|D_x^2 U |_\infty^2)^{\frac{1}{2}},\label{s0}\\
S_1 &:= 2+\textcolor{black}{28}S_0^2+1024S_0^4\label{s1}
\end{align}
and $\beta$ is a second order polynomial (see \eqref{betdef} and \eqref{betacs}) to be determined by Proposition \ref{gambound} and independent of $\epsilon$.
{
\begin{remark}\label{2term}
This particular expression for $H$ is not necessarily the best possible choice. However the above is a working expression and optimality is left as future work; see also \cite{Vaes}.
\end{remark}
}

We now present an auxiliary result which can be found as Lemma 12 of \cite{meta}. We state this along with its proofs from \cite{meta}.

\begin{lemma}\label{combound}
For 
\[
\Phi^*(h) = \frac{\abs*{M\nabla h}^2}{h},
\]
where $M$ is matrix-valued,
\[
\Gamma_{L_t^{\epsilon *},\Phi^*}(h) > \frac{(M\nabla h)\cdot[L_t^{\epsilon *},M\nabla]h}{h}
\]
holds for all $h\in\mathcal{C}_+^\infty$.
\end{lemma}
Notice the $\Phi^*$ appears in the first two terms of $H^\epsilon(t)$.
\begin{proof}

The second term in definition \eqref{gamdef} of $\Gamma_{L_t^{\epsilon *},\Phi^*}(h)$ can be calculated to be
\begin{equation}\label{second}
-d\Phi^*(h).L_t^{\epsilon *} h = - \frac{2}{h}(M\nabla h)\cdot (M\nabla L_t^{\epsilon *} h) + \frac{L_t^{\epsilon *} h}{h^2}\abs*{M \nabla h}^2.
\end{equation}
Using \eqref{product} and \eqref{chain} for $L_t^{\epsilon *}$, the first term in the definition of $\Gamma_{L_t^{\epsilon *},\Phi}$ in \eqref{gamdef} can be calculated to be
\begin{align*}
L_t^{\epsilon *}(\Phi^*(h)) &=  \frac{1}{h}L_t^{\epsilon *}(\abs{M\nabla h}^2)+\abs*{M\nabla h}^2 L_t^{\epsilon *} \bigg(\frac{1}{h}\bigg) - \sum_i \frac{4}{h^2}(M\nabla h)_i\nabla h \cdot(A^\epsilon \nabla (M\nabla h)_i) \\
& = \frac{2}{h}\bigg((M\nabla h)\!\cdot\! L_t^{\epsilon *} M\nabla h + \!\sum_i \nabla (M\nabla h)_i \cdot(A^\epsilon\nabla (M\nabla h)_i )\bigg)\\
&\quad +\abs*{M\nabla h}^2 L_t^{\epsilon *} \bigg(\frac{1}{h}\bigg) \!-\! \sum_i \!\frac{4}{h^2}(M\nabla h)_i\nabla h \cdot(A^\epsilon \nabla (M\nabla h)_i),
\end{align*}
\textcolor{black}{
where the last summands can be bounded below because $A$ and $A^\epsilon$ are positive definite, that is, for $v,w\in\mathbb{R}^{2n+m}$ and $c\in\mathbb{R}$, we have
\begin{equation*}
(cv-c^{-1}w)^\top A^\epsilon(cv-c^{-1}w) = c^2 v^\top A^\epsilon v + c^{-2}w^\top A^\epsilon w - 2v^\top A^\epsilon w >0.
\end{equation*}
Therefore, for all $i$,
\begin{align*}
-\frac{4}{h^2}&(M\nabla h)_i \nabla h \cdot(A^\epsilon\nabla (M\nabla h)_i) > -\frac{2}{h^3}(M\nabla h)_i^2(\nabla h)\cdot A^\epsilon \nabla h - \frac{2}{h}(\nabla (M\nabla h)_i)\cdot  A^\epsilon \nabla (M\nabla h)_i,
\end{align*}}
which produces the bound
\begin{align*}
L_t^{\epsilon *}(\Phi^*(h)) > \frac{2}{h}(M\nabla h)\!\cdot\! L_t^{\epsilon *} M\nabla h + \abs*{M\nabla h}^2 L_t^{\epsilon *}\!\bigg(\frac{1}{h}\bigg)-\frac{2}{h^3}\abs*{M\nabla h}^2\nabla h\cdot (A^\epsilon \nabla h).
\end{align*}
Combining this with \eqref{second} then using \eqref{chain} and noting
\begin{equation}\label{combrac}
[L_t^{\epsilon *},M\nabla]h = (L_t^{\epsilon *}(M\nabla h)_1-(M\nabla L_t^{\epsilon *} h)_1,\dots, L_t^{\epsilon *}(M\nabla h)_{2n+m}-(M\nabla L_t^{\epsilon *} h)_{2n+m}),
\end{equation}
we get
\begin{align*}
\Gamma_{L_t^{\epsilon *},\Phi^*}(h) & > \frac{1}{h}(M\nabla h)\cdot[L_t^{\epsilon *},M\nabla]h+\frac{1}{2}\abs*{M\nabla h}^2\bigg( L_t^{\epsilon *}\bigg(\frac{1}{h}\bigg)+ \frac{L_t^{\epsilon *} h}{h^2}\bigg) -\frac{1}{h^3}\abs*{M\nabla h}^2\nabla h\cdot (A^\epsilon \nabla h)\\
& = \frac{1}{h}(M\nabla h)\cdot[L_t^{\epsilon *},M\nabla]h.
\end{align*}
\end{proof} 
Using Lemma \ref{combound}, the following proposition shows the distorted entropy (\ref{entdef}) is a useful one.
\begin{prop}\label{gambound}\color{black}
There exist $\beta_0, \beta_1, \beta_2>0$ independent of $\epsilon$ such that for $\beta:\mathbb{R}\rightarrow\mathbb{R}$ given by
\begin{align}
\beta(x) &:= 1+\beta_0 + \beta_1 x  + \beta_2 x^2, \label{betdef}
\end{align}
the operator $\Psi_{T_t}$,
\begin{align}\label{PsiTt}
\Psi_{T_t} (h)\  &:= \ \frac{\abs*{2\nabla\!_x h+8S_0(\nabla\!_y h+\lambda^{-1}\nabla\!_z h)}^2}{h}+\frac{\abs*{\nabla\!_y h+S_1\lambda^{-1}\nabla\!_z h}^2}{h}+\beta(T_t^{-1})h \ln(h)
\end{align}
for $h\in\mathcal{C}_+^\infty$, satisfies
\begin{equation}\label{gamineq}
\Gamma_{L_t^{\epsilon *},\Psi_{T_t}}(h) \geq \frac{\abs{\nabla h}^2}{h}
\end{equation}
for all $0\leq \epsilon \leq 1$.
\end{prop}
{
\begin{remark}
$\beta_0$, $\beta_1$, $\beta_2$ depend on $|D_x^2U|_\infty$, $\abs*{A}$ and
\begin{equation}\label{lambdahat}
\hat{\lambda}^2 := \max\Big(\abs*{\lambda}^2,\abs*{\lambda^\top}^2,\abs*{\lambda^{-1}}^2,\abs*{\lambda^{-1}}\abs*{\lambda^\top}\Big).
\end{equation}
$H$ satisfying property \eqref{gamineq} is crucial for proving dissipation in Proposition \ref{dissipation}. 
\end{remark}
}

\begin{proof}
Let $\Phi_1$,$\Phi_2$,$\Phi_3$ be the terms in $\Psi_{T_t}$,
\begin{subequations}\label{thePhis}
\begin{align}
\Phi_1(h) &:= \frac{\abs*{2\nabla\!_x h+8S_0(\nabla\!_y h+\lambda^{-1}\nabla\!_z h)}^2}{h},\\
\Phi_2(h) &:= \frac{\abs*{\nabla\!_y h+S_1\lambda^{-1}\nabla\!_z h}^2}{h},\\
\Phi_2(h) &:= h \ln(h).
\end{align}
\end{subequations}
Note that the $\Gamma_\Phi$ operator is linear in the $\Phi$ argument by linearity of $L_t^{\epsilon *}$, so that (\ref{gamineq}) can be written as
\begin{equation*}
\Gamma_{L_t^{\epsilon *},\Phi_1}(h)+\Gamma_{L_t^{\epsilon *},\Phi_2}(h)+\beta(T_t^{-1})\Gamma_{L_t^{\epsilon *},\Phi_3}(h)\geq \frac{\abs*{\nabla h}^2}{h}.
\end{equation*}
Consider $\Gamma_{L_t^{\epsilon *},\Phi_3}$ first. Using the definition (\ref{gamdef}) of $\Gamma_{L_t^{\epsilon *},\Phi}$, the product and chain rule (\ref{product}) and (\ref{chain}) for $L_t^{\epsilon *}$, \textcolor{black}{and the coercivity property of $A$, we get}
\textcolor{black}{
\begin{align}
\Gamma_{L_t^{\epsilon *},\Phi_3}(h) &= \frac{1}{2}\bigg((\ln h+1) L_t^{\epsilon *}h + \frac{1}{h}\nabla h\cdot(A^\epsilon\nabla h) - (1+\ln h)L_t^{\epsilon *}h\bigg)\nonumber\\
&= \frac{1}{2h}\nabla h \cdot (A^\epsilon\nabla h)\nonumber\\
&\geq \frac{1}{2h}(\epsilon \abs*{\nabla\!_x h}^2 + \epsilon \abs*{\nabla\!_y h}^2 + A_c\abs*{\nabla\!_z h}^2).\label{zd}
\end{align}}

Since the goal is to show \eqref{gamineq}, the availability of \eqref{zd} counteracts any negative contributions in the $z$-derivative term, and any order $\epsilon$ contributions in the $x$- and $y$-derivatives, from $\Gamma_{L_t^{\epsilon *},\Phi_1}$ and $\Gamma_{L_t^{\epsilon *},\Phi_2}$; this counterweight materialises as $\beta$.

For $\Gamma_{L_t^{\epsilon *},\Phi_1}$ and $\Gamma_{L_t^{\epsilon *},\Phi_2}$, $S_0>0$ and $S_1>0$ as in \eqref{s0}-\eqref{s1} are used.
Lemma \ref{combound} gives
\begin{align*}
h\Gamma_{L_t^*,\Phi_2}(h) & > (\nabla\!_y+S_1\lambda^{-1}\nabla\!_z) h\cdot [L_t^{\epsilon *},\nabla\!_y+S_1\lambda^{-1}\nabla\!_z] h\\
& =  (\nabla\!_y+S_1\lambda^{-1}\nabla\!_z) h\cdot (\nabla\!_x-\lambda^\top\nabla\!_z + \epsilon T_t^{-1}\nabla\!_y + S_1\nabla\!_y + S_1T_t^{-1}\lambda^{-1}A\nabla\!_z)h\\
& = \nabla\!_x h\cdot\!\nabla\!_y h-\!\nabla\!_y h\cdot(\lambda\!^\top\nabla_{\!z} h) + \epsilon T_t^{-1}\abs*{\nabla\!_y h}^2 + S_1 \abs*{\nabla_{\!y}h}^2\!\!+S_1T_t^{-1}\nabla_{\!y}h\cdot (\lambda^{-1}A\nabla_{\!z}h)\\
&\quad + S_1\nabla_{\!x}h\cdot(\lambda^{-1}\nabla_{\!z}h) - S_1(\lambda^{-1}\nabla_{\!z}h)\cdot(\lambda^\top\nabla_{\!z} h) + \epsilon T_t^{-1}S_1\nabla\!_y h \cdot (\lambda^{-1}\nabla\!_z h)\\
&\quad + S_1^2\nabla_{\!y}h\cdot(\lambda^{-1}\nabla_{\!z}h) + S_1^2T_t^{-1}(\lambda^{-1}\nabla_{\!z}h)\cdot(\lambda^{-1}A\nabla_{\!z}h)\\
\end{align*}
In order to get a bound in terms of $(\partial_ih)^2$ terms rather than $\partial_i h\partial_j h$ terms, we bound the $\partial_i h\partial_j h$ terms in the following ways,
\begin{align*}
\nabla\!_x h\cdot\!\nabla\!_y h &\geq -\frac{1}{2}\abs*{\nabla_{\!x}h}^2-\frac{1}{2}\abs*{\nabla_{\!y}h}^2,\\
-\!\nabla\!_y h\cdot(\lambda\!^\top\nabla_{\!z} h) &\geq -\frac{1}{6}\abs*{\nabla_{\!y}h}^2 -\frac{3}{2}\abs*{\lambda^\top}^2\abs*{\nabla_{\!z}h}^2,\\
S_1T_t^{-1}\nabla_{\!y}h\cdot (\lambda^{-1}A\nabla_{\!z}h) &\geq -\frac{1}{6}\abs*{\nabla_{\!y}h}^2 -\frac{3}{2}S_1^2T_t^{-2}\abs*{\lambda^{-1}}^2\abs*{A}^2\abs*{\nabla_{\!z}h}^2,\\
S_1\nabla_{\!x}h\cdot(\lambda^{-1}\nabla_{\!z}h) &\geq -\frac{1}{2}\abs*{\nabla_{\!x}h}^2 -\frac{1}{2} S_1^2 \abs*{\lambda^{-1}}^2\abs*{\nabla_{\!z}h}^2,\\
\epsilon T_t^{-1} S_1 \nabla_{\!y}h\cdot(\lambda^{-1}\nabla_{\!z}h) &\geq -\epsilon T_t^{-1} \abs*{\nabla_{\!y}h}^2 -\frac{\epsilon}{4}S_1^2 T_t^{-1}\abs*{\lambda^{-1}}^2\abs*{\nabla_{\!z}h}^2,\\
S_1^2\nabla_{\!y}h\cdot(\lambda^{-1}\nabla_{\!z}h) &\geq -\frac{1}{6}\abs*{\nabla_{\!y}h}^2 -\frac{3}{2}S_1^4\abs*{\lambda^{-1}}^2\abs*{\nabla_{\!z}h}^2
\end{align*}
and using \eqref{lambdahat} gives 
\begin{align*}
h\Gamma_{L_t^*,\Phi_2}(h) &> -\abs*{\nabla_{\!x}h}^2+(1+\textcolor{black}{28}S_0^2+1024S_0^4)\abs*{\nabla_{\!y}h}^2 \\
&\quad - \frac{1}{2}\hat{\lambda}^2\bigg( 3+\textcolor{black}{2}S_1+S_1^2+3S_1^4+S_1^2T_t^{-1}\bigg(\abs*{A}+\frac{\epsilon}{2}\bigg)+3S_1^2T_t^{-2}\abs*{A}^2  \bigg)\abs*{\nabla_{\! z}h}^2.
\end{align*}
Lastly, $\Gamma_{L_t^{\epsilon *},\Phi_1}$ compensates for the negative $x$-derivative:
\begin{align*}
h\Gamma_{L_t^{\epsilon *},\Phi_1}(h)& > (2\nabla\!_x+8S_0(\nabla\!_y+ \lambda^{-1}\nabla\!_z))h\cdot [L_t^{\epsilon *},2\nabla\!_x+8S_0(\nabla\!_y+ \lambda^{-1}\nabla\!_z)]h\\
& = (2\nabla\!_x+8S_0(\nabla\!_y+ \lambda^{-1}\nabla\!_z))h\cdot(-2(D_x^2 U )(\nabla\!_y - \epsilon T_t^{-1} \nabla\!_x)\\
&\quad +8S_0(\nabla\!_x + \epsilon T_t^{-1}\nabla\!_y - \lambda^\top \nabla\!_z + \nabla\!_y+T_t^{-1}\lambda^{-1}A\nabla\!_z))h\\
&= ((16S_0I_n + 4\epsilon T_t^{-1} D_x^2 U)\nabla\!_x h) \cdot \nabla\!_x h +2\nabla_{\!x}h\cdot((-2D_x^2U + 8S_0(1+\epsilon T_t^{-1})I_n)\nabla_{\!y}h) \\
&\quad + 2\nabla\!_x h\cdot (8S_0(-\lambda^\top+T_t^{-1}\lambda^{-1}A)\nabla\!_z h) + ((64S_0^2I_n + 16S_0\epsilon T_t^{-1}D_x^2 U)\nabla\!_x h) \cdot \nabla\!_y h \\
&\quad + 8S_0\nabla_{\!y}h\cdot((-2D_x^2U + 8S_0)\nabla_{\!y}h) + 8S_0\nabla\!_y h \cdot (8S_0(-\lambda^\top+T_t^{-1}\lambda^{-1}A)\nabla\!_z h)\\
&\quad + ((64S_0^2I_n + 16S_0\epsilon T_t^{-1}D_x^2U)\nabla\!_x h) \cdot (\lambda^{-1}\nabla\!_z h) \\
&\quad + ((-16S_0D_x^2U + 64S_0^2(1+\epsilon T_t^{-1})I_n)\nabla_{\!y}h) \cdot (\lambda^{-1}\nabla\!_z h)\\
&\quad + 64S_0^2(\lambda^{-1}\nabla\!_z h) \cdot ((-\lambda^\top+T_t^{-1}\lambda^{-1}A)\nabla\!_z h).
\end{align*}
\textcolor{black}{Bounding the $\partial_i h\partial_j h$ terms as for $\Phi_2$}, using \eqref{lambdahat} and \eqref{secbdd} yields
\begin{align*}
h\Gamma_{L_t^*,\Phi_1}(h)&> (16S_0 - 4\epsilon T_t^{-1}|D_x^2 U|_\infty) \abs*{\nabla\!_x h}^2\\
&\quad - \Big( 2\abs*{\nabla_{\!x}h}^2 + 2|D_x^2U|_\infty^2\abs*{\nabla_{\!y}h}^2 + 8(1+\epsilon T_t^{-1})\abs*{\nabla_{\!x}h}^2 + 8S_0^2(1+\epsilon T_t^{-1})\abs*{\nabla_{\!y}h}^2\Big)\\
&\quad - \Big(2\abs*{\nabla\!_x h}^2 + 32S_0^2\hat{\lambda}^2\Big(1+T_t^{-2}\abs*{A}^2\Big)\abs*{\nabla\!_z h}^2\Big)\\
&\quad - \Big((1+8S_0\epsilon T_t^{-1}|D_x^2 U|_\infty^2)\abs*{\nabla_{\!x}h}^2 + (1024S_0^4 + 8S_0\epsilon T_t^{-1})\abs*{\nabla_{\!y}h}^2 \Big) \\
&\quad - \Big(16S_0|D_x^2U|_\infty\abs*{\nabla_{\!y}h}^2 \textcolor{black}{-} 64S_0^2\abs*{\nabla_{\!y}h}^2 \Big) - \Big( 32S_0^2\abs*{\nabla_{\!y}h}^2 \!+ 32S_0^2\hat{\lambda}^2\Big(\!1+T_t^{-2}\abs*{A}^2\Big)\abs*{\nabla_{\!z}h}^2\!\Big)\\
&\quad - \Big( (1+8S_0\epsilon T_t^{-1}|D_x^2 U|_\infty^2)\abs*{\nabla\!_x h}^2 + (1024S_0^{\textcolor{black}{4}}+ 8S_0\epsilon T_t^{-1})\hat{\lambda}^2\abs*{\nabla\!_z h}^2\Big) \\
&\quad - \Big(\Big(2|D_x^2U|_\infty^2+ 32S_0^2(1+\epsilon T_t^{-1})\Big)\abs*{\nabla_{\!y}h}^2 + (32S_0^2 + 32\epsilon T_t^{-1} S_0^2)\hat{\lambda}^2\abs*{\nabla\!_z h}^2\Big)\\
&\quad   - 64S_0^2\hat{\lambda}^2\Big(1 + T_t^{-2}\abs*{A}^2\Big)\abs*{\nabla\!_z h}^2\\
&\geq \Big(2 - 4(2 + (1+4S_0^2) S_0)\epsilon T_t^{-1}\Big)\abs*{\nabla\!_x h}^2 + \Big(S_0^2 ( -\textcolor{black}{28}-1024S_0^2 ) - 8 S_0 (1 + 5S_0)\epsilon T_t^{-1}\Big)\abs*{\nabla\!_y h}^2\\
&\quad - \Big(S_0^2\hat{\lambda}^2(\textcolor{black}{160+128T_t^{-2}\abs*{A}^2 + 1024S_0^2}) + 8S_0\hat{\lambda}^2 (1+4S_0)\epsilon T_t^{-1}\Big) \abs*{\nabla_{\! z} h}^2.
\end{align*}
Matching powers in $T_t^{-1}$ to take
\begin{subequations}\label{betacs}
\begin{align}
\beta_0 &= \frac{1}{A_c}( S_0^2 \hat{\lambda}^2(160 + 1024S_0^2) + \frac{1}{2}\hat{\lambda}^2(3+2S_1 + S_1^2+ 3S_1^4 ))\\
\beta_1 &= \frac{1}{A_c}\bigg( 4(2+(1+4S_0^2)S_0) + 8S_0(1+5S_0) + 8S_0 \hat{\lambda}^2(1+4S_0) + \frac{1}{2} \hat{\lambda}^2\bigg(S_1^2\bigg(\abs*{A}+\frac{1}{2}\bigg)\bigg)\bigg)\\
\beta_2 &= \frac{1}{A_c}\bigg( 128 S_0^2\hat{\lambda}^2 \abs*{A}^2 + \frac{3}{2} \hat{\lambda}^2 S_1^2\abs*{A}^2\bigg),
\end{align}
\end{subequations}
using $\epsilon \leq 1$ and putting together the bounds for $\Gamma_{L_t^{\epsilon *},\Phi_3}, \Gamma_{L_t^{\epsilon *},\Phi_2},\Gamma_{L_t^{\epsilon *},\Phi_1}$ gives \eqref{gamineq}.
\end{proof}

\subsection{Log-Sobolev Inequality}

\begin{proof}\textit{of Proposition \ref{logsobprop0}} Firstly, the case that $U$ satisfies \eqref{q1} is dealt with. The standard log-Sobolev inequality for a Gaussian measure \cite{Gross} 
alongside the properties that log-Sobolev inequalities tensorises and are stable under perturbations, which can be found as Theorem 4.4 and Property 4.6 in \cite{Zega} 
respectively, yields the result. In particular,
\begin{align*}
\int h\ln h d\mu_{T_t} &= \int (h \ln h-h+1)d\mu_{T_t}\\
&\leq\int(h \ln h-h+1)Z_{T_t}^{-1}e^{-\frac{U_{\!m}}{T_t}}e^{-\frac{1}{T_t}\big(\abs*{\bar{a}\circ x}^2+\frac{\abs*{y}^2}{2}+\frac{\abs*{z}^2}{2}\big)} dxdydz\\
&= e^{-\frac{U_{\!m}}{T_t}}Z_{T_t}^{-1}\int h \ln h e^{-\frac{1}{T_t}\big(\abs*{\bar{a}\circ x}^2+\frac{\abs*{y}^2}{2}+\frac{\abs*{z}^2}{2}\big)}dxdydz\\
&\leq e^{-\frac{U_{\!m}}{T_t}}\!\max{\bigg(\frac{T_t}{2},\max_i\frac{T_t}{4\bar{a}_i^2}\bigg)}Z_{T_t}^{-1}\!\!\int\! \frac{\abs*{\nabla h}^2}{h_t} e^{-\frac{1}{T_t}\big(\abs*{\bar{a}\circ x}^2+\frac{\abs*{y}^2}{2}+\frac{\abs*{z}^2}{2}\big)}dxdydz\\
&\leq e^{\frac{U_{\!M}-U_{\!m}}{T_t}} \max{\bigg(\frac{T_t}{2},\frac{T_t}{4a_m^2}\bigg)}\int \frac{\abs*{\nabla h}^2}{h} d\mu_{T_t},
\end{align*}
where the first inequality follows by \eqref{q1} since $x\ln x-x+1\geq0$ for all $x\geq0$, so that 
\begin{equation*}
C_t^{(0)} = \max{\Big(2,a_m^{-2}\Big)}\frac{T_t}{4}e^{(U_{\!M}-U_{\!m})T_t^{-1}}
\end{equation*}
For the case where $U$ is a nonnegative nondegenerate Morse function satisfying \eqref{q1menz}, the inequality in the $x$-marginals is taken as a consequence of Corollary 2.17 in \cite{Sch}; for the announced form \eqref{logSobConst0} of $C_t^{(0)}$, equation (2.18) in \cite{Sch} can be used by taking $t_{ls}^{(0)}$ large enough such that for $t>t_{ls}^{(0)}$, $T_t$ is small enough. 
As before, tensorisation with the inequality for Gaussian measures concludes.
\end{proof}

\begin{prop}\label{logsobprop}
Under Assumption \ref{assumption1}, \ref{assumption2} and for $\epsilon \geq 0$, there exists constants $t_{ls}, A_*>0$ and a finite order polynomial $r:\mathbb{R}_+\rightarrow\mathbb{R}_+$ with coefficients depending on $U$ and $\lambda$ but independent of $\epsilon$ such that the distorted entropy \eqref{entdef} satisfies
\begin{equation}\label{logSob}
H^\epsilon(t) \leq C_t \int\frac{\abs*{\nabla h_t^\epsilon}^2}{h_t^\epsilon}d\mu_{T_t},
\end{equation}
where for $t>t_{ls}$,
\begin{equation}\label{logSobConst}
C_t = A_* + r\Big(T_t^{-\frac{1}{2}}\Big) e^{\hat{E}T_t^{-1}}.
\end{equation}
\end{prop}
\begin{proof}
Given Proposition \ref{logsobprop0}, only the first two terms in the integrand of $H^\epsilon(t)$ are left, which lead directly to the inequality corresponding to $A_*$. 
\end{proof}

\subsection{Proof of Dissipation}\label{prodis}

Lemma \ref{eta} below constructs a sequence of compactly supported functions that are multiplied with the integrand in $H(t)$. It
gives sufficient properties for retrieving a bound on $\partial_tH(t)$ after passing the derviative under the integral sign and passing the limit in the sequence of approximating initial densities. The key sufficient property turns out to be (\ref{etaprop}) below.\\
Let $\varphi_k$ be given as in \eqref{molli},
\begin{equation*}
\nu_k:=\varphi_k* \mathds{1}_{(-\infty,k^2]}\leq 1
\end{equation*}
for $k>0$.

\begin{lemma}\label{eta}
For $k>0$, define the smooth functions $\eta_{k}:\mathbb{R}^{2n+m+1}\rightarrow\mathbb{R}$
\begin{equation*}
\eta_{k}=\nu_k(\textcolor{black}{-}\ln (R+2d)),
\end{equation*}
where $d>0$ is the same as in \eqref{Rbd}. The following properties hold:
\begin{enumerate}
\item $\eta_k$ is compactly supported;
\item $\eta_k$ converges to 1 pointwise as $k\rightarrow\infty$;
\item for some constant $C>0$ independent of $k$, $t$ and $0 \leq\epsilon\leq \min(1,\epsilon')$
\begin{equation}\label{etaprop}
L_t^\epsilon\eta_k\leq \frac{CT_t^{-1}}{k}.
\end{equation}
\end{enumerate}
\end{lemma}

\begin{proof}
By the quadratic assumption \eqref{q1menz} on $U$ and the bound \eqref{Rbd} on $R$, $R$ grows quadratically and in particular for an arbitrarily large constant $R_{(0)}>0$, a compact set $K$ can be chosen such that $R>R_{(0)}$ in $\mathbb{R}^{2n+m}\setminus K$; along with the support of $\nu_m$ being bounded below, the first statement is clear. The second statement is also trivial to check.\\
For the third statement, using the chain rule \eqref{Lchain1} and \eqref{Lchain2} for $L_t$,
\begin{equation*}
L_t^\epsilon\eta_m = -\nu_m'(-\ln(R+2d))L_t^\epsilon\ln(R+2d)+\nu_m''(-\ln(R+2d))(\nabla\ln(R+2d))^\top A^\epsilon \nabla \ln(R+2d).
\end{equation*}
It can be seen that $\nu_m'$ and $\nu_m''$ are estimated by terms at most of order $m^{-1}$; to see this, for all $x\in\mathbb{R}$,
\textcolor{black}{\begin{equation*}
\nu_m(x)=\int_{-\infty}^{m^2}\varphi_m(x-y)dy = \int_{x-m^2}^{\infty}\varphi_m(z)dz,
\end{equation*}
so that
\begin{equation*}
0 \geq \nu_m'(x) = -\varphi_m(x-m^2) \geq -m^{-1}\max \varphi
\end{equation*}
and
\begin{equation*}
\abs*{\nu_m''(x)} = \abs*{\varphi_m'(x-m^2)}\leq m^{-2}\max \varphi'.
\end{equation*}}
Therefore there exists a constant $\bar{C}>0$ such that
\textcolor{black}{\begin{align*}
L_t^\epsilon\eta_m &\leq -\nu_m'(-\ln (R+2d))\max(0,L_t^\epsilon\ln(R+2d)) + m^{-2}\max\varphi' \abs*{(\nabla \ln(R+2d))^\top A^\epsilon \nabla\ln(R+2d)}.\\
&\leq \bar{C}\Big(m^{-1}\max(0,L_t^\epsilon \ln(R+2d))+m^{-2}\abs*{(\nabla\ln(R+2d))^\top A^\epsilon \nabla\ln(R+2d)}\Big).
\end{align*}}
A calculation using property \eqref{Lyaprop} with \eqref{Lchain1} and \eqref{Lchain2} for $L_t^\epsilon$ reveals
\begin{align*}
L_t^\epsilon \ln(R+2d)&=\frac{L_t^\epsilon R}{R+2d}-\frac{(\nabla R)^\top A^\epsilon \nabla R}{(R+2d)^2}\\
&\leq \frac{-cT_tR+dT_t^{-1}}{R+2d}-\frac{\epsilon(\abs*{\nabla\!_xR}^2 + \abs*{\nabla\!_yR}^2) +A_c\abs*{\nabla\!_zR}^2}{(R+2d)^2}\\
&\leq \frac{-cT_t(R+d)+cT_td+dT_t^{-1}}{R+2d}-\frac{\epsilon(\abs*{\nabla\!_xR}^2 + \abs*{\nabla\!_yR}^2) + A_c\abs*{\nabla\!_zR}^2}{(R+2d)^2}\\
(\nabla\ln(R+2d))^\top A^\epsilon \nabla\ln(R+2d)&\leq(\abs*{A}+2)\abs*{\nabla \ln(R+2d)}^2 \\
&= (\abs*{A}+2)\abs*{\frac{\nabla R}{R+2d}}^2,
\end{align*}
which are bounded above as claimed considering \eqref{Rbd} and that $\nabla R$ grows linearly in space and is uniformly bounded in time.
\end{proof}

\begin{remark}\label{trunc_issue}\color{black}
Lemma \ref{eta} is different to Lemma 16 in \cite{meta}. We believe the first few equations in the 
proof of Lemma 16 in \cite{meta} contain a sign error; as a consequence the proofs in \cite{meta} beyond that point require significant modifications. Here we address this by modifying the truncation arguments we require proving \eqref{etaprop} instead of Lemma 17 of \cite{meta}. In addition, the finiteness of the distorted entropy is required, which is the reason for using the perturbed dynamics in \eqref{glep0} and then later taking $\epsilon\rightarrow 0$. 
\end{remark}

The proof of Proposition \ref{dissipation} follows in the direction of Lemma 19 of \cite{meta}.

\begin{prop}\label{dissipation}
Under Assumption \ref{assumption1}, \ref{assumption2}, \ref{assumption3} and \ref{assumption4} and for $0 < \epsilon \leq \min(1,\epsilon')$, 
it holds that for any $0<\alpha\leq \frac{1}{2}(1-\frac{\hat{E}}{E})$, there exists some constant $B>0$ and some $t_H>0$ {both independent of $\epsilon$}, such that for all $t>t_H$,
\begin{equation}\label{H_bound}
H^\epsilon(t)\leq B\bigg( \frac{1}{t}\bigg)^{1-\frac{\hat{E}}{E}-2\alpha}.
\end{equation}
\end{prop}
\begin{proof}
Consider for $t\geq 0$ the auxiliary distorted entropies
\begin{align}
H_k^\epsilon(t) &= \int \eta_k\bigg( \frac{\abs*{2\nabla\!_x h_t^\epsilon+8S_0(\nabla\!_y h_t^\epsilon + \lambda^{-1}\nabla\!_z h_t^\epsilon)}^2}{h_t^\epsilon}+\ \frac{\abs*{\nabla\!_y h_t^\epsilon + S_1\lambda^{-1}\nabla\!_z h_t^\epsilon}^2}{h_t^\epsilon} + \beta(T_t^{-1})h_t^\epsilon \ln(h_t^\epsilon)\bigg)d\mu_{T_t}\nonumber\\
& = \int\eta_k(\Phi_1(h_t^\epsilon)+\Phi_2(h_t^\epsilon)+\beta(T_t^{-1})\Phi_3(h_t^\epsilon))d\mu_{T_t}\nonumber\\
& = \int\eta_k\Psi_{T_t}(h_t^\epsilon)d\mu_{T_t},\label{mkH}
\end{align}
where $h_t^\epsilon=m_t^\epsilon\mu_{T_t} ^{-1}$ is as defined \eqref{ht}, $\Phi_1$, $\Phi_2$, $\Phi_3$ is as in \eqref{thePhis} and $\eta_k$ are as in Lemma \ref{eta}. Due to $\eta_k$, the order between the time derivative and the integral can be exchanged:
\begin{equation}\label{diffH}
\frac{d}{dt} H_k^\epsilon(t) = \int\eta_k\partial_t(\Psi_{T_t}(h_t^\epsilon))d\mu_{T_t} + T_t'\int \eta_k\partial_{T_t}(\Psi_{T_t}(h_t^\epsilon)\mu_{T_t})dxdydz.
\end{equation}
The terms will be considered separately.
Since $m_t^\epsilon$ is the density of the law of \eqref{glep0} and $L_t^{\epsilon *}$ is the $L^2(\mu_{T_t})$ adjoint of $L_t^\epsilon$, by It\^{o}'s rule for smooth compactly supported $f$ on $\mathbb{R}^{2n+m}$,
\begin{equation}\label{eqforh}
\int f\partial_t m_t^\epsilon = \partial_t\int f m_t^\epsilon = \int L_t^\epsilon f m_t^\epsilon = \int L_t^\epsilon f \frac{m_t^\epsilon}{\mu_{T_t}}\mu_{T_t} = \int f L_t^{\epsilon *}\bigg(\frac{m_t^\epsilon}{\mu_{T_t}}\bigg)\mu_{T_t}.
\end{equation}
The first term in \eqref{diffH} is then bounded as follows.
\begin{align}
\int\eta_k\partial_t(\Psi_{T_t}(h_t^\epsilon))d\mu_{T_t} &= \int \eta_k d\Psi_{T_t}(h_t^\epsilon).\partial_t h_t^\epsilon d\mu_{T_t}\nonumber\\
&= \int \eta_k d\Psi_{T_t}(h_t^\epsilon).\frac{\partial_t m_t^\epsilon}{\mu_{T_t}}d\mu_{T_t}\nonumber\\
&= \int \eta_k d\Psi_{T_t}(h_t^\epsilon).L_t^{\epsilon *}h_t^\epsilon d\mu_{T_t}\nonumber\\
&=-\int 2\eta_k\Gamma_{L_t^{\epsilon *},\Psi_{T_t}}(h_t^\epsilon)d\mu_{T_t}+\int  \eta_k L_t^{\epsilon *}(\Psi_{T_t}(h_t^\epsilon))d\mu_{T_t}\nonumber\\
&=-\int 2\eta_k \Gamma_{L_t^{\epsilon *},\Psi_{T_t}}(h_t^\epsilon)d\mu_{T_t} +\int L_t^\epsilon \eta_k \Big(\Psi_{T_t}(h_t^\epsilon)+\beta(T_t^{-1})e^{-1}\Big)d\mu_{T_t}\nonumber\\
&\leq - 2 \int\! \eta_k \frac{\abs{\nabla h_t^\epsilon}^2}{h_t^\epsilon}d\mu_{T_t}+\frac{CT_t^{-1}}{k}\int\!\Big(\Psi_{T_t}(h_t^\epsilon) + \beta(T_t^{-1})e^{-1}\Big)d\mu_{T_t},\label{4.3.1}
\end{align}
using Proposition \ref{gambound} and Lemma \ref{eta}, where $\beta(T_t^{-1})e^{-1}\int L_t^{\epsilon *}\eta_kd\mu_{T_t}=0$ is added to force 
\begin{equation*}
\beta(T_t^{-1})(h_t^\epsilon\ln h_t^\epsilon+e^{-1})\geq0,\quad \textrm{so that}\quad \Psi_{T_t}(h_t^\epsilon)+\beta(T_t^{-1})e^{-1}\geq0.
\end{equation*}
For the second term in (\ref{diffH}), consider the $\Phi_1$ and $\Phi_2$ terms in the integrand $\eta_k\partial_{T_t}(\Psi_{T_t}\mu_{T_t})=\eta_k\partial_{T_t}((\Phi_1+\Phi_2+\beta(T_t^{-1})\Phi_3)\mu_{T_t})$ of $H_{k}(t)$ with the forms
\begin{equation*}
\partial_{T_t}(\Phi_i(h_t^\epsilon)\mu_{T_t}) = \partial_{T_t}\abs*{M_i\nabla\ln\bigg(\frac{m_t^\epsilon}{\mu_{T_t}}\bigg)}^2m_t^\epsilon,\qquad i=1,2
\end{equation*}
for the corresponding matrices $M_1$ and $M_2$ depending on $S_0$, $S_1$ and $\lambda$. Applying the partial derivative in $T_t$, 
\begin{equation}\label{Mpart}
\partial_{T_t}(\Phi_i(h_t^\epsilon)\mu_{T_t}) = -2(M_i\nabla\ln h_t^\epsilon\cdot M_i\nabla\partial_{T_t}\ln\mu_{T_t})m_t^\epsilon,
\end{equation}
and using definition \eqref{equilibrium} for $\mu_{T_t}$ and $Z_{T_t} = \int_{\mathbb{R}^{2n+m}} e^{-\frac{1}{T_t}\big(U(x)+\frac{\abs*{y}^2}{2}+\frac{\abs*{z}^2}{2}\big)}dxdydz$, gives
\begin{align}
\partial_{T_t}\ln\mu_{T_t} &= \mu_{T_t}^{-1}\partial_{T_t}\bigg(Z_{T_t}^{-1}e^{-\frac{1}{T_t}\big(U(x)+\frac{\abs*{y}^2}{2}+\frac{\abs*{z}^2}{2}\big)}\bigg)\nonumber\\
&=\mu_{T_t}^{-1}\Bigg(\!- Z_{T_t}^{-2}\partial_{T_t}Z_{T_t}+\frac{Z_{T_t}^{-1}}{T_t^2}\Bigg(U(x)+\frac{\abs*{y}^2}{2}+\frac{\abs*{z}^2}{2}\Bigg)\Bigg)e^{-\frac{1}{T_t}\big(U(x)+\frac{\abs*{y}^2}{2}+\frac{\abs*{z}^2}{2}\big)}\nonumber\\
&=\mu_{T_t}^{-1}\Bigg(\!-\mu_{T_t}Z_{T_t}^{-1}\partial_{T_t}Z_{T_t}+\frac{\mu_{T_t}}{T_t^2}\Bigg(U(x)+\frac{\abs*{y}^2}{2}+\frac{\abs*{z}^2}{2}\Bigg)\Bigg)\nonumber\\
&= - \int \frac{1}{T_t^2}\Bigg(U(x)+\frac{\abs*{y}^2}{2}+\frac{\abs*{z}^2}{2}\Bigg)d\mu_{T_t} + \frac{1}{T_t^2}\Bigg(U(x)+\frac{\abs*{y}^2}{2}+\frac{\abs*{z}^2}{2}\Bigg).\label{spacef}
\end{align}
Note the exchange in differentiation and integration is justified by the quadratic bounds \eqref{q1menz} on $U$. \textcolor{black}{Integrating by parts in $y$ and $z$ (or simply using formulae for second moments) gives $
\frac{n+m}{2T_t}$ for the $\abs{y}^2$ and $\abs{z}^2$ terms in the first integral. The integral over $U$ can be dealt with using assumptions \eqref{q2} and \eqref{q3}, to be specific:
\begin{align*}
\int U d\mu_{T_t} &\leq \int (a_M^2\abs{x}^2 + U_M) d\mu_{T_t}\leq \int \bigg(\frac{a_M^2}{r_1}(\nabla U\cdot x + U_g) + U_M\bigg) d\mu_{T_t} = \frac{a_M^2}{r_1}( nT_t + U_g ) +U_M\\
\int U d\mu_{T_t} &\geq \int (a_m^2 \abs{x}^2 + U_m) d\mu_{T_t}\\
&\geq \int \bigg(\frac{a_m^2}{r_2+1}(\abs{\nabla U}^2 - U_g + \abs{x}^2) + U_m\bigg)d\mu_{T_t}\\
&\geq \int \bigg(\frac{a_m^2}{r_2+1}( 2\nabla U\cdot x - U_g ) + U_m\bigg)d\mu_{T_t} = \frac{a_m^2}{r_2+1}(2nT_t - U_g) + U_m.
\end{align*}
Plugging into \eqref{spacef} gives}
\begin{equation}\label{trick2}
p_1\Big(T_t^{-1}\Big) \leq \partial_{T_t}\ln\mu_{T_t} - \frac{1}{T_t^2}\Bigg(U(x)+\frac{\abs*{y}^2}{2}+\frac{\abs*{z}^2}{2}-\frac{n+m}{2} T_t\Bigg)\leq p_2\Big(T_t^{-1}\Big).
\end{equation}
where $p_1(x)=-\frac{a_{\!M}^2n}{r_1}x-\Big(\frac{a_{\!M}^2 U_g}{r_1}+U_{\!M}\Big)x^2$ and $p_2(x)=-\frac{2 a_m^2 n}{r_2+1}x+\Big(\frac{a_m^2 U_g}{r_2+1}-U_{\!m}\Big)x^2$.\\
Substituting \eqref{spacef} back into \eqref{Mpart},
\begin{align}
\partial_{T_t}(\Phi_i(h_t^\epsilon)\mu_{T_t}) &\leq \Bigg(\abs*{M_i\nabla\ln h_t^\epsilon}^2+T_t^{-4}\abs*{M_i\nabla\bigg(U(x)+\frac{\abs{y}^2}{2}+\frac{\abs{z}^2}{2}\bigg)}^2\Bigg)m_t^\epsilon\nonumber\\
&\leq \Phi_i(h_t^\epsilon)\mu_{T_t}+\widetilde{C}T_t^{-4}\Big(1+\abs*{x}^2+\abs*{y}^2+\abs*{z}^2\Big)m_t^\epsilon\label{phi12terms}
\end{align}
for a constant $\widetilde{C}\geq 0$ independent of $k$ and $\epsilon$ by the quadratic assumption \eqref{q3} on $\abs*{\nabla_{\!x}U}^2$ and $\eta_m\leq 1$.\\[0.5em]
For the last integrand in the last term of the right hand side of \eqref{diffH}, namely the derivative over $\Phi_3(h_t^\epsilon)\mu_{T_t} = \frac{m_t^\epsilon}{\mu_{T_t}}\ln\frac{m_t^\epsilon}{\mu_{T_t}}\mu_{T_t}$, the left inequality of \eqref{trick2} gives
\begin{align}
&\partial_{T_t}(\beta(T_t^{-1})\Phi_3(h_t^\epsilon)\mu_{T_t}) \nonumber\\
&= -T_t^{-2}\beta'(T_t^{-1})\Phi_3(h_t^\epsilon)\mu_{T_t} + \beta(T_t^{-1})\partial_{T_t}\ln\frac{m_t^\epsilon}{\mu_{T_t}}m_t^\epsilon\nonumber\\
&= -T_t^{-2}\beta'(T_t^{-1})(\Phi_3(h_t^\epsilon)+e^{-1})\mu_{T_t} + T_t^{-2}\beta'(T_t^{-1})e^{-1}\mu_{T_t} - \beta(T_t^{-1})\partial_{T_t}\ln\mu_{T_t}m_t^\epsilon\nonumber\\
& \leq T_t^{-2}\beta'(T_t^{-1})e^{-1}\mu_{T_t} + \beta(T_t^{-1})\abs*{p_1\Big(T_t^{-1}\Big) + \frac{1}{T_t^2}\bigg(\!-\frac{n+m}{2}T_t\! +U_{\!M}\!+a_M\abs*{x}^2\!+\!\frac{\abs{y}^2}{2}\!+\!\frac{\abs{z}^2}{2}\bigg)}m_t^\epsilon, \label{phi3term}
\end{align}
where in the last step $\Phi_3+e^{-1} \geq 0$, $\beta_1,\beta_2>0$ and \eqref{q1menz} have been used. 
Putting together the bounds \eqref{phi12terms} and \eqref{phi3term} and applying Corollary \ref{moments} yields
\begin{align}
\int \eta_k\partial_{T_t}(\Psi_{T_t}(h_t^\epsilon)\mu_{T_t})d\zeta &\leq q\Big(T_t^{-1}\Big)\Big(H_k^\epsilon(t)+\mathbb{E}\Big[1+\abs*{X_t^\epsilon}^2+\abs*{Y_t^\epsilon}^2+\abs*{Z_t^\epsilon}^2\Big]\Big)\nonumber\\
&\leq p\Big(T_t^{-1}\Big)\Big(H_k^\epsilon(t)+\hat{C}\Big), \label{4.3.2}
\end{align}
where $p$ and $q$ are some finite order polynomials with nonnegative coefficients, $\hat{C}> 0$, both independent of $k$ and $\epsilon$.\\[1em]
Returning to \eqref{diffH}, collecting \eqref{4.3.1} and \eqref{4.3.2} then integrating from any $s\geq 0$ to $t>s$ gives
\begin{align}
H_k^\epsilon(t)-H_k^\epsilon(s) &\leq 2\int_s^t \!\! \bigg(-\int \eta_k \frac{\abs{\nabla h_u^\epsilon}^2}{h_u}d\mu_{T_u}+ \frac{CT_u^{-1}}{k}(H^\epsilon(u)+ \beta(T_u^{-1})e^{-1}) + \abs*{T_u'}p\Big(T_u^{-1}\Big)\Big(H_k^\epsilon(u)+\hat{C}\Big)\bigg)du.\label{ineqint}
\end{align}
Fix an arbitrary $S>0$. By the square integrability Theorem 7.4.1 in \cite{MR3443169}, the log-Sobolev inequality \eqref{logSob}, \eqref{q3} and the finiteness of second moments \eqref{prop7}, it holds that
\begin{align}
\int_0^S H^\epsilon(u) du &\leq \int_0^S C_u \int\frac{\abs*{\nabla h_u^\epsilon}^2}{h_u^\epsilon} d\mu_{T_u}du\nonumber\\
&= \int_0^S C_u \int\frac{\abs*{\nabla m_u^\epsilon + T_u^{-1}m_u^\epsilon (\nabla\!_x U + y + z)}^2}{m_u^\epsilon} dxdydzdu < \infty.\label{741finite}
\end{align}
Then in \eqref{ineqint} the $k\rightarrow\infty$ limit can be taken. Due to \eqref{741finite}, the term denominated by $k$ goes to zero. Applying Fatou's lemma (adding and subtracting $\beta(T_t^{-1})e^{-1}\int\eta_md\mu_{T_t}$ wherever necessary for positivity) and using $\eta_m\leq1$, it holds that for $s<t$,
\begin{equation}\label{mid}
H^\epsilon(t)-H^\epsilon(s) \leq -2\int_s^t \!\! \int \frac{\abs{\nabla h_u^\epsilon}^2}{h_u^\epsilon}d\mu_{T_u}du + \int_s^t\!\abs*{T_u'}p\Big(T_u^{-1}\Big)\Big(H^\epsilon(u)+\hat{C}\Big)du
\end{equation}
and for\footnote{$t_{ls}$ from Proposition \ref{logsobprop}} $t_{ls}<s<t$,
\begin{equation}
H^\epsilon(t)-H^\epsilon(s) \leq \int_s^t\!\! \bigg(\!\Big(\abs*{T_u'}p\Big(T_u^{-1}\Big)- 2 C_u^{-1}\Big)H^\epsilon(u)+\hat{C}\abs*{T_u'}p\Big(T_u^{-1}\Big)\bigg)du. \label{ineqint2}
\end{equation}
Since $t^\alpha\gg(\ln t)^{\frac{\rho}{2}}$ for any $\rho,\alpha>0$ and large enough $t>0$, for any $\alpha>0$, there exists $t_1>\max(t_{ls},t_0)$, where $t_0$ is as in Assumption \ref{assumption3}, and $c_1,c_2>0$ independent of $k,\epsilon$ such that for all $t\geq t_1$,
\begin{align}
\abs*{T_t'}p\Big(T_t^{-1}\Big) &\leq c_1\bigg(\frac{1}{t}\bigg)^{1-\alpha},\label{coeff1}\\
- 2C_t^{-1} &\leq - c_2\bigg(\frac{1}{t}\bigg)^{\frac{\hat{E}}{E}+\alpha},\label{coeff2}
\end{align}
where the assumption $T_t \geq \frac{E}{\ln t}$ and \eqref{logSobConst} have been used. Using further that $E > \hat{E}$ by Assumption \ref{assumption3}, then taking $\alpha < \frac{1}{2}(1-\frac{\hat{E}}{E})$, there exists $t_2\geq t_1$ independent of $\epsilon$ such that for $t\geq t_2$,
\begin{equation}\label{newbb}
\abs*{T_t'}p\Big(T_t^{-1}\Big) - 2C_t^{-1}  \leq -c_3\bigg(\frac{1}{t}\bigg)^{\frac{\hat{E}}{E} + \alpha}
\end{equation}
and from \eqref{ineqint2}, for $t_2 < s < t$,
\begin{equation}\label{div}
H^\epsilon(t)-H^\epsilon(s) \leq \int_s^t \bigg(- c_3 \bigg(\frac{1}{u}\bigg)^{\frac{\hat{E}}{E} + \alpha}H^\epsilon(u)+\hat{C}c_1\bigg(\frac{1}{u}\bigg)^{1-\alpha}\bigg)du.
\end{equation}
To obtain the corresponding differential inequality for all time, \eqref{div} can be divided by $t-s$, mollified with \eqref{molli} for $0<k<1$ and the limit $s\rightarrow t$ can be taken:
\begin{align*}
\lim_{\hat{\epsilon}\rightarrow 0}\frac{1}{2\hat{\epsilon}}\int_{t-1}^{t+1} &\varphi_k(t - u) (H^\epsilon(u+\hat{\epsilon}) - H^\epsilon(u-\hat{\epsilon})) du\\
&\leq \lim_{\hat{\epsilon}\rightarrow 0} \frac{1}{2\hat{\epsilon}} \int_{t-1}^{t+1} \varphi_k(t-u) \int_{u-\hat{\epsilon}}^{u+\hat{\epsilon}}\bigg(- c_3 \bigg(\frac{1}{u'}\bigg)^{\frac{\hat{E}}{E} + \alpha}H^\epsilon(u')+\hat{C}c_1\bigg(\frac{1}{u'}\bigg)^{1-\alpha}\bigg)du'du\\
&\leq \int_{t-1}^{t+1} \varphi_k(t-u) \lim_{\hat{\epsilon}\rightarrow 0} \frac{1}{2\hat{\epsilon}} \int_{u-\hat{\epsilon}}^{u+\hat{\epsilon}}\bigg(- c_3 \bigg(\frac{1}{u'}\bigg)^{\frac{\hat{E}}{E} + \alpha}H^\epsilon(u')+\hat{C}c_1\bigg(\frac{1}{u'}\bigg)^{1-\alpha}\bigg)du'du\\
&= \int_{t-1}^{t+1} \varphi_k(t-u) \bigg(- c_3 \bigg(\frac{1}{u}\bigg)^{\frac{\hat{E}}{E} + \alpha}H^\epsilon(u)+\hat{C}c_1\bigg(\frac{1}{u}\bigg)^{1-\alpha}\bigg)du
\end{align*}
for $t\geq t_2 + 2$, where the second-to-last line follows from Fatou's lemma and dominated convergence (adding and subtracting $\beta(T_{u'}^{-1})e^{-1}$ to $H^\epsilon$ for Fatou); the last equality follows by the Lebesgue differentiation theorem. Therefore
\begin{equation*}
\frac{d}{dt}(\varphi_k\ast H^\epsilon) (t) \leq -c_3\bigg(\frac{1}{t+1}\bigg)^{\frac{\hat{E}}{E} + \alpha}(\varphi_k\ast H^\epsilon) (t) + \hat{C}'\bigg(\frac{1}{t-1}\bigg)^{1-\alpha}
\end{equation*}
for some constant $\hat{C}'>0$ independent of $k,\epsilon$.
Setting
\begin{equation*}
\gamma_1(t) := c_3 \bigg(\frac{1}{t+1}\bigg)^{\frac{\hat{E}}{E}+\alpha},\qquad \gamma_2(t) := \hat{C}'\bigg(\frac{1}{t-1}\bigg)^{1- \alpha}
\end{equation*}
and following the argument as per \cite{meta} from Lemma 6 in \cite{Miclo}, there exists $t_3 \geq t_2+2$, $c_4,c_5,c_6>0$ independent of $k$ and $\epsilon$ such that for $t\geq t_3$,
\begin{equation*}
\frac{d}{dt}\bigg(\frac{\gamma_2}{\gamma_1}\bigg)(t) = \frac{(t+1)^{\frac{\hat{E}}{E}+\alpha}}{(t-1)^{1-\alpha}}\bigg( \frac{c_4}{t+1} - \frac{c_5}{t-1} \bigg) \geq -c_6t^{-1},
\end{equation*}
so that there exists $t_4\geq t_3$ independent of $k$ and $\epsilon$ such that for $t\geq t_4$,
\begin{align*}
\frac{d}{dt}\bigg(\varphi_k\ast H^\epsilon - \frac{2\gamma_2}{\gamma_1}\bigg) (t) \leq -\gamma_1(t)\bigg(\varphi_k\ast H^\epsilon(t) - \frac{2\gamma_2(t)}{\gamma_1(t)}\bigg) 
\end{align*}
and consequently
\begin{equation}\label{mollidineq}
\varphi_k\ast H^\epsilon(t) \leq \frac{2\gamma_2(t)}{\gamma_1(t)} + \varphi_k\ast H^\epsilon(t_4)e^{-\int_{t_4}^t\gamma_1(u)du}.
\end{equation}
Finally, from \eqref{div} (adding and subtracting $\beta(T_{u'}^{-1})e^{-1}$ to $H^\epsilon$), it holds that for $t \geq t_4 + 2$,
\begin{equation}\label{simig}
H^\epsilon(t) = \int_{t-2k}^t \varphi_k(t-k-s)ds H^\epsilon(t)\leq \int_{t-2k}^t\varphi_k(t-k-s)H^\epsilon(s)ds + \tilde{g}(2k)
\end{equation}
for some $\tilde{g}:\mathbb{R}\rightarrow\mathbb{R}$ satisfying $\tilde{g}(k')\rightarrow 0$ as $k'\rightarrow 0$, so that \eqref{mollidineq} yields
\begin{equation*}
H^\epsilon(t) \leq \frac{2\gamma_2(t-k)}{\gamma_1(t-k)} + \varphi_k\ast H^\epsilon(t_4)e^{-\int_{t_4}^{t-k}\gamma_1(u)du} + \tilde{g}(2k),
\end{equation*}
where $\varphi_k\ast H^\epsilon(t_4)$ can be bounded independently of $k$ in a similar spirit to \eqref{simig}, and taking $k\rightarrow 0$ concludes the proof.
\end{proof}

\begin{remark}
The annealing schedule $T_t$ is chosen to satisfy the relationship \eqref{newbb} between $C_t^{-1}$ and $\abs*{T_t'}p\Big(T_t^{-1}\Big)$. 
\end{remark}

\subsection{Degenerate noise limit}\label{genapp}
After taking advantage of the square integrability Theorem 7.4.1 in \cite{MR3443169} for the case with a nondegenerate diffusion term in the proof of Proposition \ref{dissipation}, the $\epsilon\rightarrow 0$ limit is taken to obtain the same dissipation inequality in this section.
\begin{proof}\textit{(of Proposition \ref{dissipation0})} 
From \eqref{mid}, for any $0\leq s<t$ and $0 < \epsilon \leq \min(1,\epsilon')$, it holds that
\begin{equation*}
H^\epsilon(t) - H^\epsilon(s) \leq \int_s^t\abs*{T_u'}p\Big(T_u^{-1}\Big)\Big(H^\epsilon(u)+\hat{C}\Big)du,
\end{equation*}
where $p$ is a finite order polynomial with nonnegative coefficients and $\hat{C}>0$ is a constant both independent of $\epsilon$. Therefore, mollifying in time and taking $s\rightarrow t$ as in the end of the proof by Proposition \ref{dissipation}, it is straightforward that $H^\epsilon$ is uniformly bounded in\footnote{$t_H$ from Proposition \ref{dissipation}} $0\leq t \leq t_H$ and $0< \epsilon \leq \min(1,\epsilon')$. Moreover by Proposition \ref{dissipation}, the entropy $\int h_t^\epsilon \ln h_t^\epsilon d\mu_{T_t}$ is bounded uniformly in $t>t_H$ and $0< \epsilon \leq \min(1,\epsilon')$. Therefore for any $t\geq 0$ by the de la Vallée-Poussin criterion  (see for example \cite{MR1248651}), the subset $\{ h_t^\epsilon: 0< \epsilon \leq \min(1,\epsilon')\}\subset L^1(\mu_{T_t})$ is uniformly integrable and consequently the Dunford-Pettis theorem imposes the existence of a weak limit $g_t\in L^1(\mu_{T_t})$ for a (sub)sequence $(\epsilon_i)_{i\in\mathbb{N}}$ such that $\epsilon_i\rightarrow 0$,
\begin{equation*}
h_t^{\epsilon_i} \rightharpoonup g_t,\quad\textrm{in }L^1(\mu_{T_t})\quad\textrm{as } i\rightarrow \infty.
\end{equation*}
For any $S>0$, any compactly supported smooth test function $\phi:[0,S)\times\mathbb{R}^{2n+m}\rightarrow\mathbb{R}$, omitting the dependence on the space variable $\zeta = (x,y,z)$ wherever convenient and using It\^{o}'s rule,
\begin{align}
0 &= \lim_{i\rightarrow \infty}\int_{(0,S)\times\mathbb{R}^{2n+m}}(m_t^{\epsilon_i} - g_t\mu_{T_t})(-\partial_t - L_t)\phi dtd\zeta\nonumber\\
&= \lim_{i\rightarrow \infty}\int_{(0,S)\times\mathbb{R}^{2n+m}}\epsilon_i m_t^{\epsilon_i}(S_t^x + S_t^y)\phi dtd\zeta + \int_{(0,T)\times\mathbb{R}^{2n+m}} g_t\mu_{T_t}(\partial_t+L_t)\phi dtd\zeta + \int_{\mathbb{R}^{2n+m}}m_0\phi(0,\zeta) dtd\zeta,\nonumber\\
&= \int_{(0,S)\times\mathbb{R}^{2n+m}} g_t\mu_{T_t}(\partial_t+L_t)\phi dtd\zeta + \int_{\mathbb{R}^{2n+m}}m_0\phi(0,\zeta) dtd\zeta,\label{weeq}
\end{align}
so that in the distributional sense of \cite{MR3443169},
\begin{equation}\label{mkeeq00}
\begin{cases}
\partial_t (g_t\mu_{T_t}) = L_t^\top (g_t\mu_{T_t}) & \textrm{on }  \mathbb{R}^{2n+m}\quad\forall t>0,\\
(g_0\mu_{T_0}) = m_0. &
\end{cases}
\end{equation}
By Proposition \ref{refer2app}, the solution to \eqref{mkeeq00} is unique in the class of integrable solutions and since $m_t$ belongs in this same class, it holds that 
\begin{equation*}
g_t\mu_{T_t} = m_t
\end{equation*}
for all $t\in[0,S]$, which is that
\begin{equation*}
m_t^{\epsilon_i} \rightharpoonup m_t,\quad\textrm{in }L^1(\mu_{T_t})\quad\textrm{as } i\rightarrow \infty.
\end{equation*}
for all $0\leq t<S$. By Corollary 3.8 in \cite{MR2759829}, there exists a sequence $(\hat{m}_t^i)_{i\in\mathbb{N}}$ made up of convex combinations of $m_t^{\epsilon_i}$ that converge strongly to $m_t$ in $L^1$, hence a subsequence $(\hat{m}_t^{i_j})_{j\in\mathbb{N}}$ that convergences pointwise almost everywhere. By Fatou's lemma, convexity of $f(x) = x\ln x\geq e^{-1}$ for $x>0$ and Proposition \ref{dissipation}, for $t> t_H$, we get
\begin{align*}
\int h_t \ln h_t d\mu_{T_t} &= \int m_t\ln \bigg(\frac{m_t}{\mu_{T_t}}\bigg)\\
&\leq \liminf_{j\rightarrow \infty} \int \hat{m}_t^{i_j} \ln \bigg(\frac{\hat{m}_t^{i_j}}{\mu_{T_t}}\bigg)\\
&\leq B\bigg( \frac{1}{t}\bigg)^{1-\frac{\hat{E}}{E}-2\alpha}.
\end{align*}
\end{proof}

\section{Conclusions}\label{conclusion}
We explored the possibility of using the generalised Langevin equations in the context of simulated annealing.
Our main purpose was to establish convergence as for the underdamped Langevin equation and provide a proof of concept
in terms of performance improvement. Although the theoretical results hold for any scaling matrix $A$ given the stated restrictions, we saw
in our numerical results that its choice has great impact on the performance. In Section \ref{num}, $A_2,A_3$ or $A_4$ seemed to improve the exploration 
on the state space and/or the success proportion of the algorithm. There is plenty of work still required in terms 
of providing a more complete methodology for choosing $A$. This is left as future work and is also closely linked
with time discretisation issues as a poor choice for $A$ could lead to numerical integration stiffness. 
This motivates the development and study of improved numerical integration schemes, in particular, the extension of the conception and analysis on numerical schemes such as BAOAB \cite{Char} for the Langevin equation for \eqref{gl} and the extension of the work in \cite{glesamp} for non-identity matrices $\lambda$ and $A$. See \cite{2020arXiv201204245L} for work in this direction.

In addition, the system in \eqref{gl} is not the only way to add an auxiliary variable to the underdamped Langevin equations in \eqref{Sim2} 
whilst retaining the appropriate equilibrium distribution. Our choice was motivated by a clear connection to the generalised Langevin equation \eqref{gen} and
link with accelerated gradient descent, 
but it could be the case that a different third or higher order equations could be used with possibly improved performance. Along these lines, 
one could consider adding skew-symmetric terms as in \cite{skewterms}. 
As regards to theory, an interesting extension could involve
establishing how the results here can be extended to establish a comparison of optimisation and sampling 
in a nonconvex setting for an arbitrary number of dimensions similar to \cite{fast}.
We leave for future work finding optimal constants in the convergence results, investigating dependence on parameters
and how the limits of these parameters and constants relate to existing results for the Langevin equation in \eqref{Sim2}
in \cite{meta,Vaes}. 
Finally, one could also aim to extend
large deviation results in \cite{larg2,larg1,larg3} for the overdamped Langevin dynamics to the underdamped and generalised case.

\section*{Acknowledgements}

The authors would like to thank Tony Lelievre, Gabriel Stoltz and Urbain Vaes for their helpful remarks. M.C. was funded under a EPSRC studentship. G.A.P. was partially supported by
the EPSRC through grants EP/P031587/1, EP/L024926/1, and EP/L020564/1. N.K. and G.A.P. were funded in part by JPMorgan Chase $\&$ Co under a J.P. Morgan A.I. Research Awards 2019. Any views or
opinions expressed herein are solely those of the authors listed, and may differ from
the views and opinions expressed by JPMorgan Chase $\&$ Co. or its affiliates. This
material is not a product of the Research Department of J.P. Morgan Securities
LLC. This material does not constitute a solicitation or offer in any jurisdiction.
\bibliography{document}


\begin{appendices}

\section{Additional Results}\label{appendixa}
We present the analog of Proposition \ref{dissipation} for the $T_t=T>0$ sampling case and a result about the choice of the annealing schedule.
{
\begin{proof} \textit{(of Propostion \ref{constT0})}
After Pinsker's inequality \eqref{Pins} and consideration of the definition \eqref{entdef} of $H$, what remains is the partial time derivative part of the proof of Proposition \ref{dissipation}. The proof concludes by the same calculations as in Proposition \ref{dissipation}, keeping in mind $T_t'=0$, until \eqref{ineqint2} followed by the Gr{\"o}nwall argument. Note that \eqref{logSobConst} and \eqref{coeff2} are not required and a log-Sobolev constant (in $t$ also) works, in which case \eqref{logSob} and hence the current argument follow without requiring Assumption \ref{assumption2}. The limiting $\epsilon$ argument as in Proposition \ref{dissipation0} is the same.
\end{proof}
}
\begin{prop}\label{optimal}
Under Assumption \ref{assumption1}, \ref{assumption3} and \ref{assumption4}, the schedule $T_t =\frac{E}{\ln (e+t)}$, $E>\hat{E}$ is optimal in the sense that for any differentiable $f :\mathbb{R}_+\rightarrow \mathbb{R}_+$, if
\begin{equation}\label{formTt}
T_t = \frac{1}{f(t)}\bigg(\frac{\hat{E}}{\ln (e+t)}\bigg),
\end{equation}
$C_t$ is the log-Sobolev factor \eqref{logSobConst} and $p$ is the finite order polynomial with nonnegative coefficients from the proof of Proposition \ref{dissipation}, then the relation
\begin{equation}\label{dineq}
2C_t^{-1}\gg\abs*{T_t'}p\Big(T_t^{-1}\Big)
\end{equation}
holds for large times only if $\limsup_{t\rightarrow\infty} f(t)\leq1$.
\end{prop}

\begin{proof}
Suppose there exists a constant $\delta>0$ and times $(t_i)_{i\in\mathbb{N}}$ such that $0<t_i\rightarrow\infty$ and
\begin{equation*}
f(t_i)\geq 1+\delta \quad\forall i.
\end{equation*}
From \eqref{logSobConst},
\begin{equation*}
C_t^{-1} \sim \mathcal{O}(e^{-\hat{E}T_t^{-1}}T_t^{-1}),
\end{equation*}
which after substituting in \eqref{formTt} gives
\begin{equation}\label{edot}
e^{-\hat{E}T_t^{-1}}T_t^{-1}=(e+t)^{-f(t)}\frac{f(t)\ln(e+t)}{\hat{E}} \sim \mathcal{O}(t^{-f(t)}f(t)\ln t).
\end{equation}
Compare this to
\begin{equation}\label{edot2}
\abs*{T_t'}p\Big(T_t^{-1}\Big) \propto \frac{p(f(t)\ln (e+t))}{(f(t)\ln (e+t))^2}\bigg(\frac{f(t)}{e+t}+\abs*{f'(t)}\ln (e+t)\bigg),
\end{equation}
which has order at least $(tf(t))^{-1}(\ln t)^{-2}$. For $t=t_i$, $i$ large enough, $f(t)\geq 1+\delta$ and so
\begin{equation}\label{contra}
 t^{-f(t)}f(t)\ln t \ll (tf(t))^{-1}(\ln t)^{-2},
\end{equation}
which violates \eqref{dineq}.
\end{proof}

\begin{remark}
One can strengthen the proposition by making precise the form of $p$ from Proposition \ref{dissipation}, which will determine how slowly $f(t)$ is allowed to converge to $1$; in fact $p$ should be at least sixth order. This seems inconsequential with respect to optimality and so is omitted.
\end{remark}

\end{appendices}

\end{document}